\documentclass[12pt]{amsart}
\usepackage{booktabs}
\usepackage{mathrsfs}
\usepackage{graphicx}
\usepackage{array}
\usepackage{amssymb}
\usepackage{longtable}
\usepackage{multirow}
\usepackage[all]{xy}
\numberwithin{equation}{section}

\newtheorem*{thmA}{\bf Theorem A}
\newtheorem*{corB}{\bf Corollary B}
\newtheorem*{thmC}{\bf Theorem C}
\newtheorem*{thmD}{\bf Theorem D}

\newtheorem{thm}{Theorem}[section]

\newtheorem{lem}[thm]{Lemma}
\newtheorem{prop}[thm]{Proposition}
\theoremstyle{definition}
\newtheorem{defn}[thm]{Definition}
\theoremstyle{remark}
\newtheorem{rem}[thm]{Remark}

\theoremstyle{example}
\newtheorem{example}{definition}
\newtheorem{eg}[example]{Example}

\newcommand{\set}[1]{\left\{#1\right\}}
\newcommand{\ip}[1]{\langle #1 \rangle}
\newcommand{\Zeit}{\mathbb Z}
\newcommand{\Real}{\mathbb R}
\newcommand{\Cpx}{\mathbb{C}}
\newcommand{\Qua}{\mathbb{H}}
\newcommand{\Cay}{\mathbb O}
\newcommand{\Rp}{\mathbb{R}\mathrm{P}}
\newcommand{\Cp}{\mathbb{C}\mathrm{P}}
\newcommand{\Hp}{\mathbb{H}\mathrm{P}}
\newcommand{\Cayp}{\mathbb{O}\mathrm{P}}

\newcommand{\half}{\frac{1}{2}}

\newcommand{\sph}{\mathbb{S}}
\newcommand{\Disk}{\mathbb{D}}
\newcommand{\To}{\longrightarrow}

\newcommand{\qi}{\imath}
\newcommand{\qj}{\jmath}
\newcommand{\qk}{\kappa}

\newcommand{\imbed}{\hookrightarrow}

\newcommand{\gS}{\textsf{S}}
\newcommand{\T}{\textsf{T}}
\newcommand{\SO}{\textsf{SO}}
\newcommand{\gO}{\textsf{O}}
\newcommand{\Spin}{\textsf{Spin}}
\newcommand{\SU}{\textsf{SU}}
\newcommand{\Sp}{\textsf{Sp}}
\newcommand{\G}{\textsf{G}}
\newcommand{\Gt}{\mathfrak{G}_2}
\newcommand{\K}{\textsf{K}}
\newcommand{\U}{\textsf{U}}
\newcommand{\gH}{\textsf{H}}
\newcommand{\gL}{\textsf{L}}
\newcommand{\fw}{\pi}
\newcommand{\E}{\textsf{E}}
\newcommand{\F}{\textsf{F}}

\newcommand{\liesu}{\mathfrak{su}}
\newcommand{\lieu}{\mathfrak{u}}
\newcommand{\lieso}{\mathfrak{so}}
\newcommand{\liesp}{\mathfrak{sp}}
\newcommand{\liespin}{\mathfrak{spin}}
\newcommand{\liee}{\mathfrak{e}}
\newcommand{\lief}{\mathfrak{f}}
\newcommand{\lieg}{\mathfrak{g}}
\newcommand{\liegt}{\mathscr{G}_2}
\newcommand{\lieh}{\mathfrak{h}}
\newcommand{\liep}{\mathfrak{p}}
\newcommand{\liem}{\mathfrak{m}}

\newcommand{\liek}{\mathfrak{k}}
\newcommand{\liel}{\mathfrak{l}}

\newcommand{\Lie}{\mathrm{Lie}}

\newcommand{\cohomd}{cohomogeneity one manifold }

\newcommand{\Kp}{\textsf{K}^{+}}
\newcommand{\Km}{\textsf{K}^{-}}
\newcommand{\Kpm}{\textsf{K}^{\pm}}
\newcommand{\liekm}{\mathfrak{k}^{-}}
\newcommand{\liekp}{\mathfrak{k}^{+}}
\newcommand{\liekpm}{\mathfrak{k}^{\pm}}
\newcommand{\Bm}{B_{-}}
\newcommand{\Bp}{B_{+}}
\newcommand{\Bpm}{B_{\pm}}
\newcommand{\lpm}{l_{\pm}}
\newcommand{\lp}{l_{+}}
\newcommand{\lm}{l_{-}}
\newcommand{\ptpm}{p_{\pm}}
\newcommand{\ptp}{p_{+}}
\newcommand{\ptm}{p_{-}}

\newcommand{\vp}{\varphi}

\newcommand{\cohom}{\mathrm{cohom}}
\newcommand{\Ad}{\mathrm{Ad}}
\newcommand{\ad}{\mathrm{ad}}
\newcommand{\df}{\mathrm{d}}
\newcommand{\Proj}{\mathrm{Proj}}
\newcommand{\diag}{\mathrm{diag}}
\newcommand{\Diag}{\Delta}
\newcommand{\id}{\mathrm{id}}
\newcommand{\Id}{\mathrm{Id}}
\newcommand{\Sym}{\mathrm{S}}

\newcommand{\Ann}{\mathrm{Ann}}
\newcommand{\Aut}{\mathrm{Aut}}

\newcommand{\I}{\mathrm{I}}
\newcommand{\II}{\mathrm{II}}
\newcommand{\III}{\mathrm{III}}
\newcommand{\V}{\mathrm{V}}
\newcommand{\wt}{\widetilde}
\newcommand{\Gr}{\mathrm{Gr}}

\newcommand{\ot}{\otimes}

\newcommand{\tU}{\widetilde{\textsf{U}}}
\newcommand{\tKm}{\widetilde{\K^-}}
\newcommand{\tKp}{\widetilde{\K^+}}

\newcommand{\Schb}{Schwachh\"{o}fer }

\begin{document}

\title[Cohomogeneity One Manifolds]{Cohomogeneity one manifolds with a small family of invariant metrics}%
\author{Chenxu He}%
\address{Department of Mathematics, University of Pennsylvania}%
\email{hech@math.upenn.edu}%


\maketitle

\tableofcontents

%

\newpage

\begin{abstract}
In this paper, we classify compact simply connected cohomogeneity one manifolds up to equivariant
diffeomorphism whose isotropy representation by the connected component of the principal isotropy
subgroup has three or less irreducible summands. The manifold is either a bundle over a homogeneous
space or an irreducible symmetric space. As a corollary such manifolds admit an invariant metric
with non-negative sectional curvature.
\end{abstract}

\section{Introduction}

Manifolds with positive or non-negative sectional curvature have been of interest since the
beginning of the global Riemannian geometry. Finding new examples of such manifolds is a particular
difficult problem. There are few examples of positive curvature, mainly quotients of compact Lie
groups. Recently two new examples were discovered using different methods. K. Grove, L. Verdiani
and W. Ziller succeeded in putting positively curved metrics on a so called \emph{cohomogeneity
one} manifold\cite{GVZpos}, i.e., the manifold admits an isometric group action and the orbit space
is one dimension. Their example belongs to a family of infinitely many cohomogeneity one manifolds
for which there is no known obstruction to positively curved metric, see \cite{GWZ}. O. Dearicott
also proposed another construction for the same manifold, see \cite{Derricott}. Another new example
with positive curvature was discovered by P. Petersen and F. Wilhelm\cite{PWexotic} on an exotic
$7$-sphere.

The class of non-negatively curved manifolds is much larger, although methods to construct them are
very limited. Well known examples are products and quotients of such manifolds, e.g., quotients of
compact Lie groups. In \cite{Cheeger} J. Cheeger used a gluing method to put non-negatively curved
metrics on the connected sum of any two compact rank one symmetric spaces. Cheeger's gluing method
was greatly generalized by K. Grove and W. Ziller in \cite{GZMilnor} to cohomogeneity one manifold,
where they proved that if the singular orbits are codimension two, the manifold admits an invariant
metric with non-negative curvature. However \emph{not} every cohomogeneity one manifold carries an
invariant metric with non-negative sectional curvature. The first examples were discovered by K.
Grove, L. Verdiani, B. Wilking and W. Ziller in \cite{GVWZ} which were generalized to a larger
class by the author, see \cite{He}. It is thus natural to ask how large the class of cohomogeneity
one manifolds with an invariant metric of non-negative curvature is, see \cite{Zsurvey1}.

One possible approach is to study cohomogeneity one manifolds with further geometric or topological
restrictions. C. Hoelscher classified simply-connected examples in dimension 7 or less in
\cite{HoelscherClass}. Most of them carry non-negatively curved metrics except for some examples in
dimension 7, for which the existence of such metrics is still open. L. \Schb and K. Tapp studied
cohomogeneity one manifolds which have a totally geodesic principal orbit. They classified such
manifolds under some further conditions, see \cite{SchTapp}. The examples in their classification
were already known to have non-negatively curved metrics.

Cohomogeneity one manifolds also arise as interesting examples in other areas. For instance, many
new examples of Einstein and Einstein-Sasaki metrics(see, for example, \cite{BoehmEinstein},
\cite{Conti} and \cite{GibbonsHartnollYasui}), and metrics with special holonomy(see
\cite{CleytonSwann}, \cite{CveticGibbonsLuPope} and \cite{ReidegeldG2}-\cite{ReidegeldAW}) were
found among cohomogeneity one manifolds. In the study of the Ricci flow, Ricci solitons play an
important role in order to understand the singularities of the flow. Many interesting examples of
Ricci solitons also admit cohomogeneity one actions, see \cite{DancerWang}.

As far as the Riemannian metrics are concerned, one may first consider examples which are
geometrically simple, i.e., the family of invariant metrics is small. Let $M$ be a cohomogeneity
one manifold, i.e., there exists a compact Lie group $\G$ acting on $M$ by isometries and the
cohomogeneity of the action, defined as $\cohom(M,\G) = \dim(M/\G)$, is equal to $1$. If the
manifold is compact and simply-connected, then the orbit space is a closed interval $I$. In this
case, there are precisely two singular orbits $\Bpm$ with isotropy subgroups $\Kpm$ corresponding
to the endpoints of $I$, and principal orbits corresponding to the interior points with isotropy
subgroup $\gH$. Let $\lpm$ denote the codimensions of $\Bpm$. The manifold can be written as a
union of two disk bundles as $M = \G \times_{\Km}\Disk^{\lm} \cup_{\G/\gH}\G
\times_{\Kp}\Disk^{\lp}$, and thus it can be identified with the \emph{group diagram} $\gH \subset
\set{\Kpm} \subset \G$. Suppose $c(t)$ is a minimal geodesic between the singular orbits and $t$ is
the arc length. Since the $\G$-action is transitive on orbits and the metric is left invariant on
them, we only have to consider the metric along $c(t)$ and thus $g= \df t^2 + g_t$, where $g_t$ is
the metric at $c(t)$. Let $M_t = \G.c(t)$ be a principal orbit. Since $\gH_c$, the identity
component of $\gH$, fixes $c(t)$, it induces the isotropy representation on the tangent space
$T_{c(t)}M_t$. In general this representation is not irreducible and let $s$ denote the number of
irreducible summands. From Schur's lemma, if the value of $s$ is big, the family of invariant
metrics is large.

In this paper we classify compact simply-connected cohomogeneity one manifolds with $s \leq 3$. We
say a $\G$-manifold $M$ is \emph{non-primitive} if there is a $\G$-equivariant map $M \To \G/\gL$
for some proper subgroup $\gL \subset \G$. Otherwise the $G$-action is called \emph{primitive}. Our
main result is
\begin{thmA}
If a compact simply-connected Riemannian manifold $M$ admits a primitive cohomogeneity one action
with $s \leq 3$, then it is equivariantly diffeomorphic to a symmetric space with an isometric
action.
\end{thmA}
In fact only spheres, complex and quaternionic projective spaces, the Cayley plane, and
Grassmannians $\SO(m+n)/(\SO(m)\times \SO(n))(m,n\geq 2)$ appear in the classification. This easily
implies
\begin{corB}
Every compact simply-connected cohomogeneity one manifold with $s\leq 3$ admits an invariant metric
with non-negative sectional curvature.
\end{corB}
In contrast to the above corollary, the Kervaire spheres carry cohomogeneity one actions with $s=4$
and they do not admit any non-negatively curved invariant metrics, see \cite{GVWZ}.

If the cohomogeneity one manifold $M$ is non-primitive, then it has a group diagram $\gH \subset
\set{\Km, \Kp} \subset \G$ such that there is a proper subgroup $\gL \subset \G$ that contains both
$\Kpm$. It follows that the $\G$ action has no fixed points. Furthermore, $M$ is a fiber bundle
over the homogeneous space $\G/\gL$ with fiber $N$ and the $\gL$ action on $N$ is cohomogeneity one
with diagram $\gH \subset \set{\Km, \Kp} \subset \gL$. One particular class of non-primitive
manifolds is the so called \emph{double} for which $\Km = \Kp$. For a double we can let $\gL =
\Kpm$ and then $N$ is a sphere with a linear cohomogeneity one action that has two fixed points. In
general for a non-primitive action, we have
\begin{thmC}
If $M$ is a compact simply-connected cohomogeneity one manifold that admits a non-primitive action
by a compact Lie group $\G$ with $s\leq 3$, then
\begin{enumerate}
\item either $M$ is a double;
\item or $M$ is a fiber bundle over a strongly isotropy irreducible space $\G/\gL$ with fiber $N$.
\end{enumerate}
In case (2), if the action of $\gL$ on $N$ has no fixed points, then $N$ is a sphere or a three
dimensional lens space. Otherwise $N$ is a projective space or the Cayley plane.
\end{thmC}

In case (2), the space $\G/\gL$ is \emph{strongly isotropy irreducible} which means that the
isotropy representation by $\gL_c$ is irreducible. Cohomogeneity one manifolds with a fixed point
were classified in \cite{HoelscherClass}(see also in \cite{GroveSearleJDG}), where he showed that
the manifolds are equivariantly diffeomorphic to projective spaces and the Cayley plane. In Theorem
\ref{thmclasss2primitive} and \ref{thmclasss2nonprimitive} we classify cohomogeneity one manifolds
such that the action has no fixed point and $s=2$. One can prove Theorem C directly by applying
these two classifications. In this paper, we take a different approach since we also want the
classification of such cohomogeneity one diagrams, i.e., a classification up to equivariant
diffeomorphism. We first classify all non-primitive diagrams and then Theorem C easily follows.

Notice that in Theorem A and C, we do not assume that $\gH$ is connected. Notice also that $s$ is
the number of irreducible summands of the isotropy action by $\gH_c$. It is possible that there are
other cohomogeneity one manifolds where the whole group $\gH$ acts with only $3$ irreducible
summands.

A manifold $M$ may have different cohomogeneity one actions, for example, there are many
cohomogeneity one actions on spheres. One special case is that a proper normal subgroup $\G_1
\subset \G$ acts on the manifold with the same orbits. In this case, we call the $\G$-action
\emph{reducible}. Otherwise the action is called \emph{non-reducible}. Let $\gH$ be the principal
isotropy subgroup of the $\G$-action. Note that the $\G$ action being reducible is equivalent to
the fact that $\gH$ projects onto a simple factor of $\G$. Let $\gH_1 = \gH \cap \G_1$ and then it
is the principal isotropy subgroup of the $\G_1$-action. Since $\gH_1$ is a subgroup of $\gH$, the
isotropy action of $\gH_1$ on the tangent space of the principal orbit may have more irreducible
summands, i.e., a bigger value of $s$. For this reason, we consider the classification of reducible
actions as well.

\begin{thmD}
If a compact simply-connected manifold $M$ admits a cohomogeneity one action by $\G$ with $s\leq 3$
and the action is reducible, then one of the following holds:
\begin{enumerate}
\item a normal subgroup of $\G$ acts on $M$ non-reducibly and has the same value of $s$;
\item the action is non-primitive;
\item the action is primitive and it is a linear action on a sphere.
\end{enumerate}
\end{thmD}
In Theorem \ref{thmclassificationreducible}, we obtained the classification of cohomogeneity one
manifolds with $s=3$ in case (2) of Theorem D by using the classification results in Theorem A and
C with $s=1, 2$. The actions on spheres in case (3) of Theorem D are the so called \emph{sum
actions}, see definition in Section 2.
\smallskip

One can also use our classification to study some cohomogeneity one manifolds with $s = 4$ or
higher. For example, the Kervaire sphere which does not admit an invariant metric with non-negative
curvature has the group diagram
\begin{equation*}
\Zeit_2 \times \SO(n-2) \subset \set{\SO(2) \times \SO(n-2), \gO(n-1)} \subset \SO(2) \times
\SO(n),
\end{equation*}
where $n \equiv 1 \mod 4$. The diagram can be constructed from the following one which has $s=3$
\begin{equation*}
\Zeit_2 \times \SO(n-2) \subset \set{\SO(2) \times \SO(n-2), \gO(n-1)} \subset \SO(n)
\end{equation*}
by adding an $\SO(2)$ factor to $\SO(n)$ and embedding the $\SO(2)$ factor in $\SO(2)\times
\SO(n-2)$ diagonally into $\SO(2) \times \SO(n)$. Similar constructions can be applied to certain
examples in our classification and they give an interesting class of cohomogeneity one diagrams
with $s=4$. They will be discussed in a forthcoming paper.

\smallskip

The paper is organized as follows. In Section 2, we recall some basic facts about cohomogeneity one
manifolds which will be used throughout the paper. In Section 3 we consider the classification when
$s=1,2$, see Theorem \ref{thmclasss2primitive} and \ref{thmclasss2nonprimitive}. In Section 4 we
consider the classification when $s=3$. The classification for reducible actions is stated in
Theorem \ref{thmclassificationreducible} and for the non-primitive ones in Theorem
\ref{thms3nonprimitive}. The classification of primitive cohomogeneity one manifolds with $\G$
simple and $s=3$ is carried in Section 5, see Theorem \ref{thms3primitiveGsimple}, and when $\G$ is
non-simple in Section 6, see Theorem \ref{thms3Gnonsimple}. In Appendix \ref{apptwosummands}, we
correct the classification of compact simply-connected homogeneous spaces $\G/\gH$ such that the
isotropy action has two summands and $\G$ is a simple Lie group. This will be used in Section 5. In
Appendix \ref{appTables}, we collect the tables which contain our classification results. We will
see that a cohomogeneity manifold with $s\leq 3$ either has a fixed point, or is a double, a
sphere, or is contained in Tables \ref{Tableprojectivespaces}, \ref{Tables3reducibledouble},
\ref{Tables3reduciblenondouble} and \ref{Tables3nonprimitive}.

\medskip

\emph{Acknowledgements:} Part of the paper is in the author's Ph.D. thesis at University of
Pennsylvania. The author wants to thank his advisor, Professor Wolfgang Ziller, for his constant
supports and great patience, and to Professor Lorenz \Schb and Professor Megan Kerr for valuable
discussions.

\section{Preliminaries}

In this section, we recall some basic and well-known facts about cohomogeneity one manifolds. For
more detail, we refer to, for example, \cite{Alekseevsy} and \cite{GWZ}.

As mentioned already, there are precisely two non-principal orbits $\Bpm$ in a simply connected
cohomogeneity one manifold. Suppose $M$ is endowed with an invariant metric $g$ and the distance
between the two non-principal orbits is $L$. The two ending points of the minimal geodesic $c(t)$
are specified as $c(0) = \ptm \in \Bm$ and $c(L) = \ptp \in \Bp$. Thus the isotropy subgroups at
$\ptpm$ are $\Kpm$ and the principal isotropy subgroup at any point $c(t)$, $t\in (0,L)$, is $\gH$.
We then draw the following group diagram for the manifold $M$:
\begin{equation*}
\xymatrix{
                & \G             \\
 \Km \ar@{-}[ur]^{}& &  \Kp \ar@{-}[ul]                 \\
                & \gH \ar@{-}[ul]^{\sph^{\lm -1} = \Km/\gH} \ar@{-}[ur]_{\sph^{\lp-1} = \Kp/\gH  }}
\end{equation*}
The group diagram $\gH \subset \set{\Km, \Kp}\subset \G$ is not uniquely determined by the manifold
since one can switch $\Km$ with $\Kp$, change $g$ to another invariant metric and choose another
minimal geodesic $c(t)$.

We consider the family of invariant metrics on $M$. Since the union of principal orbits are open
and dense in $M$, we only have to consider the metric restricted on it. For every $t \in (0,L)$,
the principal orbit $M_t$ is diffeomorphic to the homogeneous space $\G/\gH$. Let $\lieg$ and
$\lieh$ are the Lie algebras of $\G$ and $\gH$ respectively. Fix a bi-invariant inner product $Q$
on $\lieg$ and let $\liep$ be the orthogonal compliment of $\lieh \subset \lieg$. The space $\liep$
is identified with the tangent space $T_{c(t)}M_t$ via Killing vector fields. Since $\gH$ fixed the
point $c(t)$, it induces the so called isotropy representation on $\liep$:
\begin{equation*}
\Ad(h) : X \To h.X.h^{-1}, \mbox{ for any } h \in \gH \mbox{ and } X \in \liep.
\end{equation*}
We denote its differential by $\ad$ and it defines the isotropy representation of the Lie algebra
$\lieh$:
\begin{equation*}
\ad_{Z} : X \To [Z,X], \mbox{ for any } Z \in \lieh \mbox{ and } X \in \liep.
\end{equation*}
For a cohomogeneity one manifold $M$, $\gH$ may not be connected even if we assume that $\G$ is
connected and $M$ is simply-connected. Let $\Ad_{\gH_c}$ be the restriction the isotropy
representation to the identity component $\gH_c$. In most cases, it is reducible and the number of
the irreducible summands is denoted by $s$. It is equivalent to say that the isotropy
representation $\ad_{\lieh}$ has $s$ summands. The simplest case is $s=1$, i.e., $\Ad_{\gH_c}$ is
irreducible. From Schur's lemma $g_t$ is a scalar multiplication of the identity map on
$T_{c(t)}M_t$ and thus the metric $g$ depends on one function. In the case where $s=3$ the metric
depends on three functions if the three summands are non-equivalent.

For a cohomogeneity one manifold, a convenient way to describe the manifold and the action is to
use the group diagram. However, a cohomogeneity one manifold may have different diagrams.

\begin{defn} Two group diagrams are called \emph{equivalent} if they
determine the same cohomogeneity one manifold up to equivariant diffeomorphism.
\end{defn}

The following lemma characterizes which two group diagrams are equivalent, see \cite{GWZ}.
\begin{lem}\label{LemGroupEquivalent}
Two group diagrams $\gH \subset \set{\Km, \Kp} \subset \G $ and $\widetilde{\gH} \subset
\set{\widetilde{\K}^-, \widetilde{\K}^+} \subset \G$ are equivalent if and only if after possibly
switching the roles of $\Km$ and $\Kp$, the following holds: There exist elements $b \in \G$ and $a
\in N(\gH)_c$, where $N(\gH)_c$ is the identity component of the normalizer of $\gH$, with
$\widetilde{\K}^- = b \Km b^{-1}$, $\widetilde{\gH} = b\gH b^{-1}$, and $\widetilde{\K}^+ = ab \Kp
b^{-1} a^{-1}$.
\end{lem}

\begin{rem}\label{remequivalentdiag}
If $c(t)$ is the minimal normal geodesic between the two singular orbits, then $b. c(t)$ is another
minimal geodesic between $\Bpm$ and the associated group diagram is obtained by conjugating all
isotropy groups by the element $b$. We can assume that $b \in N(\gH) \cap N(\Km)$ in order to fix
$\gH$ and $\Km$. Conjugation by an element $a$ as in the above lemma usually corresponds to
changing the invariant metric on the manifold.
\end{rem}

Let us describe a method using automorphisms of $\G$ to obtain new group diagrams from a given one
$M: \gH\subset \set{\Km, \Kp}\subset \G$. Take two automorphisms $\tau_\pm$ of $\G$ and apply them
to the triples $\gH \subset \Kpm \subset \G$. If $\tau_{-}(\gH)$ and $\tau_{+}(\gH)$ are equal,
then we have the diagram $\tau_{-}(\gH) \subset \set{\tau_{-}(\Km), \tau_{+}(\Kp)} \subset \G$ and
the manifold is $\G$ equivariantly diffeomorphic to $M_{\tau}$ defined by the diagram $\gH \subset
\set{\Km, \tau(\Kp)} \subset \G$ where $\tau = \tau_{-}^{-1}\cdot \tau_{+}$ that leaves $\gH$
invariant. If $\Km$ or $\Kp$ is also invariant by $\tau$, then the manifolds $M$ and $M_{\tau}$ are
equivariantly diffeomorphic. So we assume that only $\gH$ is invariant by the automorphism $\tau$,
i.e., $\tau$ is in the double coset space $\Aut(\G, \Km)\backslash \Aut(\G,\gH)/\Aut(\G,\Kp)$ where
$\Aut(\G,\gL)$ is the group of automorphisms of $\G$ leaving the subgroup $\gL$ invariant.
Furthermore, if two automorphisms $\tau_1$ and $\tau_2$ can be connected by a continuous path in
$\Aut(\G, \gH)$, then $M_{\tau_1}$ and $M_{\tau_2}$ are $\G$ equivariantly diffeomorphic.
\begin{defn}
A group diagrams $\Gamma_1: \gH \subset \set{\tKm, \tKp}\subset \G$ is called \emph{a variation} of
the diagram $\Gamma_2: \gH\subset \set{\Km, \Kp}\subset \G$ if $\tKm =\tau(\Km)$, $\tKp = \Kp$ for
some $\tau \in \Aut(\G,\gH)$ after possible switching of $\Km$ and $\Kp$.
\end{defn}

Next we consider several special classes of cohomogeneity one actions which have been mentioned in
the Introduction and we recall basic properties of them.

\begin{defn}
A cohomogeneity one manifold $(M,\G)$ is called \emph{a double} if it admits a group diagram such
that $\Km = \Kp$.
\end{defn}
One can put a non-negatively curved invariant metric on the disk bundle $\G \times_\K \Disk^l$
making the boundary totally geodesic. If $M$ is a double, then we can glue the two identical disk
bundles along the totally geodesic boundary so that $M$ has an invariant metric with non-negative
sectional curvatures.

\begin{defn}
A cohomogeneity one manifold $(M, \G)$ is \emph{non-primitive} if it admit a group diagram $\gH
\subset \set{\Kpm} \subset \G$ and there is a proper connected subgroup $\textsf{L}\subset \G$ such
that $\Kpm \subset \textsf{L}$. A cohomogeneity one manifold is called \emph{primitive} if it is
not non-primitive.
\end{defn}

If $M$ has a non-primitive group diagram $\gH \subset \set{\Kpm} \subset \G$ and $\Kpm \subset \gL
\subset \G$, then we have the following fibration:
\begin{equation*}
N \To M \To \G/\gL,
\end{equation*}
and the fiber $N$ carries a cohomogeneity one action by $\gL$ with the diagram $\gH \subset
\set{\Kpm} \subset \gL$. Thus $M$ is $\G$ equivariantly diffeomorphic to $\G\times_\gL N$. So $M$
has an invariant metric with non-negative sectional curvatures if $N$ admits such a metric.

\begin{defn}
A cohomogeneity one action of $\G$ on $M$ is \emph{reducible} if there is a proper normal subgroup
of $\G$ that still acts by cohomogeneity one with the same orbit. Otherwise the action is called
\emph{nonreducible}.
\end{defn}

In terms of group diagram, we have the following characterization of reducible actions, see
\cite{HoelscherClass}, Section 1.3.

\begin{prop}\label{propreduciblediagram}
Let $M$ be the \cohomd given by the group diagram $\gH \subset \set{\Km, \Kp} \subset \G$ and
suppose $\G= \G_1 \times \G_2$ with $\Proj_2(\gH)=\G_2$. Then the sub-action of $\G_1 \times
\set{1}$ on $M$ is also by cohomogeneity one, with the same orbits, and with isotropy groups
$\Kpm_1 = \Kpm \cap (\G_1 \times \set{1})$ and $\gH_1 = \gH\cap (\G_1\times \set{1})$.
\end{prop}

On the other hand, suppose $\G_1$ acts on $M$ via cohomogeneity one with the diagram $\gH_1 \subset
\set{\Kpm_1} \subset \G_1$, then one can extend this action to a possibly larger group by the so
called \emph{normal extension}, see \cite{HoelscherClass}. Let $\gL$ be a compact connected
subgroup of $N(\gH_1) \cap N(\Km_1) \cap N(\Kp_1)$ and define $\G_2 = \gL/(\gL \cap \gH_1)$. Then
one can define an isometric action of $\G_1 \times \G_2$ on $M$ orbitwise by $(\hat{g}_1, [l]).
g_1(\G_1)_x = \hat{g}_1 g_1 l^{-1}(\G_1)_x$ for $(\G_1)_x = \gH_1$ or $\Kpm_1$. This action is also
cohomogeneity one and has the following diagram
\begin{equation*}
(\gH_1 \times 1)\cdot \Delta \gL \subset \set{(\Kpm_1 \times 1)\cdot \Delta \gL} \subset \G_1
\times \G_2,
\end{equation*}
where $\Delta \gL = \set{(l,[l])| l \in \gL}.$ If the diagram is non-primitive, then its normal
extension is also non-primitive. Note that the reducible action by $\G_1 \times \G_2$ in
Proposition \ref{propreduciblediagram} occurs as a normal extension of the reduced one by $\G_1$,
see Proposition 1.18 in \cite{HoelscherClass}.

\begin{rem}
Suppose the diagram $\gH \subset \set{\Kpm}\subset \G_1 \times \G_2$ is reducible and the isotropy
representation of $\gH_c$ has $3$ irreducible summands. If we consider the non-reducible action by
$\G_1$, then in general the isotropy representation by the connected component of $\gH_1$ may have
more irreducible summands since $\gH_1$ is a subgroup of $\gH$.
\end{rem}

There is a particular class of cohomogeneity one actions on spheres, called \emph{sum action},
which can be builded from transitive actions on spheres with lower dimensions, see
\cite{HoelscherClass} and \cite{GWZ}. Let $\G_i$ act transitively, linearly and isometrically on
the sphere $\sph^{l_i}$ with isotropy subgroup $\gH_i$, $i=1,2$. Then the action of $\G_1 \times
\G_2$ on $\sph^{l_1 + l_2 + 1} \subset \Real^{l_1+l_2 +2}$ is cohomogeneity one and the isotropy
subgroups are $\G_1 \times \gH_2$, $\gH_1 \times \G_2$ and $\gH_1 \times \gH_2$, i.e., the
cohomogeneity diagram is
\begin{equation}
\gH_1 \times \gH_2 \subset \set{\G_1 \times \gH_2, \, \gH_1 \times \G_2} \subset \G_1 \times \G_2.
\label{diagramsumaction}
\end{equation}
Suppose the isotropy representation of $\G_i/\gH_i$ has $s_i$ irreducible summands($i=1,2$), then
the sum action of the diagram (\ref{diagramsumaction}) has $s= s_1 + s_2$.

\smallskip

Not every group diagram defines a simply-connected cohomogeneity one manifold. The necessary
conditions are given in \cite{GWZ}.
\begin{lem}\label{lemsimplyconnected}
Suppose a connected Lie group $\G$ acts on a simply-connected manifold $M$ by cohomogeneity one and
the diagram is $\gH \subset \set{\Kpm} \subset \G$. Then we have
\begin{enumerate}
\item There are no exceptional orbits, i.e., $\lpm \geq 2$.
\item If both $\lpm \geq 3$, then $\Kpm$ and $\gH$ are all connected.
\item If one of $\lpm$, say $\lm = 2$, and $\lp \geq 3$, then $\Km = \gH\cdot \gS^1 = \gH_c \cdot
\gS^1$, $\gH = \gH_c \cdot \Zeit_k$ and $\Kp = \Kp_c \cdot \Zeit_k$.
\end{enumerate}
\end{lem}

In the last case where $\lm=2$, we have that $\Km$ is connected, $\Zeit_k \subset
N_\G(\Kp_c)/\Kp_c$ and $\gS^1$ normalizes $\gH_c$. It follows that if $\Km$ is contained in
$\Kp_c$, then one cannot add connected components to the isotropy groups. Note that the diagram of
the connected groups $\gH_c \subset \set{\Km, \Kp_c} \subset \G$ in this case also defines a
simply-connected cohomogeneity one manifold.

\smallskip

We use representation theory of compact Lie groups to study the isotropy representation. In the
following we introduce the notations that we will use through the paper.

The exceptional Lie groups are denoted by $\E_6$, $\E_7$, $\E_8$, $\F_4$, $\Gt$ and their Lie
algebras by $\liee_6$, $\liee_7$, $\liee_8$, $\lief_4$ and $\liegt$. The complex irreducible
representations of simple Lie algebras are highest weight representations, so we can identify the
representation with its highest weight. Each highest weight is the linear combination of the so
called fundamental weights with non-negative integer coefficients. If the Lie algebra $\lieg$ has
rank $n$, then there are $n$ fundamental weights $\fw_1, \cdots, \fw_n$. We also use the notions
$\varrho_n$, $\mu_n$ and $\nu_n$ for the standard representations of $\SO(n)$, $\SU(n)$(or $\U(n)$)
and $\Sp(n)$ on $\Real^n$, $\Cpx^n$ and $\Qua^n$ respectively. $\Delta_n$ stands for the unique
spin representation of $\SO(n)$ if $n$ is odd, and $\Delta^{\pm}_n$ stands for the two spin
representations of $\SO(n)$ if $n$ is even. The fundamental weights of exceptional Lie algebras are
specified as follows:

\smallskip

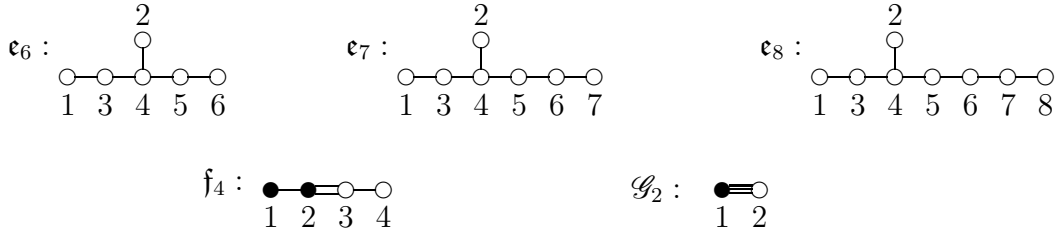
\setlength{\unitlength}{0.1cm}

\begin{picture}(2,12)
\put(5,4){$\liee_6:$} \put(12,0){\put(1,1){\circle{2}} \put(0,-4){1} \put(2,1){\line(1,0){3}}
\put(6,1){\circle{2}} \put(5,-4){3} \put(7,1){\line(1,0){3}} \put(11,1){\circle{2}} \put(10,-4){4}
\put(11,2){\line(0,1){3}} \put(11,6){\circle{2}} \put(10,8){2} \put(12,1){\line(1,0){3}}
\put(16,1){\circle{2}} \put(15,-4){5} \put(17,1){\line(1,0){3}} \put(21,1){\circle{2}}
\put(20,-4){6}}

\put(50,4){$\liee_7:$} \put(57,0){\put(1,1){\circle{2}} \put(0,-4){1} \put(2,1){\line(1,0){3}}
\put(6,1){\circle{2}} \put(5,-4){3} \put(7,1){\line(1,0){3}} \put(11,1){\circle{2}} \put(10,-4){4}
\put(11,2){\line(0,1){3}} \put(11,6){\circle{2}} \put(10,8){2} \put(12,1){\line(1,0){3}}
\put(16,1){\circle{2}} \put(15,-4){5} \put(17,1){\line(1,0){3}} \put(21,1){\circle{2}}
\put(20,-4){6} \put(22,1){\line(1,0){3}} \put(26,1){\circle{2}} \put(25,-4){7}}

\put(105,4){$\liee_8:$} \put(112,0){\put(1,1){\circle{2}} \put(0,-4){1} \put(2,1){\line(1,0){3}}
\put(6,1){\circle{2}} \put(5,-4){3} \put(7,1){\line(1,0){3}} \put(11,1){\circle{2}} \put(10,-4){4}
\put(11,2){\line(0,1){3}} \put(11,6){\circle{2}} \put(10,8){2} \put(12,1){\line(1,0){3}}
\put(16,1){\circle{2}} \put(15,-4){5} \put(17,1){\line(1,0){3}} \put(21,1){\circle{2}}
\put(20,-4){6} \put(22,1){\line(1,0){3}} \put(26,1){\circle{2}} \put(25,-4){7}
\put(27,1){\line(1,0){3}} \put(31,1){\circle{2}} \put(30,-4){8}}

\put(30,-15){\put(1,1){$\lief_4:$} \put(10,1){\circle*{2}} \put(9,-4){1} \put(11,1){\line(1,0){3}}
\put(15,1){\circle*{2}} \put(14,-4){2} \put(16,1.5){\line(1,0){3}} \put(16,0.5){\line(1,0){3}}
\put(20,1){\circle{2}} \put(19,-4){3} \put(21,1){\line(1,0){3}} \put(25,1){\circle{2}}
\put(24,-4){4}}

\put(85,-15){\put(3,0){$\liegt:$} \put(15,1){\circle*{2}} \put(14,-4){1}
\put(16,1.1){\line(1,0){3}} \put(16,1.6){\line(1,0){3}} \put(16,0.6){\line(1,0){3}}
\put(20,1){\circle{2}} \put(19,-4){2}}

\end{picture}

\vspace{2.5 cm}

We denote the standard representation of $\U(1)$ on $\Cpx$ by $\phi$. For the Lie algebras
$\lieso(4)$ and $\lieso(6)$, we specify some representations and their highest weights. For
$\lieso(4)$, the representations with highest weights $\fw_1 = \half(e_1 - e_2)$ and $\fw_2 =
\half(e_1 + e_2)$ are the two spin representations. The standard representation $\varrho_4$ of
$\SO(4)$ on $\Cpx^4$ has the highest weight $\fw_1 + \fw_2 = e_1$. For $\lieso(6)$, the
representation with highest weight $\fw_1 = e_1$ is the standard representation $\varrho_6$ of
$\SO(6)$ on $\Cpx^6$. However the representation of $\liesu(4)$ with the highest weight $\fw_1$ is
the standard representation $\mu_4$ of $\SU(4)$ on $\Cpx^4$ though $\lieso(6)$ is isomorphic to
$\liesu(4)$.

\smallskip

Some homogeneous spaces with special geometrical properties are used in the classification.

\begin{defn}
A pair of Lie group $(\K, \gH)$ is called \emph{a spherical pair} if $\K/\gH$ is a sphere.
\end{defn}
In this case, we also call the Lie algebras $(\lieg, \lieh)$ a spherical pair.

The transitive actions on spheres were classified, see, for example, \cite{BesseEinstein}. The
results are used frequently later on and we list them in Table \ref{Tabletransitiveactionsphere}.
Here $s$ is the number of the irreducible summands in the isotropy representation.

\begin{table}[!h]
\begin{center}
\begin{tabular}{|c|c|c|c|c|}
\hline
$n$ & $\K$ & $\gH$ & Isotropy representation & $s$  \\
\hline \hline
$n$ & $\SO(n+1)$ & $\SO(n)$ & $\varrho_n$ & $1$ \\
\hline
$2n+1$ & $\SU(n+1)$ & $\SU(n)$ & $[\mu_n]_\Real\oplus \Id$ & $2$ \\
\hline
$2n+1$ & $\U(n+1)$ & $\U(n)$ & $[\mu_n]_\Real \oplus \Id$ & $2$ \\
\hline
$4n+3$ & $\Sp(n+1)$ & $\Sp(n)$ & $[\nu_n]_\Real \oplus \Id \oplus \Id \oplus \Id$ & $4$ \\
\hline
$4n+3$ & $\Sp(n+1) \Sp(1)$ & $\Sp(n) \Delta \Sp(1)$ & $\nu_n\otimes \nu_1 \oplus \Id \otimes \varrho_3$ & $2$ \\
\hline
$4n+3$ & $\Sp(n+1) \U(1)$ & $\Sp(n) \Delta \U(1)$ & $[\nu_n\otimes \phi]_\Real \oplus [\Id \otimes \phi]_\Real \oplus \Id$ & $3$ \\
\hline
$15$ & $\Spin(9)$ & $\Spin(7)$ & $\varrho_7\oplus \Delta_7$ & $2$ \\
\hline
$7$ & $\Spin(7)$ & $\Gt$ & $\fw_1$ & $1$ \\
\hline
$6$ & $\Gt$ & $\SU(3)$ & $[\mu_3]_\Real$ & $1$ \\
\hline
\end{tabular}
\end{center}
\smallskip
\caption{Transitive actions on $\sph^n$.} \label{Tabletransitiveactionsphere}
\end{table}

\begin{defn}
A homogeneous space $\G/\gH$ is called \emph{isotropy irreducible} if the isotropy representation
of $\gH$ is irreducible. If the isotropy representation by the identity component $\gH_c$ of $\gH$
is also irreducible, then it is called \emph{strongly isotropy irreducible}.
\end{defn}

Every irreducible symmetric space is strongly isotropy irreducible. J. Wolf classified compact
strongly isotropy irreducible spaces which are not symmetric spaces, see \cite{WolfIrr}. Compact
homogeneous spaces which are isotropy irreducible but not strongly isotropy irreducible were
classified by M. Wang and W. Ziller, see \cite{WangZillerIsotropy}. If the space $\G/\gH$ is
strongly isotropy irreducible, then the isotropy representation $\ad_{\lieh}$ is also irreducible.
In this case, we call the pair of Lie algebras $(\lieg, \lieh)$ strongly isotropy irreducible.

\medskip


\section{Classification of cohomogeneity one manifolds with $s=1,2$}

In this section, we classify simply-connected cohomogeneity one manifolds with $s=1, 2$. The
results are also used in the classification of the case where $s=3$.

\subsection{Fixed point actions}

If $s=1$, then the cohomogeneity one action has two fixed points, i.e., $\G = \Km = \Kp$ and the
manifold is a sphere. By the classification of transitive actions of the sphere, the cohomogeneity
one manifolds are listed in Table \ref{tablesequal1}
\begin{table}[!h]
\begin{center}
\begin{tabular}{|c||c|c|c|}
\hline
$M$ & $\G = \Kpm$ & $\gH$ &  \\
\hline \hline
$\sph^{n+1}$ & $\SO(n+1)$ & $\SO(n)$ & $n\geq 1$ \\
\hline
$\sph^7$ & $\Gt$ & $\SU(3)$ & \\
\hline
$\sph^8$ & $\Spin(7)$ & $\Gt$ & \\
\hline
\end{tabular}
\end{center}
\smallskip
\caption{Cohomogeneity one manifolds with $s=1$.}\label{tablesequal1}
\end{table}

Next we quote the result when there is only one fixed point, say $\Km = \G$ and $\Kp \subset \G$ is
a proper subgroup, see \cite{HoelscherClass} and \cite{GroveSearleJDG}.

\begin{prop}\label{propfixedpointaction}
If a simply-connected manifold $M$ admits a cohomogeneity one action with exactly one fixed points,
then $M$ is a (complex or quaternion) projective space or the Cayley plane with an isometric
action. They are classified in Table \ref{Tablefixedpointaction}.
\end{prop}

\begin{table}[h!]
\begin{center}
\begin{tabular}{|c|l|c|}
\hline
$M$ & $\quad \quad \quad \quad \quad \quad \quad \quad \quad $ Group diagram & $s$ \\
\hline \hline
$\Cp^n$ & $\SU(n) \supset \set{\SU(n), \, S(\U(n-1)\U(1))} \supset \SU(n-1)$ & $2$ \\
\hline
$\Cp^n$ & $\U(n) \supset \set{\U(n), \, \U(n-1)\U(1)} \supset \U(n-1)$ & $2$ \\
\hline
$\Cayp^2$ &$\Spin(9) \supset \set{\Spin(9), \, \Spin(8)} \supset \Spin(7)$ & $2$ \\
\hline \hline
$\Hp^n$ & $\Sp(n)\times \Sp(1) \supset \set{\Sp(n)\times \Sp(1), \, \Sp(n-1)\Sp(1)\times\Sp(1)} \supset \Sp(n-1)\Diag\Sp(1)$ & $2$ \\
\hline
$\Hp^n$ & $\Sp(n)\times\U(1) \supset \set{\Sp(n)\times \U(1), \, \Sp(n-1)\Sp(1)\times \U(1)} \supset \Sp(n-1)\Diag\U(1)$ & $3$ \\
\hline
$\Cp^{2n+1}$ & $\Sp(n)\times\U(1) \supset \set{\Sp(n)\times\U(1), \, \Sp(n-1)\U(1)\times\U(1)} \supset \Sp(n-1)\Diag\U(1)$ & $3$ \\
\hline \hline
$\Hp^n$ & $\Sp(n) \supset \set{\Sp(n), \, \Sp(n-1)\Sp(1)} \supset \Sp(n-1)$ & $4$ \\
\hline
$\Cp^{2n+1}$ & $\Sp(n) \supset \set{\Sp(n), \, \Sp(n-1)\U(1)} \supset \Sp(n-1)$ & $4$ \\
\hline
\end{tabular}
\smallskip
\end{center}\caption{The cohomogeneity one action with one fixed point}\label{Tablefixedpointaction}
\end{table}

\subsection{Classification with $s=2$}
We assume that the action has no fixed points. If the action is primitive, then the manifold is a
sphere, see Theorem \ref{thmclasss2primitive}. If the action is non-primitive, then the manifold is
a double, i.e., $\Kpm = \K$, and we classified the triples $\gH \subset \K \subset \G$, see Theorem
\ref{thmclasss2nonprimitive}.

\begin{thm}\label{thmclasss2primitive}
Suppose $M$ is a compact simply-connected manifold that admits a primitive cohomogeneity one action
with $s=2$ and no fixed points. Then one of the followings holds.
\begin{enumerate}
\item The manifold is $\sph^{15}$ with the diagram $\Gt \subset \set{\Spin^+(7), \Spin^-(7)} \subset
\Spin(8)$ and the embedding $\Spin(8) \subset \SO(16)$ is given by the representation
$\Delta_8^+\oplus \Delta_8^-$.
\item The manifold is a sphere with a sum action.
\end{enumerate}
\end{thm}

\begin{proof}
From the assumption $s=2$, the space $\liep$ of the representation $\Ad_{\gH_c}$ splits into two
subspaces which are denoted by $\liep_1$ and $\liep_2$, and the representation of $\Ad_{\gH_c}$ on
each of them is irreducible.

First we assume that the two summands $\liep_1$ and $\liep_2$ are equivalent representations of
$\gH_c$. Let $\K = \Km$ and we consider the group triple $\gH \subset \K \subset \G$. Since the
sphere $\K/\gH$ is isotropy irreducible, its effective version is one of the pairs
$\SO(n+1)/\SO(n)$($n\geq 2$), $\Spin(7)/\Gt$, $\U(1)$, $\Gt/\SU(3)$, $\Sp(1)/\U(1)$,
$\Sp(2)/(\Sp(1)\times \Sp(1))$, $\SU(2)\times \SU(2)/\Diag \SU(2)$ and $\SU(4)/\Sp(2)$. Suppose
$\K/\gH = (\SO(n+1)\cdot \gL)/(\SO(n)\cdot \gL)$ for some $\gL$, then the two representations are
$\varrho_n \ot\Id$. In particular they have dimension $n$. However $\SO(n+1)$ has no irreducible
representation with dimension $n$ if $n \geq 2$. Similarly, $\K/\gH$ cannot be $\Gt/\SU(3)$,
$\Sp(1)/\U(1)$, $\Sp(2)/(\Sp(1)\times \Sp(1))$ and $\SU(4)/\Sp(2)$. The possible triples are
\begin{equation*}
\Gt \subset \Spin(7) \subset \SO(8), \quad \set{1} \subset \U(1) \subset \U(1)\times \U(1).
\end{equation*}
There are two different $\Spin(7)$, denoted by $\Spin^+(7)$ and $\Spin^-(7)$, in $\SO(8)$ that
contains $\Gt$ and they differ by an automorphism of $\SO(8)$. So there is one primitive diagram
from this triple:
\begin{equation*}
\Gt \subset \set{\Spin^+(7), \Spin^-(7)} \subset \Spin(8)
\end{equation*}
and the manifold is $\sph^{15}$. For the second triple $\set{1} \subset \U(1) \subset \U(1)\times
\U(1)$, one may choose different embedding of $\U(1) \subset \U(1)\times \U(1)$ for $\Kpm$ such
that the diagram is primitive. However all are sum actions on $\sph^3$.

Next we assume that $\liep_1$ and $\liep_2$ are non-equivalent representations. From the assumption
the action $\ad_{\lieh}$ on the spaces $\liep_1$ and $\liep_2$ are irreducible, W.L.O.G., we may
assume that $\liekm= \lieh\oplus \liep_1$ and $\liekp = \lieh\oplus \liep_2$. We denote $Q(X,Y)$ by
$\ip{X, Y}$ for $X, Y \in \lieg$. Then for any $X_1, Y_1 \in \liep_1$, $X_2, Y_2 \in \liep_2$ and
$Y_0 \in \lieh$, since $[X_2, Y_0] \in \liekp = \lieh \oplus \liep_2$, $[Y_1, X_1] \in \liekm =
\lieh\oplus \liep_1$ and $[X_2, Y_2] \in \liekp = \lieh \oplus \liep_2$, we have
\begin{eqnarray*}
\ip{[X_1, X_2], Y_0} & = & \ip{X_1, [X_2, Y_0]} = 0, \\
\ip{[X_1, X_2], Y_1} & = & \ip{Y_1, [X_1, X_2]} = \ip{[Y_1, X_1], X_2} = 0, \\
\ip{[X_1, X_2], Y_2} & = & \ip{X_1, [X_2, Y_2]} = 0.
\end{eqnarray*}
Therefore $[X_1, X_2]$ is orthogonal to any vector in $\lieg$, i.e., $[\liep_1, \liep_2] = 0$. We
define the following subspaces of $\lieg$. Let
\begin{equation*}
\lieh_0 = \Ann(\liep_1\oplus \liep_2)\cap\lieh = \set{X \in \lieh | [X,Y] = 0 \mbox{ for any }Y \in
\liep_1\oplus \liep_2},
\end{equation*}
and
\begin{equation*}
\lieh_i = \Ann(\liep_i)\cap \lieh_0^\perp\cap \lieh, \quad (i =1,2), \quad \lieh_3= (\lieh_0\oplus
\lieh_1\oplus\lieh_2)^\perp\cap \lieh,
\end{equation*}
where $\perp$ is the orthogonal complement with respect to the inner product $Q$.

Since $[\liep_1 , \liep_1] \subset \liekm = \lieh\oplus \liep_1$, $[\liep_1, \lieh] \subset
\liep_1$ and
\begin{equation*}
[[\liep_1, \liep_1],\liep_2] = - [[\liep_1, \liep_2],\liep_1] - [[\liep_2, \liep_1], \liep_1] = 0,
\end{equation*}
we have $[\liep_1, \liep_1] \subset \liep_1 \oplus \lieh_0 \oplus \lieh_2$. Moreover since
\begin{equation*}
\ip{[\liep_1, \liep_1], \lieh_0} = \ip{\liep_1, [\liep_1, \lieh_0]} = \ip{\liep_1, 0} = 0,
\end{equation*}
we have $[\liep_1, \liep_1] \subset \liep_1\oplus \lieh_2$. Denote the Lie algebra generated by
$\liep_1$ by $\Lie(\liep_1)$, then $\Lie(\liep_1) \subset \liep_1 \oplus \lieh_2$. If there is a
vector $X \in \liep_1 \oplus \lieh_2$ such that $X \perp \Lie(\liep_1)$, then $X \in \lieh_2$ and
thus $[X, \liep_1] \subset \liep_1$. For any vector $Y \in \liep_1$, we have $\ip{[X,\liep_1], Y} =
\ip{X,[\liep_1, Y]} = 0$ which implies that $X \in \Ann(\liep_1)\cap \lieh_2 = 0$. Hence we have
$\Lie(\liep_1) = \liep_1\oplus \lieh_2$. Similarly we also have $\Lie(\liep_2) = \liep_2 \oplus
\lieh_1$.

We claim that $\lieh_3 =0$. In fact, first we have $\ip{[\lieh_3,\liep_1], \liep_1} =
\ip{\lieh_3,[\liep_1, \liep_1]} = 0$ since $[\liep_1, \liep_1] \subset \liep_1\oplus \lieh_2$.
Since $\lieh_3 \subset \lieh$, $[\lieh_3, \liep_1] \subset \liep_1$ and thus $[\lieh_3, \liep_1] =
0$ which implies that $\lieh_3 \subset \Ann(\liep_1) = \lieh_0\oplus \lieh_1$. So we have $\lieh_3
= 0$ by its definition.

By the Jacobi identity, we have
\begin{eqnarray*}
0 & = & [[\lieh_0, \lieh], \liep_1\oplus \liep_2] + [[\lieh, \liep_1\oplus \liep_2], \lieh_0] +
[[\liep_1\oplus\liep_2, \lieh_0], \lieh] \\
& = & [[\lieh_0, \lieh], \liep_1\oplus \liep_2] + [[\lieh, \liep_1\oplus \liep_2], \lieh_0].
\end{eqnarray*}
Since $[\lieh, \liep_1\oplus \liep_2]\subset \liep_1 \oplus \liep_2$ and then $[[\lieh,
\liep_1\oplus \liep_2], \lieh_0] = 0$, we have $[[\lieh_0, \lieh],\liep_1\oplus\liep_2]= 0$ which
implies that $[\lieh_0, \lieh] \subset \lieh_0$, i.e., $\lieh_0$ is an ideal of $\lieh$. Similarly,
we have that $[\lieh_1, \lieh]\subset \Ann(\liep_1)\cap \lieh$. Furthermore $\langle
[\lieh_1,\lieh],\lieh_0\rangle = \langle \lieh_1,[\lieh,\lieh_0]\rangle = 0$ since $\lieh_0$ is an
ideal of $\lieh$. It follows that $[\lieh_1, \lieh] \subset \lieh_1$, i.e., $\lieh_1$ is also an
ideal of $\lieh$. Similarly $\lieh_2$ is an ideal of $\lieh$. Since $\lieg = \lieh \oplus \liep_1
\oplus \liep_2$ and $\lieh_0$ annihilates $\liep_1\oplus \liep_2$, $\lieh_0$ is an ideal of
$\lieg$. By the assumption that the $\G$ action is almost effective, we have $\lieh_0 = 0$. So we
have
\begin{equation*}
\lieg=\lieh_1\oplus \lieh_2 \oplus \liep_1 \oplus \liep_2, \quad \mbox{ and } \quad \liekm=\lieh_1
\oplus \lieh_2 \oplus \liep_1, \quad \liekp=\lieh_1\oplus \lieh_2 \oplus \liep_2.
\end{equation*}
We claim that $\Lie(\liep_1)=\lieh_2\oplus\liep_1$ is an ideal in $\lieg$. In fact, $[\lieh_2
\oplus \liep_1, \lieh_1] = [\lieh_2, \lieh_1] = 0$ and $[\lieh_2 \oplus \liep_1, \liep_2] = 0$
imply that $[\Lie(\liep_1),\lieg] \subset \Lie(\liep_1)$. Similarly $\Lie(\liep_2)$ is also an
ideal in $\lieg$. Therefore
\begin{equation*}
\G = \gL_2\times \gL_1, \quad \Km=\gH_1\times \gL_1, \quad \Kp = \gL_2\times \gH_2, \quad \mbox{
and }\quad \gH=\gH_1 \times \gH_2,
\end{equation*}
where $\lieh_i$ is the Lie algebra of $\gH_i$, $\Lie(\liep_i)$ is the Lie algebra of $\gL_i$ for
$i=1,2$ and $\gL_1/\gH_2$, $\gL_2/\gH_1$ are spheres. Hence the $\G$-action is a sum action and the
manifold $M$ is $\G$-equivariant to a sphere.
\end{proof}

\begin{thm}\label{thmclasss2nonprimitive}
Suppose $M$ is a compact simply-connected manifold that admits a cohomogeneity one action with
$s=2$ and no fixed points. If the action is non-primitive, then the manifold is a double and the
triples $\gH \subset \K \subset \G$ with $\gH$ connected are classified in Table
\ref{TableHKG2summandsKHsphere}.
\end{thm}

\begin{proof}
The manifold $M$ has the diagram as $\gH \subset \set{\Kpm = \K} \subset \G$ where $\K$ is a proper
subgroup of $\G$ and $\K/\gH$ is a sphere.

First we classify the triples $\G \supset \K \supset \gH$ such that $\G$ is simple, $\gH$ is
connected, $\K/\gH$ is a sphere and the isotropy representation $\Ad_{\gH}$ of $\G/\gH$ has two
irreducible summands. It follows that $(\G, \K)$ and $(\K, \gH)$ are strongly isotropy irreducible
and the isotropy representation of $\G/\K$ remains irreducible when restricted to $\gH$. From the
classification of transitive actions on spheres, the effective version of $\K/\gH$ is one of
$\Spin(7)/\Gt$, $\Gt/\SU(3)$ and $\SO(n+1)/\SO(n)$ with $n \geq 1$. Using the classification of
compact irreducible symmetric spaces and J. Wolf's classification, $\G \supset \K \supset \gH$ is
one in the first part of Table \ref{TableHKG2summandsKHsphere}. The last column contains further
conditions. If a homogeneous space appears in this column, it means that the space is strongly
isotropy irreducible, for example, ``$\G_1/\gH_1, \G_2/\gH_2 = \sph^k$'' means that both spaces are
strongly isotropy irreducible and the second one is also a sphere.

\begin{table}[!h]
\begin{tabular}{|l|l|l|c|}
\hline $\quad \quad \G$ & $\quad \quad \quad \K$ & $\quad \quad \quad \gH$ & \\
\hline \hline

$\SO(8)$ & $ \Spin(7)$ & $\Gt$ & \\
\hline
$\Spin(9)$ & $\Spin(8)$ & $\Spin(7)$ & \\
\hline
$\Spin(9)$ & $\Spin(7) \cdot \SO(2)$ & $\Gt\times \SO(2)$ & \\
\hline
$\SO(2n)$ & $\U(n)$ & $\SU(n)$ & $n\geq 2$, $n\ne 4$ \\
\hline
$\SU(5)$ & $\U(4)$ & $\U(1)\cdot \Sp(2)$ & \\
\hline
$\SU(p+q)$ & $S(\U(p)\times \U(q))$ & $\SU(p)\times \SU(q)$ & $p, q\geq 1, pq\geq 2$ \\
\hline
$\Sp(16)$ & $\Spin(12) $ & $\Spin(11)$ & \\
\hline
$\Sp(n)$ & $\U(n)$ & $\SU(n)$ & $n \geq 3$ \\
\hline
$\Sp(n+1)$ & $\Sp(n) \times \Sp(1)$ & $\Sp(n) \times \U(1)$ & $n \geq 1$ \\
\hline
$\E_6$ & $\Spin(10) \times \SO(2)$ & $\Spin(10)$ & \\
\hline
$\E_6$ & $\SU(6) \times \SU(2)$ & $\SU(6) \times \U(1)$ & \\
\hline
$\E_7$ & $\E_6 \times \SO(2)$ & $\E_6$ & \\
\hline
$\E_7$ & $\Spin(12) \times \Sp(1)$ & $\Spin(12)\times \U(1)$ & \\
\hline
$\E_7$ & $\Spin(12) \times \Sp(1)$ & $\Spin(11) \times \Sp(1)$ & \\
\hline
$\E_8$ & $\Spin(16)$ & $\Spin(15)$ & \\
\hline
$\E_8$ & $\Sp(1) \times \E_7$ & $\U(1) \times \E_7$ & \\
\hline
$\F_4$ & $\Sp(3) \times \Sp(1)$ & $\Sp(3)\times \U(1)$ & \\
\hline \hline {$\G_1 \times \G_2$} & {$\gH_1 \times \G_2$} & {$\gH_1 \times \gH_2$} & $\G_1/\gH_1$, $\G_2/\gH_2 = \sph^{k}$ \\
\hline
$\U(1)\times \G_1$ & $\U(1) \times \U(1) \times \gH_1$ & $\Diag \U(1) \times \gH_1$ & $\G_1/(\U(1) \gH_1)$ \\
\hline
$\SU(2) \times \G_1$ & $\SU(2) \times \SU(2) \times \gH_1$ & $\Diag \SU(2) \times \gH_1$ & $\G_1/(\SU(2) \gH_1)$ \\
\hline
\end{tabular}
\smallskip
\caption{Group triple $\G \supset \K \supset \gH$ such that $\gH$ is connected, $\Ad_{\G/\gH}$ has
$2$ irreducible summands and $\K/\gH$ is a sphere.} \label{TableHKG2summandsKHsphere}
\end{table}

The cohomogeneity one manifold defined by the diagram $\gH \subset \set{\Kpm = \K} \subset \G$ is a
sphere bundle over the homogeneous spaces $\G/\K$ which is an irreducible symmetric space except
for $\Sp(16)/\Spin(12)$. If $\K/\gH$ is a circle, then one can add components to $\gH$. Then for
each positive integer $n$, we have the cohomogeneity one diagram
\begin{equation*}
\gH \cdot \Zeit_n \subset \set{\Kpm = \K} \subset \G.
\end{equation*}

If $\G$ is not simple, then there are two constructions of such triples of connected groups $\gH
\subset \K \subset \G$, see the last two examples in Table \ref{TableHKG2summandsKHsphere}. In the
first case, $\G = \G_1 \times \G_2$, $\K = \gH_1\times \G_2$ and $\gH = \gH_1 \times \gH_2$ where
$\G_1/\gH_1$ is a strongly isotropy irreducible space and $\G_2/\gH_2$ is one of $\Spin(7)/\Gt$,
$\Gt/\SU(3)$ and $\SO(n+1)/\SO(n)$($n \geq 1$). The cohomogeneity one manifold is the product of a
sphere with the homogeneous space $\G_1/\gH_1$. In the second and last cases, $\G = \gH_0 \times
\G_1$, $\K = \gH_0 \times \gH_0 \times \gH_1$ and $\gH = \Diag \gH_0 \times \gH_1$ where $\gH_0$ is
either $\U(1)$ or $\SU(2)$ and $\G_1/\gH_0\times \gH_1$ is a strongly isotropy irreducible space.
The diagram is reducible and its nonreducible version is $\gH_1 \subset \set{\gH_0 \times \gH_1,
\gH_0 \times \gH_1} \subset \G_1$. If $\gH_0$ is $\U(1)$, then one can add components to $\gH$.
Then for each positive integer $n$, we have the diagram
\begin{equation*}
(\Zeit_n \cdot \Delta \U(1)) \times \gH_1 \subset \set{\Kpm = \U(1) \times \U(1)\times \gH_1}
\subset \U(1) \times \G_1.
\end{equation*}
Its non-reducible version is $\Zeit_n \cdot\gH_1 \subset \set{\Kpm = \U(1)\times \gH_1} \subset
\G_1$. The manifolds defined by these diagrams are sphere($\sph^2$ or $\sph^4$) bundles over
$\G_1/(\gH_0\times \gH_1)$.
\end{proof}

\begin{rem}\label{rems2lowdim}
There are two examples in low dimensions. One example has the diagram
\begin{equation*}
\Zeit_n\cdot \SU(2) \subset \set{\Kpm = \U(2)} \subset \SU(3)
\end{equation*}
and the manifold is 6 dimension, see example $N^6_F$ in \cite{HoelscherClass}. The other one has
the diagram
\begin{equation*}
\Sp(1)\U(1) \subset \set{\Kpm = \Sp(1)\Sp(1)} \subset \Sp(2)
\end{equation*}
and the manifold is 7 dimension, see example $N^7_I$ in \cite{HoelscherClass}.
\end{rem}

\medskip


\section{Special types of cohomogeneity one actions with $s=3$}

From this section on we consider the classification when $s=3$. In this section, we look at some
special types of cohomogeneity one actions, i.e., reducible and non-primitive actions. The action
of $\G$ is assumed to be effective or almost effective, i.e., the ineffective kernel is finite.


\subsection{Reducible actions}

The main result is

\begin{thm}\label{thmclassificationreducible}
If a simply-connected cohomogeneity one manifold $M$ admits a reducible action without fixed points
and $s= 3$, then one of the followings holds:
\begin{enumerate}
\item it admits a non-reducible action with $s=3$;
\item the action is primitive and is a sum action on a sphere;
\item it is a double, i.e., $\Km = \Kp$;
\item the action is non-primitive with different $\Kpm$ and the manifold is a sphere, $\Cp^2$, $\Hp^n$ or a three dimensional lens space bundle over a
homogeneous space.
\end{enumerate}
The cohomogeneity one manifolds in case (3) and (4) are classified in Table
\ref{Tables3reducibledouble} and \ref{Tables3reduciblenondouble}.
\end{thm}

\begin{proof}
We prove the theorem in two steps. In step I, we classify cohomogeneity one diagrams with connected
isotropy subgroups. In step II, we consider the possible ways to get new diagrams from those
obtained in Step I, i.e., variations by automorphisms and adding connected components to isotropy
groups. Here we only consider the inner automorphism, i.e., conjugation by a group element.

\textsc{Step I.} Suppose $\lieg = \lieg_1 \oplus \lieg_2$ and we may assume that $\Proj_2(\lieh) =
\lieg_2$ and that the projection from $\lieh$ to any primitive factor of $\lieg_1$ is not
surjective.

Let $\lieh_i = \lieg_i\cap\lieh$ for $i= 1, 2$. If $\lieh=\lieh_1 \oplus \lieh_2$, then from the
reducibility assumption, we have that $\lieh_2 = \lieg_2$ and $\lieh_1$ is a proper subalgebra of
$\lieg_1$. In this case, the non-reducible action by $\G_1$ also has $s=3$.

Next we assume that there exists a nonzero subspace $\lieh_0 \subset \lieh$ such that the images
under the two projections $\Proj_i$ are nonzero, i.e., $\lieh= \lieh_1 \oplus \Diag \lieh_0\oplus
\lieh_2$ and $\Proj_i(\lieh) = \lieh_i \oplus \lieh_0$ for $i=1, 2$. From the reducibility
assumption, we have $\lieg_2 = \lieh_0 \oplus \lieh_2$. Since the action of $\G$ is effective, we
have $\lieh_2 = 0$. There is an intermediate subalgebra $\lieh_1\oplus\lieh_0\oplus \lieh_0$
between $\lieh = \lieh_1 \oplus \Diag\lieh_0$ and $\lieg = \lieg_1 \oplus \lieh_0$ and the
principal isotropy representation is
\begin{equation*}
\chi = (\ad_{\lieg_1/(\lieh_1\oplus\lieh_0)})\oplus (\Id_{\lieh_1}\ot \ad_{\lieh_0}).
\end{equation*}
It follows that $\lieh_0$ has at most $2$ primitive factors.

\textsc{Case I.} If $\lieh_0$ has $2$ primitive factors as $\lieh_0 = \lieh_0' \oplus \lieh_0''$,
then the pair $(\lieg_1, \lieh_1\oplus \lieh_0)$ is strongly isotropy irreducible and thus
$\lieg_1$ is a simple Lie algebra or $\lieso(4)$. If $\lieg_1 = \lieso(4)$, then
$\lieh_1\oplus\lieh_0 = \lieso(2) \oplus \lieso(2)$ and the dimension of the manifold is smaller or
equal to $7$. So we assume that $\lieg_1$ is simple. The three irreducible summands are
\begin{equation*}
\chi_1 = \ad_{\lieg_1/(\lieh_1\oplus \lieh_0)}, \quad \chi_2 = \Id_{\lieh_1} \ot \ad_{\lieh_0'} \ot
\Id_{\lieh_0''} , \quad \chi_3 = \Id_{\lieh_1}\ot \Id_{\lieh_0'} \ot \ad_{\lieh_0''},
\end{equation*}
and their corresponding representation spaces are denoted by $\liep_1$, $\liep_2$ and $\liep_3$.

We claim that any two of the three representations are non-equivalent. In fact, if $\chi_2 =
\chi_3$, then $\lieh_0' = \lieh_0'' = \lieu(1)$. From the classification of strongly isotropy
irreducible spaces, such pair $(\lieg_1, \lieh_1\oplus\lieh_0)$ with $\lieg_1$ simple does not
exist. If $\chi_1$ is equivalent to one of $\chi_2$ or $\chi_3$, say $\chi_2$, then the isotropy
representation of the pair $(\lieg_1, \lieh_1\oplus \lieh_0)$ is given by $\zeta =
\Id_{\lieh_1}\ot\ad_{\lieh_0'}\ot\Id_{\lieh_0''}$. It is clear that $\lieh_0' \ne \lieu(1)$, and
furthermore there is no strongly isotropy irreducible pair with isotropy representation as $\zeta$.

We consider the intermediate subalgebra $\liek$. If it contains the subspace $\liep_1$, then it is
one of
\begin{enumerate}
\item $\lieh\oplus \liep_1$, and then $[\liep_1, \liep_1] \subset \lieh_1\oplus \liep_1$, i.e., $\liel = \lieh_1\oplus\liep_1$ is a Lie algebra;
\item $\lieh_1 \oplus \lieh_0'\oplus \lieh_0' \oplus \Diag\lieh_0'' \oplus \liep_1$, and
$[\liep_1, \liep_1] \subset \lieh_1\oplus \lieh_0' \oplus \liep_1$;
\item $\lieh_1 \oplus \Diag \lieh_0'\oplus \lieh_0'' \oplus \lieh_0'' \oplus \liep_1$, and
$[\liep_1, \liep_1] \subset \lieh_1\oplus \lieh_0'' \oplus \liep_1$.
\end{enumerate}
In all above cases, $(\lieg_1, \lieh_1\oplus\lieh_0)$ is not a symmetric pair.

In Case $(1)$, let $\G_1$, $\gH'$ and $\gL$ be the corresponding Lie groups of $\lieg_1$,
$\lieh_1\oplus \lieh_0$ and $\liel$, and then the subgroup $\gL$ acts transitively on the
homogeneous space $\G_1/\gH'$. A. L. Onishchik classified all triples $(\gL_1, \gL_2, \gL_3)$ such
that $\gL_1$ is simple, and $\gL_2$ is a subgroup of $\gL_1$ and acts transitively on the
homogeneous space $\gL_1/\gL_3$, see \cite{GorOnishchik}, p. 143 Theorem 4.5, or \S 2 in
\cite{DAtriZiller}. This classification is also used in \cite{KerinShankar} and the following Table
\ref{tableOnishchik} is part of Table 3 in the appendix of their paper where $\gL_1/\gL_3$ is not a
symmetric space.

\begin{table}[h!]
\begin{center}
\begin{tabular}{|c|c|c|c|}
\hline
$\gL_1$ & $\gL_2$ & $\gL_3$ & $\gL_2\cap\gL_3$  \\
\hline \hline
$\SO(4n)$ & $\SO(4n-1)$ & $\Sp(n)$ & $\Sp(n-1)$ \\
\hline
$\SO(4n)$ & $\SO(4n-1)$ & $\Sp(n)\U(1)$ & $\Sp(n-1)\U(1)$ \\
\hline
$\SO(4n)$ & $\SO(4n-1)$ & $\Sp(n)\Sp(1)$ & $\Sp(n-1)\Sp(1)$ \\
\hline
$\SO(2n)$ & $\SO(2n-1)$ & $\SU(n)$ & $\SU(n-1)$ \\
\hline
$\SO(16)$ & $\SO(15)$ & $\Spin(9)$ & $\Spin(7)$ \\
\hline
$\SO(8)$ & $\Spin(7)$ & $\SO(6)$ & $\SU(3)$ \\
\hline
$\SO(8)$ & $\Spin(7)$ & $\SO(5)$ & $\SU(2)$ \\
\hline
$\SO(8)$ & $\Spin(7)$ & $\SO(2)\SO(5)$ & $\SO(2)\SU(2)$\\
\hline
$\SO(7)$ & $\Gt$ & $\SO(5)$ & $\SU(2)$ \\
\hline
\end{tabular}
\end{center}\caption{Onishchik's triples $(\gL_1,\gL_2,\gL_3)$ with $\gL_1$ simple and $\gL_1/\gL_3$ non-symmetric}\label{tableOnishchik}
\end{table}
From the classification, $(\SO(4n), \SO(4n-1), \Sp(1)\Sp(n))(n\geq 2)$ is the only triple such that
$\gL_1/\gL_3$ is strongly isotropy irreducible and $\gL_3$ has at least two primitive factors. If
$\lieg_1 = \lieso(4n)$ and $\lieh_0 = \liesp(1)\oplus \liesp(n)$, then $\lieh_1 = 0$ and $\liel =
\liep_1$ would be  $\lieso(4n-1)$ which would imply that $4n-1 = 3 + \dim\liesp(n)$ and it gives us
a contradiction.

In Case $(2)$, $\lieh_1 \oplus \lieh_0' \oplus \liep_1$ is a subalgebra of $\lieg_1$. From a
similar argument as in the previous case, we have that $\lieg_1 = \lieso(4n)(n\geq 2)$, $\lieh_0 =
\liesp(1)\oplus \liesp(n)$, $\lieh_1 = 0$ and $\dim \liep_1 = \dim \lieso(4n) - \dim \liesp(1) -
\dim \liesp(n) = 6n^2 - 3n - 3$. However it is not equal to either $\dim \lieso(4n-1) -
\dim\liesp(1) = 8n^2 - 6n -2$ or $\dim \lieso(4n-1) - \dim \liesp(n) = 6n^2 - 7n + 1$ for $n\geq
2$. So $\liek$ is not in this case. A similar argument also show that $\liek$ is not in Case $(3)$.

Now we assume that $\liek$ does not contain the subspace $\liep_1$, then it is one of
\begin{enumerate}
\item $\lieh_1\oplus \lieh_0\oplus \lieh_0$;
\item $\lieh_1\oplus \lieh_0'\oplus \lieh_0'\oplus \Diag\lieh_0''$;
\item $\lieh_1\oplus \Diag\lieh_0'\oplus \lieh_0''\oplus \lieh_0''$.
\end{enumerate}
In Case $(1)$, $(\liek,\lieh)$ is not a spherical pair. If both $\liekpm$ are in Case $(2)$, then
$\lieh_0'$ is either $\lieu(1)$ or $\liesu(2)$ and the diagram is not primitive. This gives us
example \textbf{R}.1($m=1$) and \textbf{R}.2.

If both $\liekpm$ are in Case $(3)$, then we have a similar result. If $\liekm$ is in Case $(2)$
and $\liekp$ is in Case $(3)$, then $\lieh_0'$ and $\lieh_0''$ are $\lieu(1)$ or $\liesu(2)$. From
the classification of strongly isotropy irreducible spaces, both $\lieh_0'$ and $\lieh_0''$ cannot
be $\lieu(1)$. If both $\lieh_0'$ and $\lieh_0''$ are $\liesu(2)$, though the embeddings of
$\liekpm \subset \lieg$ are different, the manifold is equivariant diffeomorphic to the one in the
previous example. If $\lieh_0 = \lieu(1)\oplus \liesu(2)$, then we have example
\textbf{R}.14($m=1$).

\smallskip

\textsc{Case II.} If $\lieh_0$ is primitive, then the isotropy representation of the pair
$(\lieg_1, \lieh_1\oplus \lieh_0)$ has $2$ irreducible summands $\chi_1$, $\chi_2$ and their
representation spaces are denoted by $\liep_1$ and $\liep_2$. The representation space of $\chi_3 =
\Id_{\lieh_1}\ot \ad_{\lieh_0}$ is denoted by $\liep_3$. From the assumption that the projection
from $\lieh$ to any primitive factor of $\lieg_1$ is not surjective, $\lieg_1$ has at most two
primitive factors.

\textsc{Case II.A.} We first consider the case when $\lieg_1$ has two factors $\lieg_1'$ and
$\lieg_1''$, and then we may assume that $\lieg_1' = \lieh_1'\oplus \liep_1$ and $\lieg_1'' =
\lieh_1''\oplus\lieh_0\oplus \liep_2$ where $\lieh_1 =\lieh_1'\oplus \lieh_1''$. The only possible
pair of equivalent summands are $\chi_1$ and $\chi_3$. If we are in this case, then $\lieh_1= 0$,
$\lieg_1' = \lieu(1)$, $\lieh_0 = \lieu(1)$ and $\lieg_1'' = \liesu(2)$. Thus the manifold is $5$
dimensional. Now we assume that $\chi_i$'s are pairwisely non-equivalent. $\lieh\oplus \liep_2$ and
$\lieh\oplus \liep_1\oplus \liep_2$ are not subalgebras of $\lieg$ otherwise $\liep_2$ would be a
subalgebra of $\lieg_1''$. Furthermore the intermediate subalgebra cannot be $\lieh\oplus
\liep_3\oplus \liep_1 = \lieg_1'\oplus\lieh_1''\oplus \lieh_0\oplus\lieh_0$ since
$(\lieg_1'\oplus\lieh_1''\oplus \lieh_0\oplus\lieh_0, \lieh_1'\oplus \lieh_1''\oplus \Delta
\lieh_0)$ is not a spherical pair. So the intermediate subalgebra $\liek$ is one of the followings:
\begin{enumerate}
\item $\lieh\oplus \liep_3\oplus \liep_2 = \lieh_1'\oplus\lieg_1''\oplus\lieh_0$, and then $(\lieh_1'\oplus \lieg_1''\oplus \lieh_0, \lieh_1 \oplus \Diag
\lieh_0)$ is a spherical pair;
\item $\lieh\oplus \liep_3 = \lieh_1 \oplus\lieh_0\oplus \lieh_0$, and then $\lieh_0$ is either
$\lieu(1)$ or $\liesu(2)$;
\item $\lieh\oplus \liep_1 = \lieg_1'\oplus \lieh_1''\oplus \Diag\lieh_0$, and then $(\lieg_1',\lieh_1')$ is a strongly isotropy irreducible spherical pair.
\end{enumerate}
In Case $(1)$ the spherical is $(\lieu(n+1), \lieu(n))$ with $\lieh_0 = \lieu(1)$ or
$(\liesp(n+1)\oplus \liesp(1), \liesp(n)\oplus \Delta\liesp(1))$ with $\lieh_0 = \liesp(1)$. For
the first pair, since the sub-action by $\G_1 \times \SU(n+1) \subset \G_1 \times \U(n+1)$ also has
$s=3$ if $n\geq 2$, we only consider the pair $(\U(2), \U(1))$. If both $\liekpm$ are in Case
$(2)$, then we have example \textbf{R}.3 and \textbf{R}.4.

In Case $(2)$, if $\lieh_0 = \lieu(1)$, then $\G_1''/\gH_1''\U(1)$ is strongly isotropy
irreducible. If $\lieh_0 = \liesu(2)$, then $\G_1''/\gH_1''\SU(2)$ is strongly isotropy
irreducible. If both $\liekpm$ are in this case, then we have example \textbf{R}.5($m=1$) and
\textbf{R}.6.

If both $\liekpm$ are in Case $(3)$, then we have example \textbf{R}.7 and \textbf{R}.8. The
special case where $(\G_1,\gH_1) = (\U(1), \set{1})$ and $\gH_0 = \U(1)$ will be discussed in Step
II and it gives us example \textbf{R}.11, \textbf{R}.12, \textbf{R}.22 and \textbf{R}.23.

If $\liekpm$ are in Case (1) and (3), then the diagram is the sum action on a sphere. For other
cases, we have example \textbf{R}.15($m=1$), \textbf{R}.16, \textbf{R}.17($m=1$) and \textbf{R}.18.
\smallskip

\textsc{Case II.B.} Next we consider the case when $\lieg_1$ is primitive. We claim that any two of
$\chi_i$'s are not equivalent. If not, then we have two different cases. First if $\chi_1$ is
equivalent to $\chi_2$, then $(\lieg_1, \lieh_1\oplus\lieh_0)$ is either $(\lieso(8), \liegt)$ or
$(\lieso(7), \lieu(3))$. In the first case, $\lieg = \lieso(8)\oplus \liegt$ and $\lieh =
\Diag\liegt$. However there is no intermediate subalgebra $\liek$ such that $(\liek, \lieh)$ is a
spherical pair. In the second case, $\lieg=\lieso(7)\oplus \liesu(3)$, $\lieh = \lieu(1)\oplus
\Diag\liesu(3)$ and there is no intermediate Lie algebra $\liek$ such that $(\liek,\lieh)$ is a
spherical pair. Secondly if $\chi_3$ is equivalent to one of $\chi_1$ and $\chi_2$, say $\chi_2$,
then $\liel = \lieh_1\oplus\lieh_0\oplus\lieh_0$ is an intermediate algebra between $\lieg_1$ and
$\lieh_1\oplus \lieh_0$. The Lie algebra $\lieh_0$ embeds diagonally into $\liel$, $(\lieg_1,
\liel)$ is a strongly isotropy irreducible pair and the isotropy representation
$\ad_{\lieg_1/\liel}$ remains irreducible when restricted to $\lieh_1\oplus \lieh_0$. However there
is no such pair $(\lieg_1, \liel)$ that satisfies these properties.

Now we know that the $\liep_i$'s are pairwisely non-equivalent. There are $6$ different
possibilities for the intermediate subalgebra $\liek$:
\begin{center}
\begin{tabular}{lll}
(II.B.1) $\lieh_1\oplus \Diag \lieh_0 \oplus \liep_1\oplus \liep_2$; & (II.B.2) $\lieh_1\oplus
\lieh_0 \oplus
\liep_1 \oplus \lieh_0$; & (II.B.3) $\lieh_1\oplus \lieh_0 \oplus \liep_2 \oplus \lieh_0$; \\
(II.B.4) $\lieh_1 \oplus \Diag \lieh_0 \oplus \liep_1$; & (II.B.5) $\lieh_1 \oplus \Diag \lieh_0
\oplus \liep_2$; & (II.B.6) $\lieh_1 \oplus \lieh_0 \oplus \lieh_0$.
\end{tabular}
\end{center}

If $\liek$ is in Case (II.B.1), then $\liel = \lieh_1 \oplus \liep_1\oplus \liep_2$ is a subalgebra
of $\lieg_1$ and then $\lieg_1 = \liel \oplus \lieh_0$ is not primitive.

If $\liek$ is in Case (II.B.2), then let $\liel = \lieh_1\oplus \lieh_0\oplus \liep_1$ which is a
Lie subalgebra of $\lieg_1$ and $(\lieg_1, \liel)$ is a strongly isotropy irreducible pair. If
$\liel$ is primitive, then it is either $\liesu(n+1)$ or $\liesp(n+1)$($n\geq 1$) since
$(\liek=\liel\oplus \lieh_0, \lieh=\lieh_1\oplus\Diag\lieh_0)$ is a spherical pair for which
$\ad_{\liek/\lieh}$ has two irreducible summands. If $\liel = \liesu(n+1)$, then $\lieh_1 =
\liesu(n)$ and $\lieh_0 = \lieu(1)$. Moreover $(\lieg_1, \liesu(n+1))$ is strongly isotropy
irreducible and its isotropy representation remains irreducible when restricted to $\lieu(n)$. It
follows that it is one of
\begin{equation*}
\lieu(3)\subset \lieso(6) \subset \lieso(7), \quad \lieu(7) \subset \liesu(8) \subset \liee_7.
\end{equation*}
However for each triple above, $\ad_{\lieg_1/\lieh_1}$ also has $3$ irreducible summands, i.e., the
non-reducible action by $\G_1$ also has $s=3$. If $\liel = \liesp(n+1)$, then $\lieh_1 =
\liesp(n)$, $\lieh_0 = \liesp(1)$ and the isotropy representation of $(\lieg_1, \liesp(n+1))$
remains irreducible when restricted to $\liesp(n)\oplus\liesp(1)$. However such $\lieg_1$ does not
exist.

If $\liel$ is not primitive, then the effective version of the spherical pair $(\liel \oplus
\lieh_0, \lieh_1\oplus\Diag \lieh_0)$ is either $(\liesp(n+1), \liesp(n)\oplus \Diag \liesp(1))$
with $\lieh_0 = \liesp(1)$ or $(\lieu(n+1), \lieu(n))$ with $\lieh_0 = \lieu(1)$. In the first
case, we have that $\liel = \liesp(n+1)\oplus\liel_0$ for some nonzero Lie algebra $\liel_0$ and
the isotropy representation of $(\lieg_1, \liesp(n+1)\oplus \liel_0)$ remains irreducible when
restricted to $\liesp(n)\oplus\liesp(1)\oplus \liel_0$. However such $\lieg_1$ does not exist.

In the second case, we have that $\liel = \liesu(n+1)\oplus\liel_0$ for some nonzero Lie algebra
$\liel_0$ and the isotropy representation of $(\lieg_1, \liesu(n+1)\oplus \liel_0)$ remains
irreducible when restricted to $\lieu(n)\oplus \liel_0$. It follows that $\lieu(n)\oplus \liel_0
\subset \liesu(n+1)\oplus \liel_0 \subset \lieg_1$ is one of the following triples:
\begin{center}
\begin{tabular}{ll}
$\lieu(1)\oplus \liesp(n) \subset \liesp(1)\oplus \liesp(n) \subset \liesp(n+1)$, &
$\lieu(2)\subset \lieso(4) \subset \liegt$, \\
$\lieu(1)\oplus\liesp(3) \subset \liesp(1)\oplus\liesp(3) \subset \lief_4$, & $\lieu(1)\oplus
\liesu(6) \subset \liesu(2)\oplus \liesu(6) \subset \liee_6$, \\
$\lieu(5)\oplus \liesu(2) \subset \liesu(6) \oplus \liesu(2) \subset \liee_6$, & $\lieu(1)\oplus
\lieso(12) \subset \liesp(1) \oplus \lieso(12) \subset \liee_7$, \\
$\lieu(1)\oplus \liee_7 \subset \liesu(2)\oplus \liee_7 \subset \liee_8$, &
\end{tabular}
\end{center}
and the corresponding triples $\lieh_1\oplus\Diag \lieh_0 \subset \liek \subset
\lieg_1\oplus\lieh_0$ are
\begin{enumerate}
\item $\liesp(n)\oplus \Diag \lieu(1) \subset \liesp(n)\oplus \liesp(1)\oplus\lieu(1) \subset
\liesp(n+1)\oplus \lieu(1)$,
\item $\lieu(2)\subset \liesu(2)\oplus\liesu(2)\oplus\lieu(1) \subset \liegt\oplus\lieu(1)$,
\item $\liesp(3)\oplus \Diag \lieu(1) \subset \liesp(3)\oplus\liesp(1)\oplus\lieu(1) \subset \lief_4\oplus \lieu(1)$,
\item $\liesu(6)\oplus \Diag\lieu(1) \subset \liesu(6)\oplus \liesu(2)\oplus \lieu(1) \subset \liee_6\oplus \lieu(1)$,
\item $\liesu(2)\oplus \lieu(5) \subset \liesu(2) \oplus \liesu(6)\oplus\lieu(1) \subset
\liee_6\oplus\lieu(1)$,
\item $\lieso(12)\oplus \Diag \lieu(1) \subset \lieso(12) \oplus \liesp(1)\oplus\lieu(1) \subset \liee_7\oplus\lieu(1)$,
\item $\liee_7\oplus \Diag\lieu(1) \subset \liee_7\oplus \liesu(2)\oplus\lieu(1) \subset
\liee_8\oplus\lieu(1)$.
\end{enumerate}
In Case $(2)$ above, there is no corresponding group triple. The non-reducible version of the
triple in Case $(5)$ is $\liesu(2) \oplus\liesu(5) \subset \liesu(2)\oplus\liesu(6) \subset
\liee_6$ and its isotropy representation also has $3$ irreducible summands. The group triples $\gH
\subset \K \subset \G$ of the remaining cases are
\begin{enumerate}
\item $\Sp(n)\Diag \U(1) \subset \Sp(n)\times \Sp(1)\times \U(1) \subset \Sp(n+1)\times \U(1)$ with $n\geq
1$;
\item $\Sp(3)\Diag \U(1) \subset \Sp(3)\times \Sp(1)\times \U(1) \subset \F_4\times \U(1)$;
\item $\SU(6)\Diag \U(1) \subset \SU(6)\times \SU(2)\times \U(1) \subset \E_6 \times \U(1)$;
\item $\Spin(12)\Diag \U(1) \subset \Spin(12)\times \SU(2)\times \U(1) \subset \E_7 \times \U(1)$;
\item $\E_7 \Diag \U(1) \subset \E_7 \times \SU(2) \times \U(1) \subset \E_8\times \U(1)$.
\end{enumerate}

The discussion in Case (II.B.3) is similar to Case (II.B.2).

If $\liek$ is in Case (II.B.4), then $\lieh_1\oplus \liep_1$ is a Lie algebra and $(\lieh_1\oplus
\liep_1, \lieh_1)$ is a spherical pair with irreducible isotropy representation. Furthermore, the
pair $(\lieg_1, \lieh_1\oplus \liep_1\oplus\lieh_0)$ is strongly isotropy irreducible and its
isotropy representation remains irreducible when restricted to $\lieh_1\oplus \lieh_0$. First we
have the following possibilities of $(\lieg_1, \lieh_1\oplus \liep_1 \oplus\lieh_0)$ for which the
pair is strongly isotropy irreducible and the subalgebra is not primitive:
\begin{center}
\begin{tabular}{rlrl}
(1) &  $(\liesu(p+q), S(\lieu(p)\oplus \lieu(q)))$ with $p, q\geq 1$, & (2) & $(\lieso(p+q),
\lieso(p)\oplus \lieso(q))$ with $p, q\geq 1$, \\
(3) & $(\liesp(p+q), \liesp(p)\oplus \liesp(q))$ with $p, q\geq 1$, & (4) & $(\liesp(n), \lieu(n))$
with $n \geq 1$ \\
(5) & $(\lieso(2n), \lieu(n))$ with $n \geq 3$, & (6) & $(\liegt, \lieso(4))$, \\
(7) & $(\lief_4, \liesp(3)\oplus \liesp(1))$, & (8) & $(\liee_6, \lieso(10)\oplus \lieso(2))$, \\
(9) & $(\liee_6, \liesu(6)\oplus \liesu(2))$, & (10) & $(\liee_7, \liee_6\oplus \lieso(2))$, \\
(11) & $(\liee_7, \lieso(12)\oplus \liesu(2))$, & (12) & $(\liee_8, \liee_7\oplus \liesu(2))$, \\
(13) & $(\liesu(4), \lieso(4))$, & (14) & $(\liesu(pq), \liesu(p)\oplus \liesu(q))$ with $p,q\geq
2, pq \geq 5$, \\
(15) & $(\lief_4, \liegt\oplus \lieso(3))$, & (16) & $(\lief_4, \liesu(3)\oplus \liesu(3))$, \\
(17) & $(\liee_6, \liesu(3) \oplus \liegt)$, & (18) & $(\liee_6, \liesu(3)\oplus \liesu(3)\oplus
\liesu(3))$, \\
(19) & $(\liee_7, \liesp(3) \oplus \liegt)$, & (20) & $(\liee_7, \lieso(3)\oplus \lief_4)$, \\
(21) & $(\liee_7, \liesu(3)\oplus \liesu(6))$, & (22) & $(\liee_8, \liesu(3)\oplus \liee_6)$, \\
(23) & $(\liee_8, \liegt\oplus \lief_4)$, & (24) & $(\lieso(4n),\liesp(1)\oplus \liesp(n))$ with $n
\geq 2$, \\
(25) & $(\liesp(n), \liesp(1)\oplus\lieso(n))$ with $n \geq 3$.
\end{tabular}
\end{center}
The first $13$ cases are from the symmetric spaces and the rest are from Wolf's list. Next we
consider the triples $\lieh_1 \oplus\lieh_0 \subset \lieh_1\oplus\liep_1\oplus \lieh_0 \subset
\lieg_1$ and they are the followings:
\begin{enumerate}
\item $\liesu(p)\oplus \liesu(q) \subset S(\lieu(p)\oplus\lieu(q)) \subset \liesu(p+q)$ with $p, q \geq
1, pq \geq 2$,
\item $\lieso(n)\oplus \liegt \subset \lieso(n)\oplus \lieso(7) \subset \lieso(n+7)$ with $n\geq
2$,
\item $\liesp(n)\oplus \lieu(1) \subset \liesp(n)\oplus \liesp(1) \subset \liesp(n+1)$ with $n \geq
1$,
\item $\liesp(n)\oplus \lieso(4) \subset \liesp(n)\oplus\liesp(2) \subset \liesp(n+2)$ with $n \geq
1$,
\item $\liesu(n) \subset \lieu(n) \subset \liesp(n)$ with $n \geq 1$,
\item $\lieu(1)\oplus \lieso(5) \subset \lieu(4) \subset \liesp(4)$,
\item $\liesu(n) \subset \lieu(n) \subset \lieso(2n)$ with $n \geq 3$,
\item $\lieu(1)\oplus \liesp(3) \subset \liesp(1)\oplus \liesp(3) \subset \lief_4$,
\item $\lieso(10) \subset \lieso(10)\oplus \lieso(2) \subset \liee_6$,
\item $\lieso(9)\oplus \lieso(2) \subset \lieso(10) \oplus \lieso(2) \subset \liee_6$,
\item $\lieu(1) \oplus \liesu(6) \subset \liesu(2) \oplus \liesu(6) \subset \liee_6$,
\item $\liee_6 \subset \liee_6\oplus \lieso(2) \subset \liee_7$,
\item $\lieu(1)\oplus \lieso(12) \subset \liesu(2)\oplus \lieso(12) \subset \liee_7$,
\item $\lieso(11)\oplus \liesu(2) \subset \lieso(12) \oplus \liesu(2) \subset \liee_7$,
\item $\lieu(1)\oplus \liee_7 \subset \liesu(2)\oplus \liee_7 \subset \liee_8$,
\item $\lieso(5) \oplus\lieu(n) \subset S(\lieu(4)\oplus \lieu(n)) \subset \liesu(n+4)$ with $n \geq
1$,
\item $\lieu(2) \subset \lieso(4) \subset \liegt$.
\end{enumerate}
There is no corresponding group triples for the last two cases $(16)$ and $(17)$. In Case $(2)$ the
group triple is
\begin{equation*}
\gH = \Spin(n) \times \Gt \subset \K = \wt{\SO(n)\times \SO(7)} \subset \Spin(n+7).
\end{equation*}
If $n\geq 3$, then $\K$ is not simply-connected and its two-fold cover is $\Spin(n) \times
\Spin(7)$. It follows that $\K/\gH$ is the real projective space $\Rp^7$. For other cases, we list
the group triples $\gH \subset \K \subset \G$ below:
\begin{enumerate}
\item $\SU(p)\times \Diag\SU(q) \subset \U(1)\cdot\SU(p)\times  \Diag\SU(q) \subset \SU(p+q)\times \SU(q)$,
\item $\Gt \times \Diag \SO(2) \subset \Spin(7)\times\Diag \SO(2) \subset \Spin(9) \times \SO(2)$,
\item $\U(1)\times \Diag \Sp(n) \subset \Sp(1)\times \Diag\Sp(n) \subset \Sp(n+1)\times \Sp(n)$,
\item $\Sp(1)\times \Sp(1) \times \Diag \Sp(n) \subset \Sp(2) \times \Diag \Sp(n) \subset \Sp(n+2) \times
\Sp(n)$,
\item $\Diag \SU(n) \subset \U(1) \times \Diag \SU(n) \subset \Sp(n) \times \SU(n)$,
\item $\Sp(2) \times \Diag \U(1) \subset \SU(4) \times \Diag \U(1) \subset \Sp(4) \times \U(1)$,
\item $\Diag \SU(n) \subset \U(1) \times \Diag \SU(n) \subset \SO(2n)\times \SU(n)$,
\item $\U(1)\times \Diag \Sp(3) \subset \Sp(1)\times \Diag \Sp(3) \subset \F_4\times \Sp(3)$,
\item $\Diag \Spin(10) \subset \SO(2) \times \Diag \Spin(10) \subset \E_6 \times \Spin(10)$,
\item $\Spin(9) \times \Diag \SO(2) \subset \Spin(10) \times \Diag \SO(2) \subset \E_6\times
\SO(2)$,
\item $\U(1)\times \Diag\SU(6) \subset \SU(2)\times \Diag \SU(6) \subset \E_6\times \SU(6)$,
\item $\Diag \E_6 \subset \SO(2) \times \Diag \E_6 \subset \E_7\times \E_6$,
\item $\U(1) \times \Diag \Spin(12) \subset \SU(2) \times \Diag \Spin(12) \subset \E_7 \times
\Spin(12)$,
\item $\Spin(11) \times \Diag \SU(2) \subset \Spin(12) \times \Diag \SU(2) \subset \E_7\times
\SU(2)$,
\item $\U(1) \times \Diag \E_7 \subset \SU(2) \times \Diag \E_7 \subset \E_8\times \E_7$.
\end{enumerate}

The discussion in Case (II.B.5) is similar to Case (II.B.4).

If $\liek$ is in Case (II.B.6), then $\lieh_0$ is either $\lieu(1)$ or $\liesu(2)$ and the isotropy
representation of the pair $(\lieg_1, \lieh_1\oplus \lieh_0)$ has two irreducible summands. There
are many examples in this case.

We summarize the group triples in Case (II.B.2) and (II.B.4) in Table \ref{tabletriplereducible}.
\begin{table}[!h]
\begin{center}
\begin{tabular}{|l|l|l|l|}
\hline $\quad \quad \quad \quad \quad \G$ & $\quad \quad \quad \quad \K$ & $\quad \quad \quad \gH$ &  \\
\hline \hline
$\Sp(n+1) \times \U(1)$ & $\Sp(n) \times \Sp(1)\times \U(1)$ & $\Sp(n)\Diag \U(1)$ & $n\geq 1$ \\
\hline
$\F_4 \times \U(1)$ & $\Sp(3)\times \Sp(1)\times \U(1)$ & $\Sp(3) \Diag \U(1)$ & \\
\hline
$\E_6 \times \U(1)$ & $\SU(6) \times \SU(2) \times \U(1)$ & $\SU(6) \Diag \U(1)$ &  \\
\hline
$\E_7 \times \U(1)$ & $\Spin(12)\times \SU(2)\times \U(1)$ & $\Spin(12) \Diag \U(1)$ & \\
\hline
$\E_8\times \U(1)$ & $\E_7\times \SU(2)\times \U(1)$ & $\E_7\Diag \U(1)$ & \\
\hline \hline
$\Spin(9) \times \U(1)$ & $\Spin(7) \times \Diag \U(1)$ & $\Gt\Diag \U(1)$ & \\
\hline
$\Sp(4)\times \U(1)$ & $\SU(4)\times \Diag \U(1)$ & $\Sp(2) \Diag \U(1)$ & \\
\hline
$\E_6\times \U(1)$ & $\Spin(10) \times \Diag \U(1)$ & $\Spin(9) \Diag \U(1)$ & \\
\hline $\SU(p+q)\times \SU(q)$ & $\U(1)\cdot\SU(p) \times \Diag \SU(q)$ & $
\SU(p)\times \Diag\SU(q)$ & $p\geq 1, q\geq 2$ \\
\hline
$\Sp(n+1)\times \Sp(n)$ & $\Sp(1) \times \Diag \Sp(n)$ & $\U(1) \Diag \Sp(n)$ & $n\geq 1$ \\
\hline
$\Sp(n+2)\times \Sp(n)$ & $\Sp(2)\times \Diag \Sp(n)$ & $\Sp(1) \times \Sp(1)\Diag \Sp(n)$ & $n\geq 1$ \\
\hline
$\Sp(n) \times \SU(n)$ & $\U(1) \times \Diag \SU(n)$ & $\Diag \SU(n)$ & $n\geq 2$ \\
\hline
$\E_7\times \SU(2)$ & $\Spin(12)\times \Diag \SU(2)$ & $\Spin(11)\Diag \SU(2)$ & \\
\hline
$\SO(2n)\times \SU(n)$ & $\U(1)\times \Diag \SU(n)$ & $\Diag \SU(n)$ & $n\geq 3$ \\
\hline
$\F_4\times \Sp(3)$ & $\Sp(1) \times \Diag \Sp(3)$ & $\U(1)\Diag \Sp(3)$ & \\
\hline
$\E_6 \times \Spin(10)$ & $\U(1) \times \Diag \Spin(10)$ & $\Diag \Spin(10)$ & \\
\hline
$\E_6\times \SU(6)$ & $\SU(2)\times \Diag \SU(6)$ & $\U(1) \Diag \SU(6)$ & \\
\hline
$\E_7 \times \E_6$ & $\U(1) \times \Diag \E_6$ & $\Diag \E_6$ & \\
\hline
$\E_7\times \Spin(12)$ & $\SU(2) \times \Diag \Spin(12)$ & $\U(1) \Diag \Spin(12)$ & \\
\hline
$\E_8\times \E_7$ & $\SU(2) \times \Diag \E_7$ & $\U(1) \Diag \E_7$ & \\
\hline
\end{tabular}
\smallskip
\caption{Group triple $\G \supset \K \supset \gH$ in Case II.B.2 and
II.B.4}\label{tabletriplereducible}
\end{center}
\end{table}

We consider the construction of the cohomogeneity one diagram.

If none of the triples $\gH \subset \Kpm \subset \G$ is in Case (II.B.6), then from the fact that
the three summands are pairwisely non-equivalent, each of them should be in Table
\ref{tabletriplereducible}. It is easy to see that in Table \ref{tabletriplereducible}, for any
given pair $(\G, \gH)$, there is only one intermediate subgroup $\K$. This gives us example
\textbf{R}.13($m=1$).

If $\gH \subset \Kp \subset \G$ is in Case (II.B.6), but $\gH \subset \Km \subset \G$ is not in
Case (II.B.6), then we have example \textbf{R}.19, \textbf{R}.20($m=1$) and \textbf{R}.21.

If both $\Kpm$ are in Case (II.B.6), then we have example \textbf{R}.9($m=1$) and \textbf{R}.10.

\smallskip

\textsc{Step II.} First note that other than example \textbf{R}.11, \textbf{R}.12, \textbf{R}.22
and \textbf{R}.23, the three summands of the isotropy representation $\Ad_{\gH_c}$ of the examples
in Table \ref{Tables3reducibledouble} and \ref{Tables3reduciblenondouble} are non-equivalent. From
Step I, these four examples have $(\G, \gH)=(\U(1) \times \G_2\times \U(1), \gH_2\Delta \U(1))$
where $\U(1)$ is diagonally embedded into the last two factors of $\G$, and $\G_2/(\gH_2\times
\U(1))$ is strongly isotropy irreducible. The simplest case is when $\G_2 = \SU(2)$ and
$\gH=\Delta\U(1)$. The manifold is 5 dimensional and the isotropy representation is $[\phi]_\Real
\oplus \id \oplus \id$.

It is easy to see that the action is not a product(see Section 1.5.1 in \cite{HoelscherClass})
since the manifold is simply-connected. In dimension $5$, if the action is not a product or a sum
action, or fixed points free, then the non-reducible diagrams have $\G=\SU(2)\times \U(1)$ and
$\gH_c = \set{1}$. Such diagrams were classified in \cite{HoelscherClass} and most of them have
disconnected $\gH$. If a $\U(1) \subset \SU(2)$ normalizes $\gH$ and $\Kpm$, then one can extend
the diagram $\gH \subset \set{\Kpm} \subset \U(1)\times \SU(2)$ to a reducible one with $\G = \U(1)
\times \SU(2) \times \U(1)$ and $\gH_c = \Delta \U(1)$ via normal extension. In example $N^5$,
$Q^5_A$ and $Q^5_C$ in \cite{HoelscherClass}, one can take $\gL = \U(1) = \set{e^{\qi \theta}}$ to
extend the action to $\U(1) \times \SU(2)\times \U(1)$. Note that example $Q^5_C$ is a primitive
action and it is a sum action on $\sph^5$. However such extension does not exist for the example
$P^5$ and $Q^5_B$.

Next we consider other examples, i.e. $\gH_2$ is not a trivial group. The non-reducible diagrams
have $\G = \U(1) \times \G_2$ and $\gH = \gH_2$. If the isotropy representation of $\G_2/\gH_2$ has
two irreducible summands, then the non-reducible diagram by $\U(1) \times \G_2$ also has $s=3$. So
we only consider the pairs $(\G_2, \gH_2)$ where the isotropy representation $\G_2/\gH_2$ has more
than two irreducible summands, and they are given by $(\SO(n+2),\SO(n))(n\geq 2)$ and their finite
covers. The connected components of $\Kpm$ are given by $\U(1) \cdot \gH_2$ and the embeddings of
the $\U(1)$ factor into $\U(1)\times \G_2$ are different. The proper subgroup $\gL = \U(1) \times
\SO(2) \times \SO(n) \subset \G$ contains both $\Kpm_c$. We may assume that $\G_2$ is
simply-connected by lifting the action to its universal covering if necessary. The diagram in this
case was discussed in \cite{HoelscherClass}, see Lemma 4.3. There are two different classes of
cohomogeneity one diagrams. In one class, example \textbf{R}.12, the diagrams are doubles with
disconnected $\gH$ and the manifolds are $\sph^2$ bundles over the Stiefel manifold
$\SO(n+2)/\SO(n)$. In another class, example \textbf{R}.23, the diagrams are non-primitive and the
manifolds are bundles over $\SO(n+2)/(\SO(2)\times \SO(n))$ with fiber a three dimensional lens
space. In both classes, since $\U(1) \times \SO(2)$ normalizes $\Kpm$ and $\gH$, one can extend the
actions by $\U(1)\times \SO(n+2)$ to reducible actions by $\U(1) \times \SO(n+2) \times \U(1)$ such
that the principal isotropy representation has three irreducible summands.

We illustrate the construction in this case by the pair $(\G_2, \gH_2) = (\SO(5), \SO(3))$.

Suppose $\set{\beta(\theta) = e^{2\pi \imath \theta}| 0 \leq \theta \leq 1 }$ is the circle group
$\SO(2) \subset \SO(2) \times \SO(3) \subset \SO(5)$ and $\set{\alpha(t) = e^{2\pi \imath t} | 0
\leq t \leq 1}$ is the $\SO(2)$ factor in $\G = \SO(2) \times \SO(5)$. Let
\begin{equation*}
\Kpm_c = \left\{ \left( \alpha(n_{\pm} \theta),
\begin{pmatrix}
\beta(m_{\pm}\theta) & \\
& A
\end{pmatrix} \right)
|\, 0 \leq \theta \leq 1, A \in \SO(3) \right\} \subset \SO(2) \times \SO(2) \times \SO(3)
\end{equation*}
be the identity components of $\Kpm$ where $m_\pm$ and $n_\pm$ are integers. To obtain a diagram
which defines a double, let $n_\pm = 1$ and $m_\pm = m$, and for any integer $k$ let $\Zeit_k
\subset \set{(\alpha(\theta),\beta(m \theta))}$ be a cyclic group. Then a double has the diagram
$\gH = \Zeit_k \cdot \SO(3) \subset \set{\Kpm = \Kpm_c} \subset \SO(2) \times \SO(5)$.

We consider the diagram in the second class which is not a double. To simplify the discussion, we
assume that $n_\pm = 1$ and then $m_+ \ne m_-$. Suppose $\Zeit_{m_{\pm}}$ is the cyclic group
generated by $(\alpha(1/m_\pm), I_5)$ and $\Zeit_m$ by $(\alpha(1/m),I_5)$ where $m$ is the least
common multiple of $m_\pm$. Let $\gH_\pm = \Zeit_{m_\pm} \cdot \SO(3)$ and
\begin{equation*}
\gH = \gH_+ \cdot \gH_- = \Zeit_m \cdot \SO(3), \quad \Kpm = \Kpm_c \cdot \gH = \Kpm_c \cdot
\Zeit_{m/m_\pm},
\end{equation*}
then the diagram $\gH \subset \set{\Kpm} \subset \G = \SO(2) \times \SO(5)$ defines a
simply-connected cohomogeneity one manifold $M$. Since both $\Kpm$ is contained in $\gL = \SO(2)
\times \SO(2) \times \SO(3)$, $M$ is a fiber bundle over the space $\G/\gL = \SO(5)/(\SO(2)\times
\SO(3))$ and the $\gL$ action on the fiber is also cohomogeneity one. The action is not effective
and the effective one has the diagram
\begin{equation*}
\set{1} \subset \set{\gS^1_{1,m_-}, \gS^1_{1,m_+}} \subset \T^2,
\end{equation*}
where $\gS^1_{p,q}$ is embedded in $\T^2$ as $(e^{2\pi \imath p \theta}, e^{2\pi \imath q
\theta})(0 \leq \theta \leq 1)$. Using the van Kampen Theorem, the fibre is a lens space with
fundamental group $\pi_1 = \Zeit_{|m_+ - m_-|}$, see \cite{Neumann} and Proposition 1.8 in
\cite{HoelscherClass}.

In the following let $\gH \subset \set{\Kpm} \subset \G$ be a diagram in Table
\ref{Tables3reducibledouble} and \ref{Tables3reduciblenondouble}, and we assume that the three
summands are non-equivalent. It follows that one cannot obtain a new diagram by conjugating the
original one by an element in $N_\G(\gH)$. If one singular orbit is codimension $2$, say $\Km /\gH
= \sph^1$, then one may add connected components to isotropy subgroups. In the case where the
diagram is a double, i.e., $\Km = \Kp$, then we have the diagram $\gH \cdot \Zeit_m \subset
\set{\Kpm} \subset \G$ for every $m \in \Zeit$. Note that in example \textbf{R}.8, if one add
connected components, then the action is not effective and its effective version is the original
one with connected $\gH$. If the diagram $\Gamma$ is not a double, then it is one of the example
\textbf{R}.14, \textbf{R}.15, \textbf{R}.17, \textbf{R}.19 and \textbf{R}.20. Except for example
\textbf{R}.15 and \textbf{R}.19, there exists a proper subgroup $\gL$ contains $\Kpm$ and the
diagram $\gH \subset \set{\Kpm} \subset \gL$ is a sum action on a sphere. For each $m \in \Zeit$ we
have the diagram $\gH \cdot \Zeit_m \subset \set{\Km, \Kp \cdot \Zeit_m} \subset \G$ and $\Kp\cdot
\Zeit_m$ is contained in $\gL$. In example \textbf{R}.15 and \textbf{R}.19, $\Kp$ contains $\Km$
which implies that one can not add components to the isotropy subgroups. This finishes the proof of
Theorem \ref{thmclassificationreducible}.
\end{proof}

\subsection{Non-primitive actions}

We assume that the diagram is non-reducible and the main result is

\begin{thm}\label{thms3nonprimitive}
Suppose a compact simply-connected manifold $M$ admits a cohomogeneity one action with diagram $\gH
\subset \set{\Kpm} \subset \G$ and $s=3$. If the action is non-primitive and non-reducible, then
\begin{enumerate}
\item either the manifold is a double, i.e., $\Km = \Kp$,
\item or it is equivariantly diffeomorphic to one of the examples in Table \ref{Tables3nonprimitive}.
\end{enumerate}
\end{thm}

\begin{proof}
\textsc{Step I:} We assume that $\gH$ is connected. Let $\gH \subset \set{\Kpm} \subset \G$ be a
non-primitive diagram and $\gL$ be the minimal subgroup of $\G$ which contains both $\Kpm$. Let
$D_1$ denote the diagram $\gH \subset \set{\Kpm} \subset \gL$. If $\Kpm$ are the same, then $\gL=
\Kpm$ and the manifold is a double. In the following, we assume that at least one of $\Kpm$ is
proper in $\gL$.

\textsc{Case A.} We consider the case when $\gL$ is one of $\Kpm$, say $\gL = \Kp$ and $\Km
\subsetneq \gL$. So the diagram $D_1$ has only one fixed point and $s=2$ and $\G/\gL$ is strongly
isotropy irreducible. Suppose $\gL = \gL_1\times \gL_2$ and $\gL_1$ is the non-effective kernel of
the action given by $D_1$. Then the effective version of $D_1$ is $\gH_1 \subset \set{\K_1, \gL_2}
\subset \gL_2$. From the classification of the actions with one fixed point in Proposition
\ref{propfixedpointaction}, $\gL_1 \times \gH_1 \subset \gL_1 \times \gL_2 \subset \G$ is one of
the following triples:
\begin{enumerate}
\item $\gL_1 \times \SU(n) \subset \gL_1\times \SU(n+1) \subset \G$,
\item $\gL_1 \times \U(n) \subset \gL_1 \times \U(n+1) \subset \G$,
\item $\gL_1 \times \Spin(7) \subset \gL_1 \times \Spin(9) \subset \G$,
\item $\gL_1 \times \Sp(n) \Delta \Sp(1) \subset \gL_1 \times \Sp(n+1) \times \Sp(1) \subset \G$.
\end{enumerate}
Since the diagram of $M$ is non-reducible, any primitive factor in $\gL_1$ is not in the
non-effective kernel of the homogeneous space $\G/\gL$.

In the first case, the homogeneous space $\G/\gL$ is effective and then the isotropy representation
of $\G/(\gL_1\times \gH_1)$ has $3$ summands. Combining the classification in Table
\ref{tabletripleHKG} for the triple $\G \supset \gL_1\times \gL_2 \supset \gL_1 \times \gH_1$, we
have the following two possibilities:
\begin{enumerate}
\item $\SU(2) \times \SU(5) \subset \SU(2) \times \SU(6) \subset \E_6$,
\item $\SO(n) \times \SU(3) \subset \SO(n) \times \SU(4) \subset \Spin(6+n)$ with $n\geq 1$,
\end{enumerate}
Thus we have example \textbf{N}.1 and \textbf{N}.2. The manifolds are $\Cp^6$ bundle over $\E_6
/(\SU(6) \cdot \SU(2))$ and $\Cp^4$ bundle over $\SO(n+6)/(\SO(6)\times \SO(n))$ respectively.

In the second case, some primitive factor of $\gL_2$ but not the whole $\gL_2$ is the non-effective
kernel of the homogeneous space $\G/(\gL_1 \times \gL_2)$. Then we have example \textbf{N}.3 and
\textbf{N}.4 and the manifolds are $\Hp^{n+1}$ bundle over $\G_1/(\gL_1 \times \Sp(1))$ and
$\Cp^{n+1}$ bundle over $\G_1/(\gL_1\times \U(1))$ respectively.

In the last case, $\gL_2$ is the non-effective kernel of the homogeneous space $\G/(\gL_1\times
\gL_2)$. Then we have example \textbf{N}.5 and the manifold is the product of $\G_1/\gL_1$ with the
one defined by the diagram $D_1$.

\smallskip

\textsc{Case B.} Now we assume that both $\Kpm$ are proper subgroups in $\gL$ and then $\G/\gL$ is
isotropy irreducible. The diagram $D_1$ is primitive and has $s=2$. There are three difference
cases for the effective version of this diagram classified in Theorem \ref{thmclasss2primitive}.

\textsc{Case B.1.} Suppose the diagram $D_1$ is given by $\Gt\times \gL_1 \subset \set{\Spin^-(7)
\times \gL_1, \Spin^+(7)\times \gL_1 } \subset \Spin(8) \times \gL_1$. If $\G$ is simple, then $\G
\supset \Spin(8)\times \gL_1 \supset \Spin^-(7)\times \gL_1$ is in Table
\ref{TableHKG2summandsKHsphere} and thus $\G = \Spin(9)$ and $\gL_1 = \set{1}$. However the
isotropy representation of the homogeneous space $\Spin(9)/\Spin(8)$ splits when restricted to
$\Gt$. If $\G$ is not a simple Lie group, then we have example \textbf{N}.6 and the manifold is the
product $\sph^{15}\times \G_1/\gL_1$.

\textsc{Case B.2.} Suppose the diagram $D_1$ is given by $\set{1}\times \gL_1 \subset
\set{\U(1)_1\cdot \gL_1, \U(1)_2 \cdot \gL_1} \subset \U(1)\times \U(1)\times \gL_1$. If the
non-effective kernel of the homogeneous space $\G/\gL$ is $\U(1)\times \U(1)$, then we have a
special case of example \textbf{N}.7 for which $\gL_1 = \gL_2 = \U(1)$.

The other possibility is that $\G = \G_1 \times\U(1)$ such that $\G_1$ is simple, $\G_1/(\U(1)\cdot
\gL_1)$ is strongly isotropy irreducible and its isotropy representation remains irreducible when
restricted to $\gL_1$. It follows that the pair $(\G_1, \U(1)\cdot\gL_1)$ appears as $(\G, \K)$ in
Table \ref{TableHKG2summandsKHsphere}. Thus we have example \textbf{N}.8 with $(\gL_1, \gH_1) =
(\U(1), \set{1})$ and \textbf{N}.10 with connected $\gH'$.

\textsc{Case B.3.} Suppose the diagram $D_1$ is $\gH_1\times \gH_2 \times \gL_0 \subset \set{\gL_1
\times \gH_2\times \gL_0, \gH_1\times \gL_2\times \gL_0} \subset \gL_1\times \gL_2\times \gL_0$ and
$\gL_i/\gH_i$($i=1,2$) is a sphere with irreducible isotropy representation. If some primitive
factor $\gL'$ of $\gL$ diagonally embeds into $\G$, then $\G = \gL' \times \gL$. Note $\gL'$ is not
a factor of $\gL_0$ otherwise the diagram is reducible. Since the isotropy representation of
$\G/\gL$ remains irreducible when restricted to $\gH$, it follows that one of $\gL_1$ and $\gL_2$,
say $\gL_2 = \SO(2)$, and then $\gL' = \gL_2$. However the manifold is a sphere bundle over
$\sph^1$ which is not simply-connected. So we assume that there is no diagonally embedded factors
in $\G$. Since the diagram is non-reducible, any primitive factor of $\gL_0$ is not contained in
the non-effective kernel of the strongly isotropy irreducible space $\G /\gL$, i.e., the kernel is
a subgroup of $\gL_1\times \gL_2$. If $\gL_1\times \gL_2$ is the kernel, then we have example
\textbf{N}.7 and the manifold is a product of a sphere and the homogeneous space $\G_1/\gL_1$.

Next we assume that the non-effective kernel $\gL'$ is a proper normal subgroup of $\gL_1\times
\gL_2$. Both $\gL_1$ and $\gL_2$ cannot be $\SU(2)\times \SU(2)$(or $\SO(4)$) otherwise the diagram
of $\G$ action is reducible. Hence one of $\gL_1$ and $\gL_2$, say $\gL_1$, is primitive and then
$\gL' = \gL_1$. It follows that $\G = \gL_1 \times \G_2$ and the pair $(\G_2, \gL_2\times \gL_0)$
appears as $(\G,\K)$ in Table \ref{TableHKG2summandsKHsphere}. Conversely, for every triple $\G_2
\supset \K_2 \supset \gH_2$ in Table \ref{TableHKG2summandsKHsphere} and for any isotropy
irreducible spherical pair $(\gL_1, \gH_1)$ with $\gL_1$ primitive, we have the following diagram
\begin{equation*}
\gH_1 \times \gH_2 \subset \set{\gL_1 \times \gH_2, \gH_1 \times \K_2} \subset \gL_1 \times \G_2,
\end{equation*}
and it gives us example \textbf{N}.8 and \textbf{N}.9($m=1$).

\smallskip

\textsc{Step II: } From Step I, if a variation of a double has different singular isotropy groups,
then the new diagram must be in Table \ref{Tables3nonprimitive}. Let $D_c : \gH' \subset \set{\Kpm
= \K'} \subset \G$ be a double with $\K'/\gH' = \sph^1$. Suppose $D: \gH \subset \set{\Kpm} \subset
\G$ is a diagram with disconnected subgroups from the double $D_c$, i.e., $\Kpm_c = \K'$ and $\gH_c
= \gH'$. Let $\gH_\pm = \gH \cap \Kpm_c = \gH\cap\K'$ and then Lemma 1.13 in \cite{HoelscherClass}
tells us that $\gH$ is generated by $\gH_{\pm}$ if the manifold defined by $D$ is simply connected.
In particular it implies that $\gH \subset \K'$ since $\gH_{-} = \gH_{+}$. Since $\Kpm/\gH =
\sph^1$, we have that $\Kpm = \K'$ and the manifold is an $\sph^2$ bundle over $\G/\K'$.

For most diagrams in Table \ref{Tables3nonprimitive}, the three irreducible summands of the
isotropy representation $\Ad_{\gH_c}$ are non-equivalent. It follows that their variations by
conjugating group elements are equivalent to the original ones. In the following four cases, the
isotropy representation $\Ad_{\gH_c}$ contains two equivalent irreducible summands:
\begin{eqnarray}
& & \Gt \times \gL_1 \subset \set{\Spin^-(7)\times \gL_1, \Spin^+(7)\times \gL_1} \subset
\Spin(8)\times \G_1 \label{eqndiagramnonp2} \\
& & \SU(3) \subset \set{\U(1)\cdot \SU(3), \SU(4)} \subset \Spin(7)  \label{eqndiagramnonp1} \\
& & \gH_1 \times \Gt \subset \set{\gL_1 \times \Gt, \gH_1 \times \Spin(7)} \subset \gL_1 \times
\SO(8) \label{eqndiagramnonp3}
\end{eqnarray}
\begin{eqnarray}
& & \gH_1 \subset \set{\U(1) \cdot \gH_1, \K_1} \subset \U(1)\times \G_1. \label{eqndiagramnonp4}
\end{eqnarray}
They are example \textbf{N}.2 with $n=1$, example \textbf{N}.6 where $\G_1/\gL_1$ is strongly
isotropy irreducible, example \textbf{N}.8 for the triple $\SO(8) \supset \Spin(7) \supset \Gt$,
and example \textbf{N}.10 where the triple $\G_1 \supset \K_1 \supset \gH_1$ is in Table
\ref{TableHKG2summandsKHsphere} such that $\K_1/\gH_1 = \sph^1$.

The different diagram from a variation of the first one (\ref{eqndiagramnonp2}) is a double. For
the second diagram (\ref{eqndiagramnonp1}), all automorphisms of $\Spin(7)$ are inner ones, i.e.,
conjugation by group elements. Since for every $g \in N(\SU(3))$, $g.(\U(1)\cdot \SU(3)).g^{-1} =
\U(1)\cdot \SU(3)$, a variation does not give us a new diagram. Since the $\U(1)$ factor in $\Km$
is contained in $\Kp = \SU(4)$, one cannot add connected component to $\gH$. A similar argument
shows that any variation of the third one (\ref{eqndiagramnonp3}) does not give us a new diagram
either. If the pair $(\gL_1, \gH_1)$ is $(\U(1), \set{1})$, then one can add connected components
to $\gH = \Gt$ and $\Km=\Spin(7)$ to obtain a diagram with disconnected principal isotropy
subgroup. However the action of $\G = \U(1) \times \SO(8)$ is not effective and the diagram of the
effective action is the same as the one (\ref{eqndiagramnonp3}). The discussion of the diagram
(\ref{eqndiagramnonp4}) is similar to the example \textbf{R}.23 in the proof of Theorem
\ref{thmclassificationreducible}. Since the two singular orbits are codimension two, one can add
components to the three isotropy subgroups. The proper subgroup $\gL = \U(1) \times \K_1$ contains
both $\Kpm_c$. We can apply Lemma 4.3 in \cite{HoelscherClass} to obtain all diagrams whose
connected groups are in (\ref{eqndiagramnonp4}). It follows that the manifolds are lens space
bundles over $\G_1/\K_1$. The one with the lowest dimension is given by the triple $(\G_1, \K_1,
\gH_1) = (\SO(3) \times \U(1), \SO(2)\times \U(1), \SO(2))$ and the manifold is the product
$\sph^3\times \sph^2$. If $(\G_1, \K_1, \gH_1)$ is the triple $(\SU(3), \U(2), \SU(2))$, then the
manifold is example $N^7_H$ in \cite{HoelscherClass}.

Finally we consider the diagrams by adding connected components to isotropy groups and we assume
that the three summands are non-equivalent. If the diagram is a double and $\Kpm/\gH = \sph^1$,
then for each nonzero integer $m$, we have the diagram $\Zeit_m \cdot \gH \subset \set{\Kpm}
\subset \G$. In example \textbf{N}.7 and \textbf{N}.8, if one add connected components to isotropy
subgroups, then the action by $\G$ is not effective. If the diagram is in example \textbf{N}.1,
\textbf{N}.2, \textbf{N}.4 and \textbf{N}.5, then the diagram $\gH \subset \set{\Kpm} \subset \Kp$
has a fixed point. So one cannot add connected components to the isotropy groups.
\end{proof}

\medskip


\section{Primitive and non-reducible actions with $s=3$ and $\G$ simple}

First we state the main result in this section.
\begin{thm}\label{thms3primitiveGsimple}
Suppose $M$ is a compact simply-connected manifold with a cohomogeneity one action by a simple Lie
group $\G$. If the action is primitive and has $s=3$, then $M$ is equivariantly diffeomorphic to a
sphere, a complex projective space or the Grassmannian $\SO(m+n)/(\SO(m)\times \SO(n))(m,n \geq 2)$
with an isometric action, see Table \ref{Tablespherenotsum} and \ref{Tableprojectivespaces}.
\end{thm}

We saw in Theorem \ref{thmclassificationreducible} that if the action is reducible and primitive,
then the manifold is a sphere with a linear action. So in the following we assume that the action
is non-reducible. We prove the theorem in several steps:

\textsc{Step 1:} If the three summands of the isotropy representation are non-equivalent, then the
manifold is a sphere with a sum action or the two singular orbits $\G/\Kpm$ are strongly isotropy
irreducible, see Proposition \ref{proppairwisenonequ}.

\textsc{Step 2:} We assume that all isotropy subgroups are connected. If one of $\G/\Kpm$, say
$\G/\Km$ is not isotropy irreducible, i.e., the isotropy representation of $\G/\Km$ has two
summands, then the group triple $\G \supset \Km \supset \gH$ must be $\SO(7) \supset \U(3) \subset
\SU(3)$, see Proposition \ref{propGK2summandsnonequiv}.

\textsc{Step 3:} We classify triples of connected groups $\G \supset \K \supset \gH$ such that
$\G/\K$ is isotropy irreducible, $\K/\gH$ is a sphere and the isotropy representation of $\G/\gH$
has three summands, see Proposition \ref{proptripleHKG} and Table \ref{tabletripleHKG}. It follows
that if $M$ is not a sphere, then both triples $\G \supset \Kpm \supset \gH$ are in Table
\ref{tabletripleHKG} or $\SO(7) \supset \U(3) \supset \SU(3)$.

\textsc{Step 4:} For every quadruple $(\G, \K_1, \K_2, \gH)$ from the previous step, we consider
all possible cohomogeneity one group diagrams from it and then identify the manifolds. Theorem
\ref{thmrigiditydisconnected} is the classification for disconnected $\gH$ and Theorem
\ref{thmrigidityconnected} is for connected $\gH$.

\subsection{Three summands are pairwisely non-equivalent.}
We will show that the two singular orbits $\Bpm = \G/\Kpm$ are strongly isotropy irreducible
homogeneous spaces unless the \cohomd $M$ is equivariantly diffeomorphic to a sphere. In the
following, the three irreducible summands of $\Ad_{\gH_c}$ are denoted by $\liep_1$, $\liep_2$ and
$\liep_3$, i.e., $\liep = \liep_1 \oplus \liep_2 \oplus \liep_3$.

\begin{prop}\label{proppairwisenonequ}
If $\liep_1$, $\liep_2$ and $\liep_3$ are pairwisely non-equivalent as the $\Ad_{\gH_c}$
representations, then
\begin{enumerate}
\item either $\G/\Kpm$ are strongly isotropy,
\item or $M$ is a sphere with a sum action.
\end{enumerate}
\end{prop}

\begin{proof}
There are three different types of the group diagram. In the first type, none of $\Kpm/\gH$ are
strongly isotropy irreducible, then both $\G/\Kpm$ are strongly isotropy irreducible.

In the second type, only one of $\Kpm/\gH$, say $\Km/\gH$, is strongly isotropy irreducible. We
will show the manifold is equivariantly diffeomorphic to a sphere. For any $X$, $Y$ $\in \liep$, we
denote $Q(X,Y)$ by $\langle X,Y\rangle$. Without loss of generality, we may assume that $\liekm =
\lieh \oplus \liep_1$ and $\liekp = \lieh \oplus \liep_2 \oplus \liep_3$ since the group diagram is
primitive and $\G$-action is fixed-point free. Let $\liem_1 = \liep_1$ and $\liem_2=\liep_2\oplus
\liep_3$ and we define the following spaces:
\begin{equation*}
\lieh_0 = \Ann(\liem_1 \oplus \liem_2) \cap \lieh, \quad  \lieh_i = \Ann(\liem_i)\cap \lieh_0^\perp
\cap \lieh, \quad \lieh_3 = (\lieh_0 \oplus \lieh_1 \oplus \lieh_2)^\perp \cap \lieh.
\end{equation*}
A similar argument as in the proof of Theorem \ref{thmclasss2primitive} shows that
\begin{enumerate}
\item $\lieh_0 = \lieh_3 = 0$, and both $\lieh_2\oplus\liem_1$ and $\lieh_1\oplus \liem_2$ are ideals of
$\lieg$;
\item $\lieg=\lieh_1\oplus \lieh_2 \oplus \liem_1 \oplus \liem_2, \quad \liekm=\lieh_1 \oplus
\lieh_2 \oplus \liem_1, \quad \liekp=\lieh_1\oplus \lieh_2 \oplus \liem_2$.
\end{enumerate}
Let $\gH_1$, $\gH_2$, $\gL_1$ and $\gL_2$ be Lie groups of $\lieh_1$, $\lieh_2$, $\lieh_1 \oplus
\liem_2$ and $\lieh_2\oplus \liem_1$ respectively, then we have
\begin{equation*}
\G = \gL_1\times \gL_2, \quad \Km=\gH_1\times \gL_1, \quad \Kp = \gH_2\times \gL_2, \quad \mbox{
and }\quad \gH=\gH_1 \times \gH_2,
\end{equation*}
and hence the $\G$-action is a sum action and the manifold $M$ is $\G$-equivariant diffeomorphic to
a sphere.

Now we consider the last type where both $\Kpm/\gH$ are strongly isotropy irreducible. Since the
group diagram is primitive, without loss of generality, we may assume that $\liekm = \lieh \oplus
\liep_2$ and $\liekp = \lieh \oplus \liep_3$. Following the proof of Theorem
\ref{thmclasss2primitive}, the subspace $[\liep_2, \liep_3]$ is orthogonal to $\lieh\oplus
\liep_2\oplus\liep_3$. So we have $[\liep_2, \liep_3]\subset \liep_1$ and $[\liep_2, \liep_3]$ is
an invariant space under the action $\Ad_{\gH_c}$. It is either equal to $0$ to $\liep_1$ from the
irreducibility of $\liep_1$. In the first case, $\lieh \oplus \liep_2 \oplus \liep_3$ is an ideal
of $\lieg$, so the group generated by $\Km$ and $\Kp$ is a proper normal subgroup of $\G$ which
contradicts the primitivity assumption. Therefore $[\liep_2, \liep_3]=\liep_1$ which implies that
both $\G/\Kpm$ are strongly isotropy irreducible.
\end{proof}

\begin{rem}\label{remnonequivsummands}
Let $\gH \subset \set{\Kpm}\subset \G$ is a non-reducible and fixed points free diagram with $s=3$.
Suppose $\G/\Km$ is not strongly isotropy irreducible and the representation of $\Ad_{\Km/\gH}$ is
denoted by $\liep_3$. If $\liep_3$ is not equivalent to one of the summands $\liep_1$ and
$\liep_2$, then from the proof of Proposition \ref{proppairwisenonequ}, then we have
\begin{enumerate}
\item if $\liep_1$ and $\liep_2$ are non-equivalent, then the manifold is a sphere via sum
action;
\item if $\liep_1$ and $\liep_2$ are equivalent, then the diagram is non-primitive.
\end{enumerate}
\end{rem}

\subsection{Two or three summands are equivalent.}
We consider the triple of connected groups $\gH \subset \K \subset \G$ such that $\K/\gH$ is a
sphere, the isotropy representation $\Ad_{\gH}$ on the tangent space of $\K/\gH$ is irreducible and
the isotropy representation $\Ad_\K$ on the tangent space of $\G/\K$ has two irreducible summands.

\begin{prop}\label{propGK2summandsnonequiv}
Let $\G$ be a simple Lie group and $\gH\subset \K \subset \G$ be three connected Lie groups such
that $\K/\gH$ is a sphere and $\Ad_\gH$ is irreducible on the tangent space $\liep_3$ of $\K/\gH$.
Suppose the $\Ad_\K$ action on the tangent space of $\G/\K$ has exactly two irreducible summands
$\liep_1$ and $\liep_2$. If $\liep_1$ and $\liep_2$ remain irreducible when they are viewed as the
representations of $\Ad_\gH$, then $\liep_1$, $\liep_2$ and $\liep_3$ are pairwisely non-equivalent
except for the triple $\SU(3)\subset \U(3) \subset \SO(7)$.
\end{prop}

\begin{proof}
Since the groups $\gH, \K$ and $\G$ are connected, we consider their Lie algebras: $\lieh \subset
\liek \subset \lieg$. We denote the representation of $\ad_{\lieh}$ on $\liep_i$ by $\chi_i$ for
$i=1, 2, 3$.

First we consider the pair $(\G,\K)$ for which $\dim \liep_1= \dim \liep_2$ and there exist some
$\gH$ such that the sphere $\K/\gH$ is isotropy irreducible. From the classification of  $(\G,\K)$
in \cite{DickKerr} and Theorem \ref{thmtwosummands}, they are in Table \ref{tablesamedimGK}. To
save space, we only give the group $\K$. We list the condition for each pair in the last column.

\begin{table}[!ht]
\begin{center}
\begin{tabular}{|r|c|c|c|c|}
\hline
 & $\K$ & $\chi_1$ & $\chi_2$ & \\
\hline \hline
  $\I.1$ & $\SO(m) \times \SU(k), (k \geq 4)$ & $\Id \ot [\fw_2 \fw_{k-1}^2]_\Real$ & $\varrho_m \ot \fw_1 \fw_{k-1}$ & $2m=k^2 - 4$\\
  \hline
  $\I.2$ & $\SO(m) \times \SO(k), (k \geq 7)$ & $\Id \ot \fw_1 \fw_3$ & $\varrho_m \ot \fw_2$ & $4m = (k-3)(k+2)$ \\
  \hline
  $\I.5$ & $\SO(m) \times \Sp(k), (k\geq 3)$ & $\Id \ot \fw_1^2 \fw_2$ & $\varrho_m \ot \fw_1 ^2$ & $2m = (k-1)(2k-3)$\\
  \hline
  $\I.14$ & $\SO(65) \times \E_7$ & $\Id \ot \fw_3$ & $\varrho_{65} \ot \fw_1$ &  \\
  \hline
  $\I.16$ & $\Gt$ & $\fw_1$ & $\fw_1$ & \\
  \hline
  $\I.18$ & $\SO(m) \times \U(k), (k \geq 3)$ & $\Id \ot [\fw_2 \ot \phi]_\Real$ & $\varrho_m \ot [\fw_1
\ot \phi]_\Real$ & $2m=k-1$\\
  \hline
  $\II.5$ & $\SU(p)\times \SU(q) \times S(\U(1) \times \U(m))$ & \multicolumn{2}{c|}{$\fw_1\fw_{p-1} \ot \fw_1 \fw_{q-1}\ot \Id \ot \Id \ot
  \Id$} & $2mpq$ \\
  \cline{3-4}
          & $(p,q \geq 2, m\geq 1)$ & \multicolumn{2}{c|}{$[\fw_1\ot \fw_1 \ot \phi \ot\phi^*\ot
      \fw_{m-1}]_\Real$} & $=(p^2 -1)(q^2 -1)$ \\
  \hline
  $\III.6$ & $\Sp(m) \times \U(n)$ & $\Id \ot [\fw_1^2 \ot \phi^2]_\Real$ & $[\fw_1 \ot \fw_1 \ot \phi]_\Real$ & $4m= n+1$ \\
  \hline
  $\V.1$ & $\SO(m) \times \SO(m), (m \geq 3, m \neq 4)$ & $\fw_2 \ot \fw_1^2$ & $\fw_1^2 \ot
  \fw_2$ & \\
  \hline
  $\V.2$ & $\Sp(2)\times \Sp(2)$ & $\fw_1^2\ot \fw_2$ & $\fw_2\ot \fw_1^2$ & \\
  \hline
\end{tabular}
\smallskip
\caption{The pairs $(\G,\K)$ for which $\dim \liep_1 = \dim \liep_2$.}\label{tablesamedimGK}
\end{center}
\end{table}
It is easy to see that only in example \textbf{I}.18 when $m=1$ and $k=3$, if $\gH = \SU(3)$, then
$\chi_1|_{\gH}$ is equivalent to $\chi_2|_{\gH}$ and $\chi_3 = \Id$. For other cases, either at
least one of $\chi_1$ and $\chi_2$ splits when restricted to $\gH$ or the three summands are
non-equivalent.

Next we show that $\chi_3$ is not equivalent to $\chi_1|_{\gH}$ or $\chi_2|_{\gH}$. We prove it by
contradiction. Assume that $\chi_3$ is equivalent to $\chi_1|_{\gH}$. From the classification of
transitive actions of the sphere and the assumption that $\ad_{\liek/\lieh}$ is isotropy
irreducible, $(\liek, \lieh)$ is one of the pairs $(\lieso(k+1)\oplus \lieh_0, \lieso(k)\oplus
\lieh_0)(k\geq 1)$, $(\liegt\oplus \lieh_0, \liesu(3)\oplus \lieh_0)$ and $(\lieso(7)\oplus\lieh_0,
\liegt\oplus\lieh_0)$. Here $\lieh_0$ may be a zero vector space.

We consider the case when $\K/\gH$ is a circle first. Sine $\chi_3$ is equivalent to
$\chi_1|_{\gH}$, we have that $\dim \liep_1 = 1$ and $\liek$ contains a $\lieu(1)$ factor. However
such pair $(\G,\K)$ does not exist from the classification.

If $\liek = \lieso(k+1)\oplus \lieh_0$ and $\lieh = \lieso(k)\oplus \lieh_0$ with $k\geq 2$, then
$\chi_3 = \varrho_k\ot \Id$. If $\chi_1 = \sigma_1\ot \sigma_2$ is a representation of real type,
then $\sigma_2$ is the trivial representation and $\sigma_1|_{\lieso(k)} = \varrho_k$. The only
possible case if that $\liek=\liesu(2)\oplus \lieh_0$, $\lieh = \lieu(1)\oplus \lieh_0$ and
$\sigma_1 = \fw_1$. However $\chi_1 = \sigma_1 \ot \Id$ is of quaternionic type. If
$\chi_1=[\sigma_1\ot \sigma_2]_\Real$ and $\sigma_1\ot \sigma_2$ is not of real type, then
$\sigma_2 =\Id$ and the restriction of $[\sigma_1]_\Real$ to $\lieso(k)$ is $\varrho_k$. Such
$\sigma_1$ does not exist.

If $\liek = \liespin(7)\oplus \lieh_0$ and $\lieh=\liegt\oplus \lieh_0$, then $\chi_3 = \fw_1 \ot
\Id$. Since every irreducible representation of $\liespin(7)$ is of real type, $\chi_1 =
\sigma_1\ot \Id$ and $\sigma_1|_{\liegt} = \fw_1$. So we have $\sigma_1 = \varrho_7$ and then
$\chi_1 = \varrho_7 \ot \Id$. Four examples, \textbf{I}.30, \textbf{I}.32, \textbf{II}.13 and
\textbf{III}.10, have such $\chi_1$. However for each of them, $\chi_2$ splits when restricted to
$\gH$.

If $\liek = \liegt\oplus \lieh_0$ and $\lieh = \liesu(3) \oplus \lieh_0$, then $\chi_3=
[\fw_1]_\Real \otimes \Id$. Since every irreducible representation of $\liegt$ is of real type,
$\chi_1= \sigma_1 \ot \Id$. However $\sigma_1|_{\liesu(3)}$ cannot be $[\fw_1]_\Real$ for any
irreducible representation $\sigma_1$. This finishes the proof.
\end{proof}

\subsection{Group triples $\G \supset \K \supset \gH$ with $\G/\K$ strongly isotropy irreducible}

The triples such that $\Ad_\gH$ has only three irreducible summands and $\K/\gH$ is a sphere are
classified in

\begin{prop}\label{proptripleHKG}
The triples $\G \supset \K  \supset \gH$ with $\G/\K$ a strongly isotropy irreducible homogeneous
space, $\K/\gH$ a sphere and $\gH$ connected, for which $\Ad_\gH$ has exactly three irreducible
summands are listed in Table \ref{tabletripleHKG}.
\end{prop}

\begin{center}
\begin{longtable}{|l|l|l|l|}

\hline
$\quad \quad \quad \G$ & $\quad \quad \quad \quad \K$ & $\quad \quad \quad \quad \gH$ &  \\
\hline \hline
\endfirsthead

\multicolumn{4}{c}%
{{\tablename\ \thetable{} -- continued from previous page}} \\
\hline
$\quad \quad \quad \G$ & $\quad \quad \quad \quad \K$ & $\quad \quad \quad \quad \gH$ &  \\
\hline \hline
\endhead

\hline \multicolumn{4}{|r|}{{Continued on next page}} \\ \hline
\endfoot

\endlastfoot

$\SU(2)\times \SU(2)$ & $\Diag \SU(2)$ & $\U(1)$ & \\
\hline
$\Spin(n) \times \Spin(n)$ & $\Diag \Spin(n)$ & $\Spin(n-1)$ & $n \geq 6$ \\
\hline
$\Spin(7)\times \Spin(7)$ & $\Diag \Spin(7)$ & $\Gt$ & \\
\hline
$\Gt \times \Gt$ & $\Diag \Gt$ & $\SU(3)$ & \\
\hline \hline
$\SU(4p)$ & $\SU(4)\times \SU(p)$ & $\Sp(2)\times \SU(p)$ & $p\geq 2$\\
\cline{2-3}
& $\SU(p)\times \SU(4)$ & $\SU(p)\times \Sp(2)$ & \\
\hline
$\SU(2p)$ & $\SU(p)\times \SU(2)$ & $\SU(p) \times \U(1)$ & $p \geq 3$\\
\hline
$\SU(16)$ & $\Spin(10)$ & $\Spin(9)$ & \\
\hline
$\SU(4)$ & $\SU(2)\times \SU(2)$ & $\U(1)\times \SU(2)$ & \\
\cline{3-4}
&  & $\SU(2)\times \U(1)$ & \\
\hline
$\SU(3)$ & $\U(2)$ & $\U(1) \times \U(1)$ & \\
\hline
$\SU(2)$ & $\SO(2)$ & $\set{1}$ & \\
\hline \hline
$\SO(p+2)$ & $\SO(p+1)$ & $\SO(p)$ & $p\geq 4$ \\
\hline
$\SO(p+q+2)$ & $\SO(p+1)\times \SO(q+1)$ & $\SO(p) \times \SO(q+1)$ & $p, q \geq 1$ \\
\cline{3-3}
&  & $\SO(p+1)\times \SO(q)$ & \\
\hline
$\Spin(6+2p)$ & $\Spin(6) \times \SO(2p)$ & $\SU(3)\times \SO(2p)$ & $p\geq 1$\\
\cline{2-3}
& $\SO(2p)\times \Spin(6)$ & $\SO(2p)\times \SU(3)$ & \\
\hline
$\Spin(7+2p)$ & $\Spin(6) \times \SO(2p+1)$ & $\SU(3)\times \SO(2p+1)$ & $p\geq 0$\\
\hline
$\Spin(128)$ & $\Spin(16)$ & $\Spin(15)$ & \\
\hline
$\Spin(16)$ & $\Spin(9)$ & $\Spin(8)$ & \\
\hline
$\SO(7)$ & $\U(3)$ & $\SU(3)$ & \\
\hline
$\SO(8)$ & $\U(4)$ & $\SU(4)$ & \\
\hline
$\SO(8)$ & $\Spin(7)$ & $\Spin(6)$ & \\
\hline
$\Spin(7)$ & $\Spin(6)$ & $\SU(3)$ & \\
\hline
$\Spin(7)$ & $\Gt$ & $\SU(3)$ & \\
\hline \hline
$\Sp(2)$ & $\Sp(1)\times \Sp(1)$ & $\Diag \Sp(1)$ & \\
\hline \hline
$\E_6$ & $\SU(6) \cdot \SU(2)$ & $\SU(5)\cdot \SU(2)$ & \\
\hline
$\E_6$ & $\SU(3)\times \Gt$ & $\SU(3)\times \SU(3)$ & \\
\hline
$\F_4$ & $\Spin(9)$ & $\Spin(8)$ & \\
\hline
$\F_4$ & $\SO(3)\times \Gt$ & $\SO(3)\times \SU(3)$ & \\
\hline
$\Gt$ & $\SO(4)$ & $\SO(3)$ & \\
\hline

\caption{Group triple $\G\supset \K \supset \gH$ such that $\G/\K$ is strongly isotropy
irreducible, $\Ad_\gH$ has 3 summands and $\K/\gH$ is a sphere.}\label{tabletripleHKG} \\
\end{longtable}
\end{center}

The proof is straightforward. We use the classifications of compact irreducible symmetric spaces
and strongly isotropy irreducible homogeneous space $\G/\K$ in \cite{WolfIrr}. For each pair $(\G,
\K)$ we list the possible $\gH$'s such that $\K/\gH$ is a sphere and then compute the isotropy
representation $\Ad_\gH$ of $\G/\gH$. If $\Ad_\gH$ has precisely three irreducible summands, then
we include the triple $\G\supset \K \supset \gH$ in our list.

\begin{rem}
The triples $\set{1}\subset \SO(2)\subset \SU(2)$, $\U(1)\subset \Diag \SU(2)\subset \SU(2)\times
\SU(2)$ and $\Diag \Sp(1) \subset \Sp(1)\times \Sp(1) \subset \Sp(2)$ can be viewed as the triple
$\SO(p)\subset \SO(p+1)\subset \SO(p+2)$ when $p=1$, $2$ and $3$.
\end{rem}

\subsection{Construction of cohomogeneity one group diagrams} From the previous sections, we only
have to consider the triple $\SU(3) \subset \U(3) \subset \SO(7)$ and those in Table
\ref{tabletripleHKG}.

We consider the case where $\gH$ is not connected first and we have
\begin{thm}\label{thmrigiditydisconnected}
The \cohomd defined by a primitive group diagram $\gH \subset \set{\Kpm}\subset \G$ with $\G$ a
simple Lie group, $\gH$ disconnected and $s=3$ is a complex project space, see Example 1 and 2.
\end{thm}

\begin{proof}
By Lemma \ref{lemsimplyconnected}, we may assume that $\lm=1$, i.e., $\Km$ is connected and
$\Km/\gH_c = \sph^1$. From the classification in Proposition \ref{proptripleHKG}, $\gH_c \subset \K
\subset \G$ is one of
\begin{eqnarray}
& & \SO(p) \subset \SO(2) \times \SO(p) \subset \SO(2+p), \quad (p\geq 3), \label{HKGSOpSO2p}\\
& & \SU(3) \subset \U(3) \subset \SO(7) \label{HKGSU3SO7}\\
& & \SU(4)\subset \U(4) \subset \SO(8)\label{HKGSU4SO8}.
\end{eqnarray}

Suppose that $\gH_c \subset \K \subset \G$ is the triple in (\ref{HKGSOpSO2p}). If $\Kp_c$, the
connected component of $\Kp$, is $\SO(2)\times \SO(p)$ and then both $\lpm$ equal to $1$. Since
$N(\Kp_c) = S(\gO(2)\times \gO(p)) = \SO(2)\times \SO(p)\cup (\SO(2)\times \SO(p)) \cdot A$ with $A
= \diag(1,-1,-1,I_{p-1})$, the diagram is
\begin{eqnarray*}
& & \SO(p)\cdot \set{1,A} \subset \set{S(\gO(2)\times \gO(p)), S(\gO(2)\times \gO(p))}\subset
\SO(p+2),
\end{eqnarray*}
The \cohomd defined by the above diagram is a double. It is not simply connected and finitely
covered by the manifold defined by the diagram $\SO(p)\subset \set{\SO(2)\times \SO(p),
\SO(2)\times \SO(p)}\subset \SO(p+2)$.

If $\Kp_c$ is $\SO(p+1)$ and then $N(\Kp_c)/\Kp_c = \Zeit_2$ generated by the matrix $\diag(-I_2,
I_p)$. So we have
\begin{eg} The diagram is
\begin{equation*}
\gH=\SO(p)\cdot \Zeit_2 \subset \set{\SO(2)\times \SO(p), \gO(p+1)} \subset \SO(2+p)
\end{equation*}
and the manifold is $\Cp^{p+1}$, see, for example, \cite{GWZ}.
\end{eg}

Next we consider variations of these diagrams. If $p$ is odd, then $\Aut(\G, \gH) = S(\gO(2)\times
\gO(p))$ which is the same as $\Aut(\G,\Km)$. If $p$ is even, then $\G$ has an outer automorphism
which is conjugation by $\diag(-1,I_{p+1})$. It is clear that this automorphism leaves $\Km$
invariant. Therefore the variation does not give another new diagram.

If $p=6$, there are a few more possible constructions. Let us lift $\G=\SO(8)$ to its universal
cover $\Spin(8)$, then the triple is lifted to $\Spin(6)\subset (\SO(2)\times \Spin(6))/\Zeit_2
\subset \Spin(8)$ where $\Zeit_2$ is generated by $-\id \in \Spin(8)$. $\Spin(8)$ has another order
$3$ outer automorphism denoted by $\sigma$. There are three different embeddings of $\Spin(7)$ into
$\Spin(8)$ and $\sigma$ permutes them. If $\sigma$ leaves some $\Spin(6)$ invariant, then it would
be contained in the intersection of the three $\Spin(7)$'s which would imply this $\Spin(6)$ is
contained in $\Gt$. Therefore there is no such $\Spin(6)$ invariant by the automorphism $\sigma$.
On the other hand there is another intermediate subgroup $\tU$, the image of $\U(4)\subset \SO(8)$
by the lifting, between $\Spin(6) = \SU(4)$ and $\Spin(8)$. Since $\Spin(8)/\tU = \SO(8)/\U(4)$ is
simply-connected, $\tU$ is connected. Both $(\SO(2)\times \Spin(6))/\Zeit_2$ and $\tU$ contain
$\Spin(6)$ and the isotropy representation of the space $\Spin(8)/\Spin(6)$ contains only one
trivial representation $\Id$, so they are the same subgroup in $\Spin(8)$. Divided by the
ineffective kernel, the diagram
\begin{equation*}
\Spin(6)\subset \set{(\SO(2)\times \Spin(6))/\Zeit_2, \tU}\subset \Spin(8)
\end{equation*}
reduces to
\begin{equation*}
\SO(6)\subset \set{\SO(2)\times \SO(6), \SO(2)\times \SO(6)}\subset \SO(8)
\end{equation*}
which is a double.

If $\gH_c \subset \K \subset \G$ is the one in (\ref{HKGSU3SO7}), then we lift $\SO(7)$ to
$\Spin(7)$ and obtain the following triple
\begin{equation}
\SU(3) \subset \SO(2)\times \SU(3) \subset \Spin(7) \label{HKGSU3Spin7}.
\end{equation}
From the classification, $\Kp_c$ is either $\Gt$ or $\Spin(6)$. If $\Kp_c = \Gt$, then
$N(\Kp_c)/{\Kp_c} = \Zeit_2$. So we have

\begin{eg}
The diagram is
\begin{equation*}
\Zeit_2 \cdot \SU(3)\subset \set{\SO(2)\cdot \SU(3), \Zeit_2 \cdot \Gt} \subset \Spin(7),
\end{equation*}
where $\Zeit_2$ is the center of $\Spin(7)$. The manifold is the complex projective space $\Cp^7$,
see, for example, \cite{Uchida}.
\end{eg}

Next we consider the case where $\Kp_c = \Spin(6)$. Since $\SO(2)\cdot\SU(3)$ is the 2-fold cover
of $\U(3) \subset \SO(7)$ and there is only one $\SO(6)$ in $\SO(7)$ contains $\SU(3)$ that also
contains $\U(3)$, it follows that $\Spin(6)$ which is the 2-fold cover of $\SO(6)$ contains
$\SO(2)\cdot \SU(3)$. So one cannot add components to isotropy subgroups, i.e., $\gH$ is connected.

If $\gH_c \subset \K \subset \G$ is the one in (\ref{HKGSU4SO8}) and $\Kp_c = \Spin(7)$, then
$N_\G(\Kp_c) = \Kp_c$ which implies $\gH$ is connected. If both $\Kpm_c = \K = \U(4)$, then we have
$N_\G(\K)/\K = \Zeit_2$ and it is generated by the diagonal matrix $A=\diag(I_4, -I_4)$. Since
$N_\G(\gH_c) = N_\G(\K)$ and there is no circle group inside $N_\G(\gH_c)/\gH_c$ containing $A$,
this triple does not give any cohomogeneity one diagram with a disconnected $\gH$. .
\end{proof}

\medskip

Next we consider the cases where $\gH$ is connected. Since there is no exceptional orbit, both
$\Kpm$ are connected. In the classification in Proposition \ref{proptripleHKG}, many pairs
$(\gH,\G)$ contain only one intermediate subgroup.

\begin{defn}
Two irreducible representations $\vp$ and $\psi$ of $\gH$ are \emph{outer equivalent} if $\vp =
\tau(\psi)$ by an outer automorphism of $\gH$.
\end{defn}

Recall that $\chi_1$, $\chi_2$ and $\chi_3$ are the three irreducible summands of the isotropy
representation $\Ad_\gH$ on $\G/\gH$. Let $\Ad_\gH(\G/\K)$ and $\Ad_\gH(\K/\gH)$ be the
restrictions of $\Ad_\gH$ to the tangent spaces of $\G/\K$ and $\K/\gH$ respectively.

\begin{lem}\label{lemnonautoequivalent}
Suppose that any irreducible summand of $\Ad_\gH(\K/\gH)$ is not equivalent or outer equivalent to
any summand in $\Ad_\gH(\G/\K)$, then the \cohomd defined by any variation of the diagram
$\gH\subset \set{\K, \K} \subset \G$ is a double.
\end{lem}

\begin{proof}
We give a proof when $\Ad_\gH(\K/\gH)=\chi_3$ is irreducible. The other case where
$\Ad_\gH(\K/\gH)$ is reducible follows easily. Let $\tau \in \Aut(\G,\gH)$, then $\tau$ is an
automorphism of $\gH$ and it permutes the three summands. By assumption, $\tau(\chi_3)$ is not
equivalent to $\chi_1$ or $\chi_2$, so $\tau(\chi_3) = \chi_3$ which implies the Lie algebra of
$\K$ and hence $\K$ itself is invariant by $\tau$. Therefore the manifold defined by $\gH\subset
\set{\K,\tau(\K)}\subset \G$ is a double.
\end{proof}

We list all triples which satisfy the condition in Lemma \ref{lemnonautoequivalent}.

\begin{prop}\label{propnonautoeuivalent}
The group triples $\gH \subset \K \subset \G$ in Table \ref{tabletripleHKG} such that any
irreducible summands of $\Ad_\gH(\K/\gH)$ is not equivalent or outer equivalent to the summand of
$\Ad_\gH(\G/\K)$ are classify in Table \ref{tablenonautoequivalent}.
\end{prop}

\begin{center}
\begin{longtable}{|l|l|c|c|}

\hline
$\quad \quad \quad \G$ & $\quad \quad \quad \quad \gH$ & $\Ad_\gH(\K/\gH)$ & $\Ad_\gH(\G/\K)$ \\
\hline \hline
\endfirsthead

\multicolumn{4}{c}%
{{\tablename\ \thetable{} -- continued from previous page}} \\
\hline
$\quad \quad \quad \G$ & $\quad \quad \quad \quad \gH$ & $\Ad_\gH(\K/\gH)$ & $\Ad_\gH(\G/\K)$ \\
\hline \hline
\endhead

\hline \multicolumn{4}{|r|}{{Continued on next page}} \\ \hline
\endfoot

\endlastfoot

  $\SU(4p)$ & $\Sp(2)\times \SU(p)$ & $\fw_1\ot\Id$ & $(\fw_2\oplus\fw_1) \ot (\fw_1 + \fw_{p-1})$ \\
  \cline{2-4}
  & $\SU(p)\times \Sp(2)$ & $\Id \ot \fw_1$ & $(\fw_1 + \fw_{p-1}) \ot (\fw_2\oplus\fw_1)$ \\
  \hline
  $\SU(2p)$ & $\SU(p) \times \U(1)$ & $\Id \ot [\phi]_\Real$ & $(\fw_1 + \fw_{p-1}) \ot (\Id \oplus [\phi]_\Real)$ \\
  \hline
  $\SU(16)$ & $\Spin(9)$ & $\fw_1$ & $2\fw_4 \oplus \fw_3$ \\
  \hline
  $\SU(4)$ & $\U(1)\times \SU(2)$ & $[\phi]_\Real \ot \Id$ & $(\Id \oplus [\phi]_\Real)\ot 2\fw_1$ \\
  \cline{2-4}
  & $\SU(2)\times \U(1)$ & $\Id \ot [\phi]_\Real$ & $2\fw_1 \ot (\Id \oplus [\phi]_\Real)$ \\
  \hline
  $\SO(p+q+2)$ & $\SO(p) \times \SO(q+1)$ & $\fw_1 \ot \Id$ & $(\fw_1\oplus\Id)\ot \fw_1$\\
  \cline{2-4}
  & $\SO(p+1)\times \SO(q)$ & $\Id \ot \fw_1$ & $\fw_1 \ot (\fw_1\oplus\Id)$ \\
  \hline
  $\Spin(7+2p)$ & $\SU(3)\times \SO(2p+1)$ & $(\Id \oplus [\fw_1]_\Real) \ot\Id$ & $[\fw_1]_\Real \ot \fw_1$ \\
  \hline
  $\Spin(6+2p)$ & $\SU(3)\times \SO(2p)$ & $(\Id \oplus [\fw_1]_\Real) \ot\Id$ & $[\fw_1]_\Real \ot \fw_1$ \\
  \cline{2-4}
  & $\SO(2p)\times \SU(3)$ & $\Id \ot (\Id \oplus [\fw_1]_\Real)$ & $\fw_1 \ot [\fw_1]_\Real$ \\
  \hline
  $\Spin(128)$ & $\Spin(15)$ & $\fw_1$ & $\fw_5 \oplus \fw_6$ \\
  \hline
  $\Spin(16)$ & $\Spin(8)$ & $\fw_1$ & $(\fw_3 + \fw_4)\oplus \fw_2$ \\
  \hline
  $\SO(7)$ & $\SU(3)$ & $\Id$ & $[\fw_1]_\Real \oplus [\fw_1]_\Real$ \\
  \hline
  $\SO(8)$ & $\SU(4)$ & $\Id$ & $\fw_2 \oplus \fw_2$ \\
  \hline
  $\E_6$ & $\SU(5)\cdot \SU(2)$ & $(\Id \oplus [\fw_1]_\Real) \ot \Id$ & $[\fw_2]_\Real \ot \fw_1$ \\
  \hline
  $\E_6$ & $\SU(3)\times \SU(3)$ & $\Id \ot [\fw_1]_\Real$ & $(\fw_1 + \fw_2)\ot ([\fw_1]_\Real\oplus \Id)$ \\
  \hline
  $\F_4$ & $\SO(3)\times \SU(3)$ & $\Id \ot [\fw_1]_\Real$ & $4\fw_1\ot([\fw_1]_\Real\oplus \Id)$ \\
  \hline

\caption{Group pair $\G \supset \gH$ in Table \ref{tabletripleHKG} such that $\Ad_\gH(\K/\gH)$ has
no summand equivalent or outer equivalent to one summand in
$\Ad_\gH(\G/\K)$.}\label{tablenonautoequivalent}
\end{longtable}
\end{center}

In Table \ref{tablenonautoequivalent}, when $\G=\SO(p+q+2)$ the two factors of $\gH$ should be of
different sizes, i.e., $p\ne q+1$ or $p+1 \ne q$.

Now we state the result when $\gH$ is connected:

\begin{thm}\label{thmrigidityconnected}
The \cohomd defined by a primitive group diagram $\gH \subset \set{\Kpm}\subset \G$ with $\G$
simple, $\gH$ connected and $s=3$ is either $\sph^7$, $\sph^{14}$, $\sph^{25}$, $\Cp^7$ or the
Grassmannian $\SO(p+q+2)/(\SO(p+1)\times \SO(q+1))$ with $p, q\geq 1$, see Example 3 -- 9.
\end{thm}

\begin{proof}
There are two main steps in the proof. In Step 1, we consider the pairs ($\gH$, $\G$) for which
there are at least two intermediate groups. In Step 2, we consider the variations of doubles. We
fix the notations for the outer automorphisms of $\SO(2m)$(or $\Spin(2m)$): $\lambda$ is the degree
$2$ outer automorphism and $\sigma$ is the degree $3$ outer automorphism of $\Spin(8)$.

\textsc{Step 1:} From the classification of the triples, between the following four pairs of
($\gH$, $\G$), there are more than one intermediate subgroups $\K$. They are
\begin{eqnarray}
& & \U(1)\times \U(1) \subset \set{S(\U(1)\times\U(2)), S(\U(2)\times \U(1)} \subset \SU(3),
\label{eqndiagramS7} \\
& & \SU(3) \subset \set{\Spin(6), \Gt} \subset \Spin(7), \label{eqndiagramS14} \\
& & \SU(3) \subset \set{\SO(2)\cdot \SU(3), \Gt} \subset \Spin(7), \label{eqndiagramQ7} \\
& & \SU(3) \subset \set{\SO(2)\cdot \SU(3), \Spin(6)} \subset \Spin(7), \label{eqndiagramnonp} \\
& & \SU(4) \subset \set{\Spin(7), \U(4)} \subset \SO(8), \label{eqndiagramCP7}
\end{eqnarray}
and
\begin{equation}\label{eqndiagramGr}
\SO(p)\times \SO(q) \subset \set{\SO(p)\times \SO(q+1), \SO(p+1)\times \SO(q)} \subset \SO(p+q+1),
\end{equation}
where $p$, $q \geq 1$.

\begin{eg}
The manifold defined by the diagram (\ref{eqndiagramS7}) is the sphere $\sph^{7}$ and the embedding
$\SU(3) \hookrightarrow \SO(8)$ is given by the adjoint representation of $\SU(3)$, see \cite{GWZ}
and example $Q^7_E$ in \cite{HoelscherClass}.
\end{eg}

\begin{eg}
The manifold defined by the diagram (\ref{eqndiagramS14}) is the sphere $\sph^{14}$ and the
embedding $\Spin(7) \hookrightarrow \SO(15)$ is given by $\varrho_7 \oplus \Diag_7$ where $\Diag_7$
is the spin representation of $\Spin(7)$, see, for example, \cite{GWZ}.
\end{eg}
We know that $N_{\Spin(7)}(\SU(3))/\SU(3) = \Zeit_2$ and the generator can be represented as, for
example, $A =\diag(I_3, -I_4)$. Both $\Gt$ and $\Spin(6)$ are invariant under the conjugation of
$A$. Hence any variation of the diagram gives the same cohomogeneity one manifold.

\begin{eg}
The manifold defined by the diagram (\ref{eqndiagramQ7}) is the Grassmannian
$\SO(9)/\SO(2)\times\SO(7)$, see, for example, \cite{Uchida}.
\end{eg}
Following a similar argument as in the previous case, any variation of this diagram does not give
us a new \cohomd.

The diagram (\ref{eqndiagramnonp}) is not primitive and we have seen that $\SO(2)\cdot \SU(3)$ is
contained in $\Spin(6)$.

\begin{eg}
The manifold defined by the diagram (\ref{eqndiagramCP7}) is the projective space $\Cp^7$, see, for
example, \cite{Uchida}.
\end{eg}
A similar argument shows that any variation does not give us a new \cohomd.

\begin{eg}
The manifold defined by the diagram (\ref{eqndiagramGr}) is the Grassmannian
$\SO(p+q+2)/(\SO(p+1)\times \SO(q+1))$ and the embedding $\SO(p+q+1)\hookrightarrow \SO(p+q+2)$ is
given by $\varrho_{p+q+1} \oplus \id$.
\end{eg}
Let $\Km$ and $\Kp$ denote $\SO(p)\times \SO(q+1)$ and $\SO(p+1)\times \SO(q)$ respectively and
assume that one of $p, q$ is bigger than $1$. If $p\ne q$, by Proposition
\ref{propnonautoeuivalent}, any $\tau\in \Aut(\G,\gH)$ leaves both $\Kpm$ invariant. So we only
need to consider the case $p=q$. In this case, there is one automorphism of $\gH$ given by
conjugation of the matrix
\begin{equation}\label{eqnconjugationSOp}
J=
\begin{pmatrix}
  &   & I_p \\
  & 1 &     \\
I_p & &  \\
\end{pmatrix},
\end{equation}
where the entries without specifying values have zeros. But $\Kpm$ switch each other by the
conjugation of $J$. Therefore there is no new manifold from the variation.

\textsc{Step 2:} Combining the results in Proposition \ref{proptripleHKG} and Proposition
\ref{propnonautoeuivalent}, there are a few triples $\gH\subset \K \subset \G$ which need to be
considered. In the following, we analyze each of them.

\smallskip

$\bullet \quad$ $\U(1)\times \U(1) \subset \U(2) \subset \SU(3).$

There are only two different $\U(2)$'s between $\U(1)\times \U(1)$ and $\SU(3)$ and the primitive
diagram gives the sphere $\sph^7$. It is already appeared in Step 1.

\smallskip

$\bullet \quad$ $\SO(p)\times \SO(p)\subset \SO(p+1)\times \SO(p) \subset \SO(2p+1)(p\geq 2).$

The conjugation by $J$ defined in (\ref{eqnconjugationSOp}) maps $\SO(p+1)\times \SO(p)$ to
$\SO(p)\times \SO(p+1)$, so the variation gives the Grassmannian $\SO(2p+2)/(\SO(p+1)\times
\SO(p+1))$ which already appeared in Step 1.

\smallskip

$\bullet\quad$ $\SO(p) \subset \SO(p+1) \subset \SO(p+2).$

If $p$ is odd, then $\Aut(\G,\gH) = N_\G(\gH) = S(\gO(2)\times \gO(p))$ is connected and hence any
variation gives the double. If $p$ is even then the automorphism $\lambda$ leaves $\K$ invariant
too. If $p=6$, then $\sigma$ does not leave any $\SO(6)$ invariant. So this triple only gives a
double.

\smallskip

$\bullet\quad$ $\SU(3)\subset \Spin(6) \subset \Spin(7).$

Let $i: \SU(3)\imbed \Spin(6)$ and $j: \Spin(6) \imbed \Spin(7)$ be the embeddings. Since $\SU(3)$
is simply-connected, we have the following commutative diagram:
\begin{equation*}
\xymatrix{
  \SU(3) \ar[d]_{\id} \ar[r]^{i} & \Spin(6) \ar[d]_{\pi} \ar[r]^{j} & \Spin(7) \ar[d]^{\pi}  \\
   \SU(3) \ar[r]_{\gamma} & \SO(6) \ar[r] & \SO(7).}
\end{equation*}

The embedding $\gamma$ is given by the representation $[\fw_1]_\Real$ of $\SU(3)$. The outer
automorphism(the complex conjugation) of $\SU(3)$ is given by an inner automorphism of $\SO(7)$,
the conjugation by the matrix $\diag(I_3, -I_4)$, and $\SO(6)$ is invariant by the conjugation. So
every element in $N_{\Spin(7)}(\SU(3))$ leaves $\Spin(6)$ invariant and the variation gives only a
double.

\smallskip

$\bullet\quad$ $\SU(3) \subset \Gt \subset \SO(7).$

As seen in the previous example, conjugation by the matrix $\diag(I_3, -I_4)$ represents the outer
automorphism of $\SU(3)$. From the embedding of the Lie algebras $\liegt \subset \lieso(7)$, see
for example \cite{He}, it is easy to check that $\liegt$ is invariant by the conjugation and hence
$\Gt$ is also invariant. So only the double can be obtained from this triple.

\smallskip

$\bullet\quad$ $\Spin(6) \subset \Spin(7) \subset \SO(8).$

The subgroup $\Spin(6)$ embeds in $\SO(8)$ as the ordinary $\SU(4)$ and then $N_{\SO(8)}(\Spin(6))$
is a circle. It follows that any variation by an element in $\SO(8)$ gives us a double. $\Spin(7)$
is also invariant under the outer automorphism $\lambda$ of $\SO(8)$, so we only have a double from
this triple.

\smallskip

$\bullet\quad$ $\Spin(8)\subset \Spin(9) \subset \F_4.$

The pair $(\Spin(8), \F_4)$ appeared in the classification of isotropy irreducible Riemannian
manifolds in \cite{WangZillerIsotropy}. There are three different embeddings of $\Spin(9)$ in
$\F_4$ which are denoted by $\K_i$($i=1, 2, 3$) and every outer automorphism of $\Spin(8)$ lifts to
an inner automorphism of $\F_4$. We use the same notations as $\lambda$ and $\sigma$ for their
images in $\Aut(\F_4)$. Then $\lambda$ exchanges $\K_1$, $\K_2$ and fixes $\K_3$, and $\sigma$
permutes $\K_i$ cyclically. Other than the diagram $\Spin(8) \subset \set{\K_1, \K_1} \subset \F_4$
which defines the double, we have the following three group diagrams:
\begin{equation}
\Spin(8) \subset \set{\K_1, \K_2}\subset \F_4, \quad \Spin(8) \subset \set{\K_2, \K_3}\subset \F_4,
\quad \Spin(8) \subset \set{\K_1, \K_3}\subset \F_4. \label{diagramF4Spin8}
\end{equation}
If we apply $\sigma$ to the first diagram, then we get the second one. Then we apply $\lambda$ to
the second one, we obtain the last one. So the three group diagrams above are equivalent.

\begin{eg}
The diagram is
\begin{equation*}
\Spin(8) \subset \set{\Spin(9)_1, \Spin(9)_2} \subset \F_4
\end{equation*}
and the manifold is the sphere $\sph^{25}$ where $\F_4$ is embedded into $\SO(26)$ by its unique
$26$ dimensional representation, see, for example, \cite{GWZ}.
\end{eg}

\smallskip

$\bullet\quad$ $\SO(3)\subset \SO(4) \subset \Gt.$

All three groups are embedded in $\SO(7)$ which acts on the Cayley numbers $\Cay$ fixing the
identity element $1$ and $\Gt$ is the automorphism group of $\Cay$.

Let $\set{1,\qi,\qj,\qk, e, \qi e, \qj e, \qk e}$ be the basis of $\Cay$ over the reals, then
$\Cay$ can be written as $\Qua \oplus \Qua e$. For every element $(q_1, q_2)\in \Sp(1)\times
\Sp(1)$, it acts on $a+ b e \in \Cay$ by $(q_1 a \bar{q_1}) + (q_2 b \bar{q_1})e$. The kernel of
the action is $\set{(1,1),(-1,-1)}$, so it induces an action by $\SO(4)$. If we choose $(q_1,
q_2)\in \Diag \Sp(1)$, then it induces an $\SO(3)$ action on the Cayley numbers. It is clear from
the action that $\SO(3)$ fixes the elements $1$, $e$ and its normalizer in $\SO(7)$ consists of the
reflection about the real line and the rotation $R(t)$ as follows:
\begin{eqnarray*}
\qi \mapsto \qi \cos t + \qi e \sin t, & \qj \mapsto \qj \cos t + \qj e \sin t,
& \qk \mapsto \qk \cos t + \qk e \sin t, \\
\qi e \mapsto -\qi\sin t + \qi e \cos t, & \qj e \mapsto -\qj\sin t + \qj e \cos t, & \qk e \mapsto
-\qk\sin t + \qk e \cos t.
\end{eqnarray*}
The reflection is not an automorphism of $\Cay$ and a computation shows that $R(t) \in \Gt$ if and
only if $t$ equals to $0, \frac{2}{3}\pi$ or $\frac{4}{3}\pi$. Therefore $N_\G(\gH)/\gH = \Zeit_3$
and it is generated by $\theta=R(\frac{2}{3}\pi)$. From the action of $\SO(4)$ on $\Cay$, we know
that $\theta$ does not leave $\SO(4)$ invariant. So except the double, we have

\begin{eg}
The diagram is
\begin{equation*}
\SO(3) \subset \set{\SO(4), \Ad_{\theta}(\SO(4))} \subset \Gt
\end{equation*}
and the manifold is the Grassmannian $\SO(7)/(\SO(3)\times \SO(4))$ and $\Gt$ acts on it via the
embedding $\Gt \subset \SO(7)$ by its unique $7$ dimensional representation.
\end{eg}
\end{proof}

\medskip

\section{Primitive and non-reducible actions with $s=3$ and $\G$ not a simple Lie group}

In this section, we give a classification when $\G$ is not simple and $s=3$. We assume that the
diagram is primitive and nonreducible. The main result is

\begin{thm}\label{thms3Gnonsimple}
Suppose a compact simply-connected manifold $M$ admits a cohomogeneity one action by $\G$ and the
cohomogeneity one diagram is primitive and non-reducible. If $\G$ is not a simple Lie group, then
$M$ is equivariantly diffeomorphic to a sphere, a complex or quaternionic projective space, or the
Grassmannian $\SO(5)/\SO(3)\times \SO(2)$, see Table \ref{Tablespherenotsum} and
\ref{Tableprojectivespaces}.
\end{thm}

We separate the proof of this theorem into two different cases. In one case, there exists a
primitive factor of $\gH$ which is not contained in a single primitive factor of $\G$, i.e., it is
diagonally embedded in $\G$. In the other case, we assume that such diagonally embedded factor does
not exist. The results in the two cases are stated in Proposition \ref{propclasscasea} and
Proposition \ref{propclasscaseb} respectively.

Suppose the Lie algebra $\lieg$ of $\G$ has the decomposition as $\lieg = \lieg_1 \oplus \lieg_2$
with $\lieg_1$ a simple factor. Accordingly the Lie algebra $\lieh$ decomposes as $\lieh = \lieh_0
\oplus \lieh_1 \oplus \lieh_2$ such that the embedding $\lieh \subset \lieg$ is given by $(X_0,X_1,
X_2)\mapsto ((X_0,X_1),(X_0,X_2))$. Since the diagram is non-reducible, $\lieh_0\oplus \lieh_i$ is
a proper subspace of $\lieg_i$ for $i=1,2$. Fix a bi-invariant inner product $Q_i$ on $\lieg_i$ and
denote the orthogonal complement of $\lieh_0 \oplus \lieh_i$ by $\liep_i$ for $i = 1, 2$ and they
are nonzero vector spaces. We separate our discussion into two different cases.

\textsc{Case A. } $\lieh_0$ is a nontrivial Lie algebra.

Let us denote the isotropy representation of the pair $(\lieg_i, \lieh_0\oplus \lieh_i)$ by
$\zeta_i$, then the isotropy representation $\Ad_{\gH_c}$ is
\begin{equation*}
\chi = (\zeta_1 \ot \Id) \oplus (\Id \ot \zeta_2) \oplus (\Id \ot \ad_{\lieh_0}\ot \Id),
\end{equation*}
where $\ad_{\lieh_0}$ is the isotropy representation of the pair $(\lieh_0\oplus\lieh_0, \Diag
\lieh_0)$ and $\Diag \lieh_0$ is the image of the diagonal embedding $\lieh_0 \subset \lieh_0\oplus
\lieh_0$. From the assumption that $s=3$, each $\chi_i$ is irreducible, and then $\zeta_1$ and
$\zeta_2$ are irreducible and $\lieh_0$ is primitive. Since the diagram is non-reducible, $\chi_3 =
\Id \ot \ad_{\lieh_0}\ot \Id$ is not equivalent to $\chi_1$ or $\chi_2$.

\begin{prop}\label{propclasscasea}
Suppose $\gH \subset \set{\Kpm} \subset \G$ is a primitive and non-reducible group diagram. If the
triple of Lie algebras $\set{\lieh \subset \liekpm \subset \lieg}$ lies in \textsc{Case A}, then
the cohomogeneity one manifold is either a sphere, a complex or quaternionic projective space, or
the Grassmannian $\SO(5)/(\SO(2)\times \SO(3))$.
\end{prop}

\begin{proof}
If $\chi_1 = \zeta_1\ot \Id$ is equivalent to $\chi_2 = \Id \ot \zeta_2$, then both $\lieh_1$ and
$\lieh_2$ are zero vector space. It follows that $\lieg_2 = \lieg_1$. There are a few examples of
the triple $\lieh\subset \liek \subset \lieg$ listed as follows:
\begin{enumerate}
\item $\Delta \lieu(1) \subset \liesu(2) \oplus \lieu(1) \subset \liesu(2) \oplus \liesu(2)$,
\item $\Delta \lieu(1) \subset \Delta \liesu(2) \subset \liesu(2)\oplus \liesu(2)$,
\item $\Delta \lieu(1) \subset \lieu(1)\oplus \lieu(1) \subset \liesu(2) \oplus \liesu(2)$,
\item $\Delta \lieso(n) \subset \Delta \lieso(n+1) \subset \lieso(n+1) \oplus \lieso(n+1)$ with $n\geq 3$,
\item $\Delta \liesu(2) \subset \liesu(2)\oplus \liesu(2) \subset \liesp(2) \oplus \liesp(2)$,
\item $\Delta \lieso(3) \subset \lieso(3)\oplus \lieso(3) \subset \liesu(3)\oplus \liesu(3)$,
\item $\Delta \lieso(3) \subset \lieso(3) \oplus \lieso(3) \subset \liegt \oplus \liegt$.
\end{enumerate}
Only when $\lieh = \Delta \lieu(1)$ and $\lieg=\liesu(2)\oplus \liesu(2)$, we have primitive
diagrams. The cohomogeneity one manifolds are 6 dimension, see example $Q^6_A$ and $Q^6_C$ in
\cite{HoelscherClass}. The manifolds are $\sph^6$, $\Cp^3$ and the Grassmannian
$\SO(5)/\SO(3)\times \SO(2)$.

Suppose the three summands are non-equivalent, then we have the following possibilities for the Lie
algebra $\liek$ of the singular isotropy subgroups:
\begin{center}
\begin{tabular}{rl}
\textbf{A.1}: & $\liek = \lieh_1 \oplus \lieh_0\oplus \lieg_2$ and $\G_2/\gH_2$ is a sphere; \\
\textbf{A.2}: & $\liek = \lieg_1\oplus \lieh_0\oplus \lieh_2$ and $\G_1/\gH_1$ is a sphere; \\
\textbf{A.3}: & $\liek = \lieh_1 \oplus \lieh_0\oplus \lieh_0 \oplus \lieh_2$ and $\lieh_0 =
\lieu(1)$ or $\liesu(2)$; \\
\textbf{A.4}: & $\liek = \lieh\oplus \liep_1$ and $[\liep_1, \liep_1] \subset \lieh_1$; \\
\textbf{A.5}: & $\liek = \lieh\oplus \liep_2$ and $[\liep_2, \liep_2] \subset \lieh_2$.
\end{tabular}
\end{center}

Case \textbf{A.4} and \textbf{A.5} are excluded by the non-reducibility assumption. We consider
Case \textbf{A.4} first.

Suppose the Lie algebra of one singular isotropy subgroup is given in Case \textbf{A.4}. Then we
have $[\liep_1, \liep_1]$ is a proper subspace of $\lieh_1\oplus \lieh_0$ which implies that the
strong isotropy pair $(\lieg_1, \lieh_1\oplus \lieh_0)$ is not a symmetric pair and it is in Wolf's
classification\cite{WolfIrr}. If $\lieh_1$ is a zero vector space, then $[\liep_1, \liep_1] = 0$
and then we have
\begin{equation*}
Q_1([X_0, Y_1], Y_2) = Q_1(X_0, [Y_1, Y_2]) = 0, \mbox{ for any } X_0 \in \lieh_0, Y_1, Y_2 \in
\liep_1,
\end{equation*}
i.e., $\ad_{\lieh_0}$ is the trivial representation and it implies $\lieg_1 = \lieu(1)$ and
$\lieh_0 = 0$ which contradicts the assumption that $\lieh_0$ is not the zero vector space. If
$\lieh_1 \ne 0$, then $\lieg' = \liep_1 \oplus \lieh_1$ is a Lie subalgebra of $\lieg_1$. Let
$\G_1$, $\gL$ and $\G'$ be the connected Lie groups whose Lie algebras are $\lieg_1$,
$\lieh_0\oplus \lieh_1$ and $\lieg'$ respectively, then $\G'$ acts transitively on the homogeneous
space $\G_1/\gL$ with $\G_1$ a simple Lie group, i.e., $(\G_1, \gL, \G')$ is in Table
\ref{tableOnishchik}. Since the pair $(\lieg_1, \lieh_0\oplus\lieh_1)$ is strongly isotropy
irreducible, it is $(\lieso(4n),\liesp(1)\oplus \liesp(n))$ where $n\geq 2$. It also follows that
$\lieh_1 \oplus \liep_1 = \lieso(4n-1)$. Since $\lieg_1 = \lieh_0 \oplus \lieh_1 \oplus \liep_1$,
we have $\dim \lieh_0 = 4n-1$ that is not equal to $\dim \liesp(1)$ or $\dim \liesp(n)$ when $n\geq
2$, i.e., $\lieh_0$ is not either $\liesp(1)$ or $\liesp(n)$ which gives a contradiction.

In Case \textbf{A.5}, since the pair $(\lieg_2, \lieh_0\oplus \lieh_2)$ is strongly isotropy
irreducible and the diagram is non-reducible, $\lieg_2$ is a simple Lie algebra. Using a similar
argument as in Case \textbf{A.4}, this case is also excluded.

In Case \textbf{A.3}, the pair $(\lieg, \liek)$ is not strongly isotropy irreducible. The isotropy
representation $\Ad_{\lieg/\liek}$ has two irreducible summands $\chi_1 = \zeta_1\ot \Id$ and
$\chi_2 = \Id \ot \zeta_2$ and the isotropy representation of the pair $(\liek,\lieh)$ is $\chi_3 =
\Id \ot\ad_{\lieh_0}\ot \Id$. Since $\lieg_1$ is simple, $\chi_1$ is not equivalent to $\chi_3$. If
$\chi_2$ is equivalent to $\chi_3$, then $\lieg_2$ contains $\lieh_0$ factor and the diagram is
reducible. From Remark \ref{remnonequivsummands}, since $\chi_1$ is not equivalent to $\chi_2$, the
manifold is a sphere via sum action.

In Case \textbf{A.1} the pair $(\lieg_2,\lieh_0\oplus \lieh_2)$ is strongly isotropy irreducible,
with $\lieh_0$ a primitive Lie algebra and $(\lieg_2, \lieh_2)$ a spherical pair. Similar
properties hold for the pair $(\lieg_1,\lieh_0\oplus \lieh_1)$ in Case \textbf{A.2}. From the
classification of strongly isotropy irreducible spaces, it follows that $(\lieg_2,\lieh_0\oplus
\lieh_2)$ is either $(\liesu(n+1), \lieu(1)\oplus\liesu(n))$ or
$(\liesp(n+1),\liesp(1)\oplus\liesp(n))$ with $n\geq 1$.

\smallskip

If one triple is in Case \textbf{A.3}, then the manifold is a sphere via sum action. We assume that
both triples are strongly isotropy irreducible and in Case \textbf{A.1} or \textbf{A.2}. Thus
$\lieh_0$ is either $\lieu(1)$ or $\liesp(1)$ and then the assumption that the diagram is
non-reducible implies that $\lieg_1$ and $\lieg_2$ are simple Lie algebras. W.L.O.G., we may assume
that $(\lieg_1, \lieh_1) = (\liesu(p+1), \liesu(p))$ and then $\lieh_0 = \lieu(1)$. If $\lieh_2 =
0$, then $\lieg_2 = \liesu(2)$ and the triple of Lie algebras is
\begin{equation*}
\liesu(p)\oplus\lieu(1) \subset \liesu(p+1)\oplus \lieu(1) \subset \liesu(p+1)\oplus\liesu(2).
\end{equation*}
So the principal isotropy representation is
\begin{equation*}
\chi = (\Id \ot [\phi]_\Real)\oplus ([\fw_1]_\Real\ot [\phi]_\Real) \oplus (\Id \ot \Id).
\end{equation*}
Let $\liekm = \liesu(p+1)\oplus \lieu(1)$ and then $\liekp = \liesu(p)\oplus
\lieu(1)\oplus\liesu(2)$ otherwise $\liekp = \liekm$ and the manifold is a double. Thus we have
\begin{eg}
The diagram is
\begin{equation*}
\SU(p)\times \Diag \U(1) \subset \set{\SU(p+1)\times\U(1), \U(p)\times\SU(2)} \subset
\SU(p+1)\times \SU(2)
\end{equation*}
and the manifold is the complex projective space $\Cp^{p+2}$, see for example, \cite{Uchida}.
\end{eg}
When $p=1$, the diagram already appeared in \cite{HoelscherClass} as example $Q^6_D$.

\smallskip

If $\lieh_2 \ne 0$, then $(\lieg_2,\lieu(1)\oplus\lieh_2)$ is a strongly isotropy irreducible pair
and the three summands of the principal isotropy representation are pairwisely non-equivalent. If
the manifold is not a double, then $\liekm = \liesu(p)\oplus \lieu(1)\oplus \lieg_2$, $\liekp =
\liesu(p+1)\oplus \lieu(1)\oplus\lieh_2$ and $\lieh = \liesu(p)\oplus \Diag \lieu(1)\oplus\lieh_2$.
Furthermore we have that $(\lieg_2, \lieh_2)$ is a spherical pair and then is $(\liesu(q+1),
\liesu(q))$ for $q\geq 2$. Therefore we have
\begin{eg}
The diagram is
\begin{equation*}
\SU(p) \Diag \U(1) \SU(q) \subset \set{\U(p)\SU(q+1), \SU(p+1)\U(q)} \subset \SU(p+1)\SU(q+1)
\end{equation*}
and the manifold is the complex projective space $\Cp^{p+q+1}$.
\end{eg}

Next we assume that $(\lieg_1, \lieh_1) = (\liesp(p+1),\liesp(p))$ and then $\lieh_0 = \liesp(1)$.
A similar argument shows that if the manifold is not a double, then we have
\begin{eg}
The diagram is
\begin{equation*}
\Sp(p)\Diag \Sp(1) \Sp(q) \subset \set{\Sp(p+1)\Sp(q), \Sp(p)\Sp(q+1)} \subset \Sp(p+1)\Sp(q+1)
\end{equation*}
for $p,q\geq 1$, and the manifold is the quaternionic projective space $\Hp^{p+q+1}$, see
\cite{Iwataquater}.
\end{eg}
\end{proof}

\smallskip

\textsc{Case B. } $\lieh_0$ is a trivial Lie algebra, i.e., each primitive factor of $\lieh$ lies
in some primitive factor of $\lieg$ as a proper subspace and thus $\lieg$ has two or three factors.

\begin{prop}\label{propclasscaseb}
Suppose $\gH \subset \set{\Kpm} \subset \G$ is a primitive and non-reducible group diagram. If the
triples of Lie algebras $\set{\lieh \subset \liekpm \subset \lieg}$ are in \textsc{Case B}, then
the cohomogeneity one manifold is a sphere with a sum action.
\end{prop}

\begin{proof} We claim that $\lieg$ has exactly two primitive factors. In fact, from Proposition 1.20 in
\cite{HoelscherClass}, there are at most two $\lieu(1)$ factors since the manifold is
simply-connected. If there are exactly two $\lieu(1)$ factors, i.e., $\G = \G_0\times \mathsf{T}^2$
with $\G_0$ semisimple, then both $\lpm = 1$, $\Kpm = \gH\cdot \gS^1_\pm$ and the projections of
$\gS^1_\pm$ to $\T^2$ generate $\T^2$. Since $s=3$, it follows that $\G_0/\gH$ is strongly isotropy
irreducible and both $\gS^1_\pm$ are subgroup in $\T^2$.  So the projections of $\gS^1_\pm$ to the
$\G_0$ factor are zero and there exists an intermediate subgroup $\gH\times \T^2$ which implies
that the diagram is not primitive. Next we may assume that $\lieg$ has three factors and at most
one of them is $\lieu(1)$. It follows that the three principal isotropy summands are pairwisely
non-equivalent and then the pair $(\lieg,\liek)$ is strongly isotropy irreducible. Let
$\lieg=\lieg_1 \oplus \lieg_2 \oplus \lieg_3$ and $\lieh = \lieh_1 \oplus \lieh_2 \oplus \lieh_3$,
then each $(\lieg_i, \lieh_i)$ is strongly isotropy irreducible or $\lieg_i = \lieu(1)$ and
$\lieh_i = 0$. W.L.O.G., we may assume that $\liek = \lieg_1 \oplus \lieg_2 \oplus\lieh_3$ and then
from the classification of transitive actions on the sphere, there is no Lie group pair $(\K, \gH)$
with $\K/\gH$ a sphere such that the Lie algebras are given by $(\liek, \lieh)$.

Suppose $\lieh = \lieh_1 \oplus \lieh_2$ and $\lieg = \lieg_1 \oplus \lieg_2$ where $\lieg_1$,
$\lieg_2$ are primitive Lie algebras and $\lieh_i \subset \lieg_i$ for $i=1, 2$. W.L.O.G., we may
assume that $\lieg_1 = \lieh_1 \oplus \liep_1$, $\lieg_2 = \lieh_2 \oplus \liep_2 \oplus \liep_3$
and $\liep_i$'s are nonzero vector spaces. For the intermediate Lie algebra $\liek$, note that
$(\liek, \lieh)$ is a spherical pair that implies $\liek$ cannot be $\lieh\oplus \liep_1 \oplus
\liep_2$ or $\lieh\oplus \liep_1 \oplus \liep_3$, so we have the following four possibilities:
\begin{center}
\begin{tabular}{rl}
\textbf{B.1}: & $\liek = \lieg_1 \oplus \lieh_2$ and $(\lieg_1, \lieh_1)$ is a spherical pair; \\
\textbf{B.2}: & $\liek = \lieh \oplus \liep_3$, $\liel=\lieh_2 \oplus \liep_3$ is a Lie algebra
and $(\liel,\lieh_2)$ is a spherical pair; \\
\textbf{B.3}: & $\liek = \lieh \oplus \liep_2$, $\liel=\lieh_2 \oplus \liep_2$ is a Lie algebra and
$(\liel, \lieh_2)$ is a spherical pair; \\
\textbf{B.4}: & $\liek = \lieh_1 \oplus \lieg_2$ and $(\lieg_2, \lieh_2)$ is a spherical pair.
\end{tabular}
\end{center}

We see that only in Case \textbf{B.4}, $(\lieg, \liek)$ is a strongly isotropy irreducible pair. In
this case, since $\lieg_2$ is primitive, $(\lieg_2, \lieh_2) = (\liesu(k+1), \liesu(k))$ with
$k\geq 2$. Hence the triple of Lie algebras is
\begin{equation*}
\lieh_1\oplus \liesu(k) \subset \lieh_1\oplus \liesu(k+1) \subset \lieg_1\oplus \liesu(k+1), \quad
k\geq 2
\end{equation*}
and $(\lieg_1,\lieh_1)$ is strongly isotropy irreducible.

\smallskip

In Case \textbf{B.1} the pair $(\lieg, \liek)$ is not strongly isotropy irreducible. If $\dim
\liep_1$ is bigger than one, then the isotropy representation of $(\liek, \lieh)$ is not equivalent
to any summands in the isotropy representation of $(\lieg,\liek)$ which implies that the
cohomogeneity one manifold is a sphere. So we assume that $\dim \liep_1 = 1$, i.e., $\lieg_1 =
\lieu(1)$ and $\lieh_1 = 0$. Then one of $\liep_2$ and $\liep_3$, say $\liep_2$, is one dimensional
otherwise the three summands are pairwisely non-equivalent. Let $\liekm = \liek = \lieu(1)\oplus
\lieh_2$. If the manifold has one singular orbit whose codimension is bigger than $2$, then $\liekp
= \liep_0 \oplus \liep_3\oplus \lieh$ where $\liep_0$ is a 1-dimensional subspace in $\liep_1
\oplus \liep_2$. In particular, the pair $(\lieg, \liekp)$ is strongly isotropy irreducible and
thus belongs to Case \textbf{B.4}. It follows that $\lieg_2 = \liesu(k+1)$ and $\lieh_2 =
\liesu(k)$ with $k\geq 2$. Since the normalizer of $\liesu(k)$ in $\lieu(1)\oplus \liesu(k+1)$ is
$\lieu(1)\oplus\liesu(k)$, $\liekp$ has to be $\liesu(k+1)$ and then the diagram of the Lie
algebras is
\begin{equation*}
\liesu(k) \subset \set{\lieu(1)\oplus \liesu(k), \liesu(k+1)} \subset \lieu(1)\oplus\liesu(k+1),
\quad k\geq 2.
\end{equation*}
If $\gH$ is connected, then the group diagram is
\begin{equation*}
\SU(k) \subset \set{\U(1)\times \SU(k), \SU(k+1)} \subset \U(1)\times \SU(k+1),
\end{equation*}
and if $\gH$ is not connected, then for each $m \geq 2$, we have the following group diagram
\begin{equation*}
\Zeit_m \cdot \SU(k) \subset \set{\U(1)\times \SU(k), \Zeit_m\cdot \SU(k+1)} \subset \U(1) \times
\SU(k+1).
\end{equation*}
The \cohomd for these diagrams is the sphere $\sph^{2k+3}$ via a sum action. When $k=2$, see
example $Q^7_G$ in \cite{HoelscherClass}.

In Case \textbf{B.2}, $(\lieg, \liek)$ is not strongly isotropy irreducible and the principal
isotropy representation is
\begin{equation*}
\chi =(\ad_{\lieg_1/\lieh_1} \ot \Id) \oplus (\Id \ot \ad_{\lieg_2/\lieh_2}|_{\liep_2}) \oplus (\Id
\ot \ad_{\lieg_2/\lieh_2}|_{\liep_3})
\end{equation*}
and $\liep_3$ is the representation space of $\chi_3 = \ad_{\liek/\lieh}$. If $\chi_1$ and $\chi_2$
are not equivalent to $\chi_3$, then the manifold is a sphere. We consider the case when $\chi_2$
is equivalent to $\chi_3$, and thus the isotropy representation of $(\lieg_2, \lieh_2)$ has two
equivalent summands. From Dickinson-Kerr's classification and the proof of Proposition
\ref{propGK2summandsnonequiv}, $(\lieg_2,\lieh_2)$ is either $(\lieso(8),\liegt)$ or $(\lieso(7),
\lieu(3))$. Let $\liekm = \liek$ and if $\ad_{\liekp/\lieh}$ is irreducible, then $\liekp$ is
either $\lieh\oplus \liep_1$ or $\lieh \oplus \liep_0$ where $\liep_0$ is a subspace in
$\liep_2\oplus \liep_3$ with dimension $\dim \liep_2$. In the first case, $\ad_{\liekp/\lieh}$ is
not equivalent to $\ad_{\lieg/\lieh}|_{\liep_2}$ and $\ad_{\lieg/\liekp}$ has two summands. So the
manifold is a sphere. In the second case, the diagram is not primitive since both $\liekpm$ are
subalgebras of $\lieh_1\oplus\lieg_2$. Therefore the isotropy representation of $(\lieg,\liekp)$
has two summands, i.e., it belongs to Case \textbf{B.4}. It follows that $\lieg_1 = \liesu(k+1)$,
$\lieh_1 = \liesu(k)$ and $(\lieg_2, \lieh_2)$ is strongly isotropy irreducible which contradicts
the fact that it is one of $(\lieso(8), \liegt)$ and $(\lieso(7),\lieu(3))$. Next we consider the
case when $\chi_1$ is equivalent to $\chi_3$, and thus $(\lieg_1, \lieh_1) = (\lieu(1), 0)$ and
$\dim \liep_3 = 1$. Let $\liekm = \liek = \lieh_2\oplus \liep_3$. Since $ \lieu(1)\oplus \liekm$ is
a Lie subalgebra of $\lieu(1)\oplus \lieg_2$ and the diagram is primitive, $\liekp$ contains the
subspace $\liep_2$ which implies that $\liekp$ has codimension one in $\lieg$ and thus $(\lieg,
\liekp)$ is strongly isotropy irreducible, i.e., it is in Case \textbf{B.4}. So we have $(\lieg_2,
\lieh_2) = (\liesu(k+1), \liesu(k))$ for $k\geq 2$ and $\liekp=\liesu(k+1) = \lieg_2$. The diagram
is not primitive since both $\liekpm$ are contained in $\lieg_2$.

Similar argument shows that there is no new \cohomd other than a sphere in Case \textbf{B.3}.

Finally we consider Case \textbf{B.4}. From the previous discussion, we may assume that both
triples $\lieh \subset \liekpm \subset \lieg$ are in this case. It is easy to see that both
$(\lieg_i, \lieh_i)$ are spherical pair and the action is a sum action on a sphere.
\end{proof}

\medskip

\appendix

\section{Compact homogeneous spaces with two isotropy summands}\label{apptwosummands}

W. Dickinson and M. Kerr classified compact simply-connected homogeneous spaces $\G/\K$ in
\cite{DickKerr} for which $\G$ is a simple Lie group and the isotropy representation has two
summands. Their classification is not complete and also contains some mistakes.

For the pair $(\G, \K) = (\E_6, \Spin(6)\Spin(4)\SO(2))$ listed as \textbf{IV}.10 in their paper,
the isotropy representation should be
\begin{equation*}
\chi = \fw_2 \otimes [\fw\otimes \Id]_\Real \otimes \Id + [\fw_1\otimes \fw \otimes\Id\otimes
\phi]_\Real + [\fw_3\otimes \Id \otimes \fw \otimes \phi ]_\Real
\end{equation*}
and has three irreducible summands. For the pair $(\E_8, \Spin(12)\Spin(4))$ listed as
\textbf{IV}.37, the isotropy representation should be
\begin{equation*}
\chi = \fw_1\otimes [\fw\otimes \Id]_\Real + \fw_6\otimes \fw \otimes \Id + \fw_6\otimes \Id
\otimes \fw
\end{equation*}
and has three irreducible summands.

Except for these two examples, they also missed 5 pairs for which the isotropy representation has
two summands. The complete classification is

\begin{thm}\label{thmtwosummands}
Suppose $\G/\K$ is a compact simply-connected homogeneous space with $\G$ a simple Lie group. If
the isotropy representation of $\K$ has exactly two summands, then
\begin{enumerate}
\item either $(\G, \K)$ is listed in the paper \cite{DickKerr} except the examples IV.10 and IV.37;
\item or $(\G, \K)$ is one of the followings
\[
\begin{array}{rlr}
\mathbf{I}.31 & \Spin(10)\U(1)< \U(16) < \SO(32) & \rho = [\fw_4\otimes \phi]_\Real \\
              & d_1 = 210, d_2 = 240 & \chi = \fw_4\fw_5\otimes \Id \oplus [\fw_3\otimes
              \phi]_\Real \\
              & & \\
\mathbf{I}.32 & \SO(m)\times \Spin(7) < \SO(m) \times \SO(8) < \SO(m+8)(m \geq 1) & \rho =
\fw_1\otimes \Id \oplus \Id \otimes \fw_3 \\
              & d_1 = 7, d_2 = 8m & \chi = \Id \otimes \fw_1 \oplus \fw_1 \otimes \fw_3 \\
              & & \\
\mathbf{II}.15 & \SU(6) < \Sp(10) < \SU(20) & \rho = \fw_3 \\
               & d_1 = 175, d_2 = 189 & \chi = \fw_3^2 \oplus \fw_2 \fw_4 \\
               & & \\
\mathbf{III}.12 & \Sp(m)\times \U(1) < \U(2m) < \Sp(2m)(m\geq 2) & \rho = [\fw_1\otimes\phi]_\Real \\
                & d_1 = (m-1)(2m+1), d_2 = 2m(2m+1) & \chi = \fw_2\otimes \Id \oplus [\fw_1^2\otimes
                \phi]_\Real \\
                & & \\
\mathbf{V}.19 & \SU(2)\times \SU(2) < \SO(8) & \rho = \fw\otimes \fw^3 \\
               & d_1 = 7, d_2 = 15 & \chi = \Id\otimes \fw^6 \oplus \fw^2 \otimes\fw^4.
\end{array}
\]
\end{enumerate}
\end{thm}

\smallskip

\begin{proof}
Since $\G$ and $\K$ are connected Lie groups, we consider their Lie algebras $\lieg$ and $\liek$,
and thus the isotropy representation $\ad_{\lieg/\liek}$ has two summands. We separate the proof
into two parts: $\lieg$ is classical or exceptional Lie algebra. In the first part, we consider the
case that $\lieg$ is a classical Lie algebra, i.e., it is one of $\lieso(n)$, $\liesu(n)$ and
$\liesp(n)$.

\textsc{Case I.} If $\lieg$ is $\lieso(n)$ and $\rho: \liek \To \lieg$ is the embedding, then the
isotropy representation $\chi$ of $\lieg/\liek$ is determined by $\Lambda^2 \rho = \ad_{\liek}
\oplus \chi$ where $\Lambda^2 \rho$ is the exterior square of the representation $\rho$.

\textsc{Case I.a.} If $\rho = \rho_1 \oplus \rho_1^*$ and $\rho_1$ is an irreducible representation
of complex type, i.e., $\rho_1(\liek) \subset \lieu(m)$ with $n = 2m$, then $\Lambda^2\rho =
[\Lambda^2\rho_1]_\Real \oplus [\rho_1\otimes \rho_1^*]$ where $[\Lambda^2\rho_1]_\Real =
\Lambda^2\rho_1 \oplus \Lambda^2\rho_1^*$ and $\rho_1\otimes \rho_1^*$ contains the representation
$\ad_{\liek}$. If $\liek$ is $\lieu(m)$, then $(\lieg, \liek)$ is a symmetric pair and thus
$\ad_{\lieg/\liek}$ is isotropy irreducible. If $\liek = \liesu(m)$, then it is example
\textbf{I}.24. If $\liek$ is not $\lieu(m)$ or $\liesu(m)$, then the irreducible representation
with highest weight $\lambda + \lambda^*$ is contained in $\rho_1\otimes \rho_1^*$ but not in
$\ad_{\liek}$, see \cite{WangZillerEinstein}, where $\lambda$ is the highest weight of the
representation $\rho_1$. It follows that $\Lambda^2 \rho_1$ is irreducible, i.e., $\Lambda^2 \mu_m$
remains irreducible when restricted from $\lieu(m)$ to $\liek$. Such embeddings $\rho_1 : \liek \To
\lieu(m)$ are classified by D.B. Dykin in \cite{Dynkin}. Furthermore the isotropy representation
$\ad_{\lieso(2m)/\liek}$ for each pair has two irreducible summands. The pairs $(\lieg, \liek)$
give us example \textbf{I}.25 -- \textbf{I}.28 in \cite{DickKerr} and the example \textbf{III}.31.

%

If $\rho = \rho_1 \oplus \rho_1$ and $\rho_1$ is of quaternionic type, then $\Lambda^2 \rho =
\Lambda^2 \rho_1 \oplus \Lambda^2 \rho_1 \oplus [\rho_1\otimes \rho_1]$ and $\Lambda^2 \rho_1$ is
of real type and $\rho_1\otimes \rho_1 = \Sym^2\rho_1 \oplus \Lambda^2 \rho_1$. So the isotropy
representation $\ad_{\lieg/\liek}$ has more than two irreducible summands.

\smallskip

\textsc{Case I.b.} Next we assume that $\rho = \rho_1 \oplus \rho_1^* \oplus \rho_2$ where $\rho_1$
is irreducible of complex or quaternionic type and $\rho_2$ is a nontrivial representation with $l
= \dim_{\Real} \rho_2 \geq 2$. If $\rho_1 = \phi$, then $\liek = \lieu(1)\oplus \liek_0$, $\rho=
[\phi]_\Real \otimes\Id \oplus \Id \otimes \rho_2$ and $\rho(\liek) \subset \lieso(2)\oplus
\lieso(l) \subset \lieso(2+l)$. It follows that $\Lambda^2\rho = \Id \oplus [\Id\otimes
\Lambda^2\rho_2]\oplus ([\phi]_\Real \otimes \rho_2)$ and $\ad_{\liek} = \Id \oplus \ad_{\liek_0}$
is contained in $\Id \oplus [\Id \otimes \Lambda^2\rho_2]$. Since $\liek_0$ is proper in
$\lieso(l)$ and $\ad_{\lieg/\liek}$ has two irreducible summands, we have that $\rho_2$ is
irreducible and the pair $(\lieso(l), \liek_0)$ is isotropy irreducible. They are example
\textbf{I}.1 -- \textbf{I}.18(when $m=2$) in \cite{DickKerr} and  the special case of the example
\textbf{I}.32 when $m=2$.

If $\rho_2$ is the trivial representation, then $\liek = \lieu(m)$ and $\lieg=\lieso(2m+1)$ with $m
\geq 2$. It is example \textbf{I}.18(when $m=1$) in \cite{DickKerr}.

\smallskip

\textsc{Case I.c.} We assume that every irreducible summand in $\rho$ is of real type. If $\rho$ is
not irreducible, then $\rho = \rho_1 \oplus \rho_2$ with $\dim_\Real \rho_i = n_i$ for $i=1, 2$.
Let $\liek_2 = \ker \rho_1$ and $\liek_1 =\ker \rho_2$. Then $\liek = \liek_1 \oplus \liek_2$,
$\Lambda^2\rho = \Lambda^2\rho_1 \oplus \Lambda^2\rho_2 \oplus [\rho_1 \otimes \rho_2]$ and
$\ad_{\liek} = \ad_{\liek_1} \oplus \ad_{\liek_2}$. Since $\ad_{\liek_i} \subseteq \Lambda^2\rho_i$
and the equality implies that $\liek_i = \lieso(n_i)$, one of $\liek_i$, say $\liek_2$, is equal to
$\lieso(n_2)$ and then $(\lieso(n_1), \liek_1)$ is isotropy irreducible and $\rho_1$ is
irreducible. These give us example \textbf{I.}1 -- \textbf{I}.17, \textbf{I}.19 and \textbf{I}.30
in \cite{DickKerr} and the example \textbf{I}.32.

If $\rho$ is irreducible, then $\liek$ has at most two simple factors. We use the classification of
Kraemer in \cite{Kraemer}. These give us example \textbf{I}.20 -- \textbf{I}.23, \textbf{I}.29,
\textbf{V}.1 -- \textbf{V}.5 in \cite{DickKerr} and the example \textbf{V}.19. This finishes the
proof when $\G$ is an orthogonal group.

\smallskip

\textsc{Case II.} If $\lieg$ is $\liesu(n)$ and $\rho: \liek \To \lieg$ is the embedding, then the
isotropy representation $\chi$ of $\lieg/\liek$ is determined by $\rho \otimes \rho^*= \Id \oplus
\ad_{\liek} \oplus \chi$.

\textsc{Case II.a.} The image is contained in $\liesp(m)$($n=2m$) or $\lieso(n)$. We consider the
first case where $\rho(\liek)$ is contained in $\lieso(n)$. Let $\chi_2$ be the isotropy
representation of $\ad_{\liesu(n)/\lieso(n)}$, then it remains irreducible when restricted to the
proper subgroup $\liek$. Dynkin classified the triples $(\liek, \lieso(N), \chi)$ where $\chi \ne
\varrho_N$ such that the restriction of $\chi$ to $\liek$ remains irreducible. They are example
\textbf{II}.13 and \textbf{II}.14 in \cite{DickKerr}.

Next we assume that $\rho(\liek) \subset \liesp(m) \subset \liesu(2m)$. It follows that
$\ad_{\liesu(2m)/\liesp(m)}$ remains irreducible when restricted to the proper subalgebra $\liek$.
Such $\liek$ are classified by Dynkin. They are example \textbf{II}.9 -- \textbf{II}.12 in
\cite{DickKerr} and the example \textbf{II}.15.

\smallskip

\textsc{Case II.b.} Suppose that $\rho=\rho_1 \oplus \rho_2$ is reducible with $m=\deg \rho_1$ and
$n = \deg \rho_2$. Let $\liek_2 = \ker\rho_1$ and $\liek_1 = \ker \rho_2$, and then $\liek =
\liek_1 \oplus \liek_2$, $\rho= \sigma_1 \otimes \Id \oplus \Id \otimes \sigma_2$ where
$\sigma_i$'s are representations of $\liek_i$. It follows that
\begin{equation*}
\rho \otimes \rho^* = (\sigma_1\otimes \sigma_1^*)\otimes \Id + \Id\otimes(\sigma_2\otimes
\sigma_2^*) + (\sigma_1\otimes \sigma_2^*) + (\sigma_1^*\otimes \sigma_2).
\end{equation*}

Both of $\liek_1$ and $\liek_2$ cannot be proper subalgebras of $\liesu(m)$ and $\liesu(n)$
respectively. The case when $\liek_1 = \liesu(m)$ with $\sigma_1 = \mu_m$ and $\liek_2 = \liesu(n)$
with $\sigma_2 =\mu_n$ is example \textbf{II}.7 in \cite{DickKerr}. If one of $\liek_1$ and
$\liek_2$, say $\liek_1$, is a proper subalgebra of $\liesu(m)$, then $\liek_2$ has to be
$\lieu(n)$ with $\sigma_2 = \mu_n$ and $\sigma_1$ is an irreducible complex representation. Let
$\chi_1$ denote the isotropy representation of $\ad_{\liesu(m)/\liek_1}$ and then the summands of
$\ad_{\lieg/\liek}$ are $\chi_1 \otimes \Id$ and $[\sigma_1\otimes \mu_n]_\Real$. Such $\liek$'s
can be classified using Kraemer's results and they give us example \textbf{II}.1 -- \textbf{II}.6
and \textbf{II}.8 in \cite{DickKerr}.

\smallskip

\textsc{Case II.c.} Suppose that $\rho$ is an irreducible complex representation and $\rho(\liek)$
is not contained in some $\liesp(m)$ or $\lieso(n)$. Then we can assume that $\liek$ has at most
two simple factors and we can use Kraemer's classification. They are example \textbf{V}.6 --
\textbf{V}.8 in \cite{DickKerr}. This finishes the proof when $\G$ is a unitary group $\SU(n)$.

\smallskip

\textsc{Case III.} If $\lieg$ is $\liesp(n)$ with $n\geq 3$ and $\rho: \liek \To \lieg$ is the
embedding, then the isotropy representation $\chi$ of $\lieg/\liek$ is determined by $\Sym^2 \rho =
\ad_{\liek} \oplus \chi$.

\textsc{Case III.a.} Suppose the image of $\rho(\liek)$ is contained in $\lieu(n)$. If $\liek$ is
semi-simple, i.e., $\rho(\liek) \subset \liesu(n)$, then since $\ad_{\liesp(n)/\liesu(n)}$ already
has two summands, we have $\liek = \liesu(n)$. It is  example \textbf{III}.8 in \cite{DickKerr}. If
$\liek$ contains $\lieu(1)$ factors and its semi-simple part idea is denoted by $\liek_0$, then
$\liesu(n)/\liek_0$ is isotropy irreducible and the restriction the representation $\fw_1^2$ from
$\liesu(n)$ to $\liek_0$ remains irreducible. From Dynkin's classification, $\liek_0$ is
$\liesp(n/2)$ and the embedding is $\fw_1$. It gives us the example \textbf{III}.12.

\smallskip

\textsc{Case III.b.} If $\rho = \rho_1 \oplus \rho_2$ and $\rho_1$ is irreducible of real or
quaternionic type, then $\rho(\liek) \subset \liesp(n_1)\oplus \liesp(n_2)$. We may assume that
$\liek = \liesp(n_1) \oplus \liek_0$ and $\rho = \nu_{n_1}\otimes \Id \oplus \Id \otimes \sigma$.
The isotropy representation of $\liesp(n_2)/\liek_0$ is irreducible and it is denoted by $\chi_1$.
Since $\ad_{\lieg/\liek}$ has two irreducible summands, $\sigma$ is irreducible. They are example
\textbf{III}.1 -- \textbf{III}.7 in \cite{DickKerr}. In fact Kraemer missed the pair $(\liegt\oplus
\liesp(1), \liesp(7))$ in his classification.

\smallskip

\textsc{Case III.c.} If $\rho$ is irreducible of real or quaternionic type, then we may assume that
$\liek$ has at most two primitive factors. We can use Kraemer's classification and they are example
\textbf{III}.9 -- \textbf{III}.11 and \textbf{V}.9 -- \textbf{V}.15 in \cite{DickKerr}. This
finishes the proof when $\G$ is a symplectic group.

\smallskip

In the second part, we consider the case when $\lieg$ is an exceptional Lie algebra. There are two
different cases. First we assume that there is an intermediate subalgebra $\liel$ between $\liek$
and $\lieg$. So both pair $(\lieg, \liel)$ and $(\liel, \liek)$ are isotropy irreducible. From the
classification strong isotropy irreducible homogeneous spaces in \cite{WolfIrr}, we can determine
the possible $\liel$ for each $\lieh$. Then we look at the possible $\liek$ such that $\liel/\liek$
is isotropy irreducible and $\ad_{\lieg/\liel}$ remains irreducible when restricted to $\liek$.
Such examples of $(\lieg,\liek)$ give us example \textbf{IV}.1 -- \textbf{VI}.48 in \cite{DickKerr}
except \textbf{IV}.10 and \textbf{IV}.37. Next we assume that $\liek$ is maximal in $\lieg$ and
then the isotropy representation of $\lieg$ splits as $\ad_{\liek}$ and $\ad_{\lieg/\liek}$ when it
is restricted to $\liek$. We check the table of the branching rules in \cite{McKayPatera} to see
for which $(\lieg,\liek)$, $\ad_{\lieg/\liek}$ has exactly two summands. They are example
\textbf{V}.16 -- \textbf{V}.18 in \cite{DickKerr} and thus the classification is finished.
\end{proof}


\medskip

\section{Collection of tables}\label{appTables}

In this appendix, we collect the tables which contain the classification of the case $s=3$.

\subsection{Primitive actions}
From the classifications in Theorem \ref{thmclassificationreducible}, \ref{thms3primitiveGsimple}
and \ref{thms3Gnonsimple}, except for a few sum actions on spheres, the primitive actions without
fixed points and with $s=3$ are non-reducible and the manifolds are spheres, projective spaces and
Grassmannian manifolds $\Gr_m(\Real^{m+n}) = \SO(m+n)/(\SO(m)\times \SO(n))(m,n \geq 2)$. In the
following, we first describe the actions on spheres, including the reducible ones, and then we list
the actions on the other manifolds.

\subsubsection{Actions on spheres}
The cohomogeneity one actions on spheres were classified in \cite{Straume}, see the group diagrams
in \cite{GWZ}. A large class of cohomogeneity one actions on spheres with $s=3$ is given by the sum
actions. Recall that if $\gL_i/\gH_i = \sph^{l_i}$($i=1,2$), then $\gL_1 \times \gL_2$ acts on the
$\sph^{l_1+l_2+1}$ via cohomogeneity one with diagram
\begin{equation*}
\gH_1 \times \gH_2 \subset \set{\gL_1 \times \gH_2, \, \gH_1 \times \gL_2} \subset \gL_1 \times
\gL_2,
\end{equation*}
and $s= s_1 + s_2$ where $s_i$ is the number of the irreducible summands in the isotropy
representation of $\gL_i/\gH_i$. It is easy to see that any variation of the diagram is equivalent
to the original one. If one singular orbit is codimension 2, say $\gL_1/\gH_1 = \sph^1$, then for
every $m\ne 0$, we have the diagram
\begin{equation*}
\Zeit_m \times \gH_2 \subset \set{\U'(1) \times \gH_2, \Zeit_m \times \gL_2} \subset \U(1) \times
\gL_2,
\end{equation*}
where the $\U'(1)$ factor may diagonally embedded into $\U(1) \times \gL_2$. The action is not
effective if $m\geq 2$ and the non-effective kernel is $\Zeit_m\times 1$.

Every sum action is primitive. The action is reducible if one of the spheres is given as
$\U(n+1)/\U(n)$, $\SO(4)/\SO(3)$, $\Sp(n+1)\U(1)/\Sp(n)\Delta\U(1)$ or
$\Sp(n+1)\Sp(1)/\Sp(n)\Delta\Sp(1)$, where $n \geq 1$. Using Table
\ref{Tabletransitiveactionsphere} of transitive actions on spheres, one can easily write down the
diagram of sum actions with $s\leq 3$.

Other than sum actions, there are a few cohomogeneity one actions on sphere which have $s=3$. All
of them are primitive actions and they are listed in Table \ref{Tablespherenotsum} where $\pi$ is
the representation of $\G$ on $\Real^n$. Note that the actions of $\SU(2)\times \SU(2)$ on $\sph^6$
and $\Spin(7)$ on $\sph^{14}$ are special cases of what are called generalized sum actions in
\cite{GWZ}.

\begin{table}[!h]
\begin{center}
\begin{tabular}{|c|c|c|c|c|c|}
\hline
$n$ & $\G$ & $\pi$ & $\Km$ & $\Kp$ & $\gH$ \\
\hline \hline
$6$ & $\SU(2)\times \SU(2)$ & $\fw_1 \otimes \fw_1 \oplus \Id \otimes \fw_1^2$ & $\Delta \SU(2)$ & $\SU(2) \times \U(1)$ & $\Delta \U(1)$ \\
\hline
$7$ & $\SU(3)$ & $\fw_1\fw_2$ & $S(\U(1)\times \U(2))$ & $S(\U(2)\times \U(1))$ & $\U(1)\times \U(1)$ \\
\hline
$14$ & $\Spin(7)$ & $\varrho_7\oplus \Delta_7$ & $\Spin(6)$ & $\Gt$ & $\SU(3)$ \\
\hline
$25$ & $\F_4$ & $\fw_1$ & $\Spin(9)_1$ & $\Spin(9)_2$ & $\Spin(8)$ \\
\hline
\end{tabular}
\end{center}
\smallskip
\caption{Cohomogeneity one actions on $\sph^n$ with $s=3$ which are not sum
actions.}\label{Tablespherenotsum}
\end{table}

\subsubsection{Actions on projective spaces and Grassmannian manifolds}
The cohomogeneity one actions on projective spaces were classified in \cite{Uchida},
\cite{Iwataquater} and \cite{IwataCayley}. Note that all cohomogeneity one actions on $\Cp^n$ and
$\Hp^n$ are obtained from an action on an odd dimensional sphere when $\U(1)$ and $\Sp(1)$ is a
normal subgroup in $\G$ with induced action given by a Hopf action, see \cite{GWZ}. Table
\ref{Tableprojectivespaces} list these actions as well as those on the Grassmannian manifolds for
which $s=3$.

\begin{table}[!h]
\begin{center}
\begin{tabular}{|c|c|c|c|c|}
\hline
 & $\G$ & $\Km$ & $\Kp$ & $\gH$ \\
\hline \hline
$\Cp^3$ & $\SU(2) \times \SU(2)$ & $\U(1) \times \U(1)$ & $\Zeit_2 \cdot \Delta \SU(2)$ & $\Zeit_2 \cdot \Delta\U(1)$ \\
\hline
$\Cp^7$ & $\Spin(7)$ & $\SO(2)\cdot \SU(3)$ & $\Zeit_2\cdot \Gt$ & $\Zeit_2 \cdot \SU(3)$ \\
\hline
$\Cp^7$ & $\SO(8)$ & $\U(4)$ & $\Spin(7)$ & $\SU(4)$ \\
\hline
$\Cp^{n+1}$ & $\SO(2+n)$ & $\SO(2) \times \SO(n)$ & $\gO(n+1)$ & $\Zeit_2 \cdot \SO(n)$ \\
\hline
$\Cp^{p+q+1}$ & $\SU(p+1)\times \SU(q+1)$ & $\SU(p+1)\U(q)$ & $\U(p)\SU(q+1)$ & $\SU(p) \Delta \U(1)\SU(q)$ \\
\hline
$\Hp^{p+q+1}$ & $\Sp(p+1)\times \Sp(q+1)$ & $\Sp(p+1)\Sp(q)$ & $\Sp(p)\Sp(q+1)$ & $\Sp(p) \Delta \Sp(1)\Sp(q)$ \\
\hline \hline
$\Gr_2(\Real^5)$ & $\SU(2)\times \SU(2)$ & $\U(1) \times \U(1)$ & $\Delta \SU(2)$ & $\Delta \U(1)$ \\
\hline
$\Gr_3(\Real^7)$ & $\Gt$ & $\SO(4)$ & $\SO(4)'$ & $\SO(3)$ \\
\hline
$\Gr_2(\Real^9)$ & $\Spin(7)$ & $\SO(2)\cdot \SU(3)$ & $\Gt$ & $\SU(3)$ \\
\hline
$\Gr_p(\Real^{p+q+1})$ & $\SO(p+q+1)$ & $\SO(p)\times \SO(q+1)$ & $\SO(p+1)\times \SO(q)$ & $\SO(p)\times \SO(q)$ \\
\hline
\end{tabular}
\end{center}
\smallskip
\caption{Fixed-point free cohomogeneity one actions on projective spaces and Grassmannian manifolds
with $s=3$.}\label{Tableprojectivespaces}
\end{table}

\subsection{Reducible, non-primitive actions}

The reducible, non-primitive cohomogeneity one manifolds with $s=3$ and whose reduced diagram has
$s\geq 4$ are classified in Table \ref{Tables3reducibledouble} and Table
\ref{Tables3reduciblenondouble}, see Theorem \ref{thmclassificationreducible}.

Table \ref{Tables3reducibledouble} is the classification where the manifold is a double, i.e., $\Km
= \Kp$, and Table \ref{Tables3reduciblenondouble} is the one where the manifold is not a double. In
both tables, the last column contains the conditions for the groups. If a homogeneous space appears
in this column, it means that the space is strongly isotropy irreducible, for examples, in
\textbf{R}.1 the space $\G_1/(\gH_1 \gH_2 \U(1))$ is strongly isotropy irreducible, in \textbf{R}.3
the space $\G_1/\gH_1$ is strongly isotropy irreducible and not the circle, and in \textbf{R}.7 the
space $\G_1/\gH_1$ is an isotropy irreducible sphere. In both tables, $k\geq 2$ and $m \geq 1$ are
positive integers. $\Zeit_m$ is a cyclic group in $\Km/\gH$ if it is a circle, and $\Zeit_1$ stands
for the trivial group with one element. In Table \ref{Tables3reduciblenondouble}, for the singular
isotropy subgroups, we write $\Km$ on the top of $\Kp$. There are further conditions for some of
them.
\begin{itemize}
\item In example \textbf{R}.9 and \textbf{R}.10, the groups satisfy the Condition T1: the homogenous
space $\G_1/(\gH_1\times \gH_0)$ has two irreducible isotropy summands, $\G_1$ is simple and
$\gH_0$ is $\U(1)$(example \textbf{R}.9) or $\SU(2)$(example \textbf{R}.10).
\item In example \textbf{R}.11, the $\U'(1)$ factor of $\K$ is embedded into $\G$ as $\set{(e^{\imath \theta}, e^{\imath
p\theta}, 1)|\, 0 \leq \theta \leq 2\pi}$ for some integer $p$. The diagram is the normal extension
of example $Q^5_A$ in \cite{HoelscherClass}.
\item In example \textbf{R}.12, the embedding of the $\U'(1)$ factor of $\K$ into $\G$ is given as
\begin{equation*}
\set{\left(e^{\imath \theta}, \begin{pmatrix}\beta(p \theta) & \\ & A \end{pmatrix}, 1\right)| \, 0
\leq \theta \leq 2 \pi, A \in \SO(n)} \subset \U(1) \times \SO(2+n) \times \U(1),
\end{equation*}
where $\beta(p \theta)$ is the rotation by angle $p \theta$ and $p$ is an integer.
\item Example \textbf{R}.22 is the normal extension of example $N^5$ in \cite{HoelscherClass}, see
the embeddings of $\U(1)_{-}$ and $\U(1)_{+}$ there.
\item Example \textbf{R}.23 was discussed in the proof of Theorem \ref{thmclassificationreducible}, or
see Lemma 4.3 in \cite{HoelscherClass}. We consider the corresponding non-reducible diagram as $\G
= \U(1) \times \SO(n+2)$. The reducible one can be obtained by normal extension since the subgroup
$\SO(2) \subset \SO(n+2)$ normalizes all isotropy subgroups. The embeddings of $\Kpm_c$ are given
as
\begin{equation*}
\Kpm_c = \set{\left(e^{\imath n_{\pm} \theta},
\begin{pmatrix}\beta(m_{\pm}\theta) & \\ & A \end{pmatrix}\right) | \, 0 \leq \theta \leq 2\pi, A \in \SO(n)}
\end{equation*}
such that $\Km_c \ne \Kp_c$ where $m_{\pm}$ and $n_{\pm}$ are integers. Take two cyclic groups
$\Zeit_{k_{\pm}} \subset \set{(e^{\imath n_{\pm} \theta}, \beta(m_{\pm}\theta))}$ and let
$\gH_{\pm} = \Zeit_{k_{\pm}}\cdot \SO(n)$. Let $\gH = <\gH_-,\gH_+>$ generated by $\gH_{\pm}$ and
let $\Kpm = \Kpm_c \cdot \gH$, then the diagram defines a simply connected manifold if and only if
$\gcd(n_-, n_+, d) =1$ where $d$ is the index of $\gH\cap \Km_c \cap \Kp_c$ inside $\Km_c \cap
\Kp_c$.
\end{itemize}

\medskip

\begin{table}[!h]
\begin{center}
\begin{tabular}{|r|l|l|l|l|}
\hline
 & $\quad \quad \quad \quad \G$ & $\quad \quad \quad \quad \Kpm$ & $\quad \quad \quad \gH$ &  \\
\hline \hline
R.1 & $\G_1 \times \U(1) \times \gH_2$ & $\gH_1 \U(1) \Delta \gH_2 \U(1)$ & $\Zeit_m \cdot \gH_1 \Delta\U(1)\Delta \gH_2$ & $\G_1/(\gH_1\gH_2\U(1))$ \\
    & & & & $\gH_2$ primitive \\
\hline
R.2 & $\G_1 \times \SU(2) \times \gH_2$ & $\gH_1 \SU(2) \Delta \gH_2 \SU(2)$ & $\gH_1 \Delta\SU(2)\Delta \gH_2$ & $\G_1/(\gH_1\gH_2\SU(2))$ \\
    & & & & $\gH_2$ primitive \\
\hline
R.3 & $\G_1 \times \SU(2)\times \U(1)$ & $\gH_1\SU(2)\U(1)$ & $\gH_1 \Delta\U(1)$ & $\G_1/\gH_1 \ne \sph^1$ \\
\hline
R.4 & $\G_1 \times \Sp(n+1)\Sp(1)$ & $\gH_1 \Sp(n+1)\Sp(1)$ & $\gH_1 \Sp(n)\Delta \Sp(1)$ & $\G_1/\gH_1$ \\
\hline
R.5 & $\G_1 \times \G_2\times \U(1)$ & $\gH_1 \gH_2 \U(1)\U(1)$ & $\Zeit_m\cdot \gH_1 \gH_2\Delta \U(1)$ & $\G_1/\gH_1 \ne \sph^1$ \\
    &  & & & $\G_2/(\gH_2\U(1))$ \\
\hline
R.6 & $\G_1 \times \G_2\times \SU(2)$ & $\gH_1 \gH_2 \SU(2)\SU(2)$ & $\gH_1 \gH_2\Delta \SU(2)$ & $\G_1/\gH_1$ \\
    &  & & & $\G_2/(\gH_2\SU(2))$ \\
\hline
    & & & & $\G_1/\gH_1 = \sph^k$ \\
R.7 & $\G_1\times \G_2 \times \gH_0$ & $\G_1\gH_2\Delta \gH_0$ & $\gH_1 \gH_2\Delta \gH_0$ & $\G_2/(\gH_2 \gH_0)$ \\
    & & & & $\gH_0$ primitive \\
\hline
R.8 & $\U(1)\times \G_1 \times \gH_0$ & $\U(1)\gH_1\Delta \gH_0$ & $\gH_1\Delta \gH_0$ & $\G_1/(\gH_1 \gH_0)$ \\
    & & & & $\gH_0$ simple \\
\hline
R.9 & $\G_1\times \U(1)$ & $\gH_1 \U(1) \U(1)$ & $\Zeit_m \cdot \gH_1\Delta\U(1)$ & Condition T1 \\
\hline
R.10 & $\G_1\times \SU(2)$ & $\gH_1 \SU(2)\SU(2)$ & $\gH_1\Delta\SU(2)$ & Condition T1 \\
\hline
R.11 & $\U(1)\times \SU(2)\times \U(1)$ & $\U'(1)\Delta \U(1)$ & $\Zeit_m \cdot \Delta \U(1)$ &  \\
\hline
R.12 & $\U(1) \times \SO(n+2) \times \U(1)$ & $\U'(1) \SO(n)\Delta \U(1)$ & $\Zeit_m \cdot \SO(n)\Delta \U(1)$ & $n\geq 2$\\
\hline
R.13 & \multicolumn{3}{c|}{$\G \supset \K=\Kpm \supset \Zeit_m \cdot \gH$ in Table \ref{tabletriplereducible}} & $m=1$ if $\K/\gH\ne \sph^1$ \\
\hline
\end{tabular}
\end{center}
\smallskip
\caption{Non-primitive fixed-point free cohomogeneity one manifolds with reducible actions. The
action has $s=3$ and its reduced action has $s\geq 4$. Part I: the manifold is a double}
\label{Tables3reducibledouble}
\end{table}

\begin{table}[!h]
\begin{center}
\begin{tabular}{|r|l|l|l|l|}
\hline
 & $\quad \quad \quad \quad \G$ & $\quad \quad \quad \quad \Kpm$ & $\quad \quad \quad \gH$ &  \\
\hline \hline
R.14 & $\G_1 \times \U(1) \times \SU(2)$ & $\gH_1 \U(1)\Delta \SU(2) \U(1)$ & $\Zeit_m \cdot\gH_1 \Delta \U(1)\Delta \SU(2)$ & $\G_1/(\gH_1\U(1)\SU(2))$ \\
    & & $\Zeit_m \cdot\gH_1 \SU(2)\Delta \U(1)\SU(2)$ & &  \\
\hline
R.15 & $\G_1 \times \SU(2)\times\U(1)$ & $\gH_1 \U(1)\U(1)$ & $\gH_1 \Delta\U(1)$ & $\G_1/\gH_1 \ne \sph^1$  \\
     & & $\gH_1\SU(2)\U(1)$ & & \\
\hline
R.16 & $\G_1 \times \Sp(n+1)$ & $\gH_1\Sp(n+1)\Sp(1)$ & $\gH_1 \Sp(n)\Delta \Sp(1)$ & $\G_1/\gH_1$ \\
     & $\times \Sp(1)$ & $\gH_1 \Sp(n)\Sp(1)\Sp(1)$ & & \\
\hline
R.17 & $\G_1 \times \G_2 \times \U(1)$ & $\gH_1\gH_2\U(1)\U(1)$ & $\Zeit_m \cdot\gH_1 \gH_2\Delta \U(1)$ & $\G_1/\gH_1 = \sph^k$  \\
     & & $\Zeit_m \cdot \G_1 \gH_2\Delta\U(1)$ & & $\G_2/(\gH_2\U(1))$  \\
\hline
R.18 & $\G_1 \times \G_2 \times \SU(2)$ & $\gH_1\gH_2\SU(2)\SU(2)$ & $\gH_1 \gH_2\Delta \SU(2)$ & $\G_1/\gH_1 = \sph^k$  \\
     & & $\G_1 \gH_2\Delta\SU(2)$ & & $\G_2/(\gH_2\SU(2))$ \\
\hline
R.19 & $\G_1 \times \U(1)$ & $\gH_1\U(1) \U(1)$ & $\gH_1\Delta \U(1)$ & $\G \supset \Kp \supset \gH_c$\\
     & & $\gH_1 \Sp(1)\U(1)$ & & is 1--5 in Table \ref{tabletriplereducible}\\
\hline
R.20 & $\G_1 \times \U(1)$ & $\gH_1\U(1) \U(1)$ & $\Zeit_m \cdot\gH_1\Delta \U(1)$ & $\G \supset \Kp_c \supset \gH_c$\\
     & & $\Zeit_m \cdot\K_1 \Delta \U(1)$ & & is 6--8 in Table \ref{tabletriplereducible}\\
\hline
R.21 & $\G_1 \times \SU(2)$ & $\gH_1\SU(2)\SU(2)$ & $\gH_1\Delta \SU(2)$ & $\G \supset \Kp \supset \gH$\\
     & & $\K_1 \Delta \SU(2)$ & & is 9--13 in Table \ref{tabletriplereducible}\\
\hline
R.22 & $\U(1) \times \SU(2)$ & $\U(1)_{-}\cdot \gH$ & $\gH$ & \\
     & $ \times \U(1)$       & $\U(1)_{+}\cdot \gH$ & & \\
\hline
R.23 & $\U(1) \times \SO(n+2)$ & $\Km_c\cdot \gH$ & $\gH$ & $n\geq 2$ \\
     & $\times \U(1)$          & $\Kp_c\cdot \gH$ & & \\
\hline
\end{tabular}
\end{center}
\smallskip
\caption{Non-primitive fixed-point free cohomogeneity one manifolds with reducible actions. The
action has $s=3$ and its reduced action has $s\geq 4$. Part II: the manifold is not a double}
\label{Tables3reduciblenondouble}
\end{table}

Since the actions are non-primitive, the manifolds are bundles over lower dimensional bases. To
identity the bundle structure, it is easy to work with non-reducible diagrams. Table
\ref{Tables3reduciblenonreducible1} and \ref{Tables3reduciblenonreducible2} list the diagrams of
the reduced actions in Table \ref{Tables3reducibledouble} and \ref{Tables3reduciblenondouble}. Note
that in example \textbf{R}.19, \textbf{R}.20 and \textbf{R}.21 we keep the reducible diagrams.

\begin{table}[!h]
\begin{center}
\begin{tabular}{|r|l|l|l|l|}
\hline
 & $\quad \quad \quad \G$ & $\quad \quad \Kpm$ & $\quad \quad \gH$ &  \\
\hline \hline
R.1 & $\G_1$ & $\gH_1 \U(1)$ & $\Zeit_m \cdot \gH_1$ & $\G_1/(\gH_1\gH_2\U(1))$ \\
    & & & & $\gH_2$ primitive \\
\hline
R.2 & $\G_1$ & $\gH_1 \SU(2)$ & $\gH_1$ & $\G_1/(\gH_1\gH_2\SU(2))$ \\
    & & & & $\gH_2$ primitive \\
\hline
R.3 & $\G_1 \times \SU(2)$ & $\gH_1\SU(2)$ & $\gH_1$ & $\G_1/\gH_1 \ne \sph^1$ \\
\hline
R.4 & $\G_1 \times \Sp(n+1)$ & $\gH_1 \Sp(n+1)$ & $\gH_1 \Sp(n)$ & $\G_1/\gH_1$ \\
\hline
R.5 & $\G_1 \times \G_2$ & $\gH_1 \gH_2 \U(1)$ & $\Zeit_m\cdot \gH_1 \gH_2$ & $\G_1/\gH_1 \ne \sph^1$ \\
    &  & & & $\G_2/(\gH_2\U(1))$ \\
\hline
R.6 & $\G_1 \times \G_2$ & $\gH_1 \gH_2 \SU(2)$ & $\gH_1 \gH_2$ & $\G_1/\gH_1$ \\
    &  & & & $\G_2/(\gH_2\SU(2))$ \\
\hline
    & & & & $\G_1/\gH_1 = \sph^k$ \\
R.7 & $\G_1\times \G_2$ & $\G_1\gH_2$ & $\gH_1 \gH_2$ & $\G_2/(\gH_2 \gH_0)$ \\
    & & & & $\gH_0$ primitive \\
\hline
R.8 & $\U(1)\times \G_1$ & $\U(1)\gH_1$ & $\gH_1$ & $\G_1/(\gH_1 \gH_0)$ \\
    & & & & $\gH_0$ simple \\
\hline
R.9 & $\G_1$ & $\gH_1 \U(1)$ & $\Zeit_m \cdot \gH_1$ & Condition T1 \\
\hline
R.10 & $\G_1$ & $\gH_1 \SU(2)$ & $\gH_1$ & Condition T1 \\
\hline
R.11 & $\U(1)\times \SU(2)$ & $\U'(1)$ & $\Zeit_m$ &  \\
\hline
R.12 & $\U(1) \times \SO(n+2)$ & $\U'(1) \SO(n)$ & $\Zeit_m \cdot \SO(n)$ & $n\geq 2$\\
\hline
R.13 & \multicolumn{3}{c|}{$\G \supset \K=\Kpm \supset \Zeit_m \cdot \gH$ in Table \ref{tabletriplereducible}} & $m=1$ if $\K/\gH\ne \sph^1$ \\
\hline
\end{tabular}
\end{center}
\smallskip
\caption{The diagrams of reduced actions in Table \ref{Tables3reducibledouble}.}
\label{Tables3reduciblenonreducible1}
\end{table}

\begin{table}[!h]
\begin{center}
\begin{tabular}{|r|l|l|l|l|}
\hline
 & $\quad \quad \quad \G$ & $\quad \quad \Kpm$ & $\quad \quad \gH$ &  \\
\hline \hline
R.14 & $\G_1$ & $\gH_1 \U(1)$ & $\Zeit_m \cdot\gH_1$ & $\G_1/(\gH_1\U(1)\SU(2))$ \\
    & & $\Zeit_m \cdot\gH_1 \SU(2)$ & &  \\
\hline
R.15 & $\G_1 \times \SU(2)$ & $\gH_1 \U(1)$ & $\gH_1$ & $\G_1/\gH_1 \ne \sph^1$  \\
     & & $\gH_1\SU(2)$ & & \\
\hline
R.16 & $\G_1 \times \Sp(n+1)$ & $\gH_1\Sp(n+1)$ & $\gH_1 \Sp(n)$ & $\G_1/\gH_1$ \\
     &  & $\gH_1 \Sp(n)\Sp(1)$ & & \\
\hline
R.17 & $\G_1 \times \G_2$ & $\gH_1\gH_2\U(1)$ & $\Zeit_m \cdot\gH_1 \gH_2$ & $\G_1/\gH_1 = \sph^k$  \\
     & & $\Zeit_m \cdot \G_1 \gH_2$ & & $\G_2/(\gH_2\U(1))$  \\
\hline
R.18 & $\G_1 \times \G_2$ & $\gH_1\gH_2\SU(2)$ & $\gH_1 \gH_2$ & $\G_1/\gH_1 = \sph^k$  \\
     & & $\G_1 \gH_2$ & & $\G_2/(\gH_2\SU(2))$ \\
\hline
R.19 & $\G_1 \times \U(1)$ & $\gH_1\U(1) \U(1)$ & $\gH_1\Delta \U(1)$ & $\G \supset \Kp \supset \gH_c$\\
     & & $\gH_1 \Sp(1)\U(1)$ & & is 1--5 in Table \ref{tabletriplereducible}\\
\hline
R.20 & $\G_1 \times \U(1)$ & $\gH_1\U(1) \U(1)$ & $\Zeit_m \cdot\gH_1\Delta \U(1)$ & $\G \supset \Kp_c \supset \gH_c$\\
     & & $\Zeit_m \cdot\K_1 \Delta \U(1)$ & & is 6--8 in Table \ref{tabletriplereducible}\\
\hline
R.21 & $\G_1 \times \SU(2)$ & $\gH_1\SU(2)\SU(2)$ & $\gH_1\Delta \SU(2)$ & $\G \supset \Kp \supset \gH$\\
     & & $\K_1 \Delta \SU(2)$ & & is 9--13 in Table \ref{tabletriplereducible}\\
\hline
R.22 & $\U(1) \times \SU(2)$ & $\U(1)_{-}\cdot \gH$ & $\gH$ & \\
     &        & $\U(1)_{+}\cdot \gH$ & & \\
\hline
R.23 & $\U(1) \times \SO(n+2)$ & $\Km_c\cdot \gH$ & $\gH$ & $n\geq 2$ \\
     &         & $\Kp_c\cdot \gH$ & & \\
\hline
\end{tabular}
\end{center}
\caption{The diagrams of reduced actions in Table \ref{Tables3reduciblenondouble}.}
\label{Tables3reduciblenonreducible2}
\end{table}

\smallskip

In certain cases, the diagram defines a product action, i.e., the manifold is a product of a
cohomogeneity one manifold and a homogeneous space. In Table \ref{Tablebundlereducible}, we specify
each manifold in Table \ref{Tables3reducibledouble} and \ref{Tables3reduciblenondouble} as a
product(if the action is a product action) or a (possibly non-trivial) bundle. The first seven
examples \textbf{R}.3, \textbf{R}.4, \textbf{R}.7, \textbf{R}.8, \textbf{R}.11, \textbf{R}.15 and
\textbf{R}.16 are products. Note that the family in example \textbf{R}.22 contains both trivial and
non-trivial $\sph^3$ bundle over $\sph^2$ depending on the embeddings of the circles $\U(1)_-$ and
$\U(1)_+$, see \cite{HoelscherClass}.

\begin{table}[!h]
\begin{center}
\begin{tabular}{|r|c|c|}
\hline
 & Fiber & Base \\
\hline \hline
R.3 & $\sph^4$ & $\G_1/ \gH_1$ \\
\hline
R.4 & $\sph^{4n+4}$ & $\G_1/\gH_1$ \\
\hline
R.7 & $\sph^{k+1}$ & $\G_2/\gH_2$ \\
\hline
R.8 & $\sph^{2}$ & $\G_1/\gH_1$ \\
\hline
R.11 & $\sph^2$ & $\sph^3$ \\
\hline
R.15 & $\Cp^2$ & $\G_1/\gH_1$ \\
\hline
R.16 & $\Hp^{n+1}$ & $\G_1/\gH_1$ \\
\hline \hline
R.1 & $\sph^2$ & $\G_1/(\gH_1\cdot \U(1))$ \\
\hline
R.2 & $\sph^4$ & $\G_1/(\gH_1\cdot \SU(2))$ \\
\hline
R.5 & $\sph^2$ & $\G_1/\gH_1 \times \G_2/(\gH_2 \cdot \U(1))$ \\
\hline
R.6 & $\sph^4$ & $\G_1/\gH_1 \times \G_2/(\gH_2 \cdot \SU(2))$ \\
\hline
R.9 & $\sph^2$ & $\G_1/(\gH_1 \cdot \U(1))$ \\
\hline
R.10 & $\sph^4$ & $\G_1/(\gH_1 \cdot \SU(2))$ \\
\hline
R.12 & $\sph^2$ & $\SO(n+2)/\SO(n)$ \\
\hline
R.13 & sphere & $\G/\K$ \\
\hline
R.14 & $\sph^5$ & $\G_1/(\gH_1 \cdot \U(1) \cdot \SU(2))$ \\
\hline
R.17 & $\sph^{k+2}$ & $\G_2/(\gH_2 \cdot \U(1))$ \\
\hline
R.18 & $\sph^{k+4}$ & $\G_2/(\gH_2 \cdot \SU(2))$ \\
\hline
R.19 & $\Cp^2$ & $\G_1/(\gH_1\cdot \SU(2))$ \\
\hline
R.20 & sphere & $\G_1/(\K_1 \cdot \U(1))$ \\
\hline
R.21 & sphere & $\G_1/(\K_1 \cdot \SU(2))$ \\
\hline
R.22 & $\sph^3$ & $\sph^2$ \\
\hline
R.23 & lens space & $\SO(n+2)/(\SO(2)\times \SO(n))$ \\
\hline
\end{tabular}
\end{center}
\smallskip
\caption{The bundle structure of the manifolds in Table \ref{Tables3reducibledouble} and
\ref{Tables3reduciblenondouble}. The first seven examples are
products.}\label{Tablebundlereducible}
\end{table}

\newpage

\subsection{Non-primitive, non-reducible actions}

The non-primitive, non-reducible cohomogeneity one manifolds with $s=3$ which are not doubles are
classified in Table \ref{Tables3nonprimitive}, see Theorem \ref{thms3nonprimitive}.

In Table \ref{Tables3nonprimitive} we use the same conventions in the tables of reducible,
non-primitive actions, i.e., $k\geq 2$ and the homogeneous space appeared in the last column is
strongly isotropy irreducible. The further conditions for some diagrams are
\begin{itemize}
\item In example \textbf{N}.5, the groups satisfy Condition T2: $\G_1/\gL_1$ is strongly isotropy irreducible
and the diagram $\gH_1 \subset \set{\K_1, \gL_2} \subset \gL_2$ has $s=2$.
\item In example \textbf{N}.7, we have $l_1, l_2 \geq 1$.
\item In example \textbf{N}.8, we have $l\geq 1$ and $k \geq 2$.
\item In example \textbf{N}.8, \textbf{N}.9 and \textbf{N}.10, the triple is not the last two
examples in Table \ref{TableHKG2summandsKHsphere}. Otherwise the diagram is reducible.
\item Example \textbf{N}.10 was already discussed in the proof of Theorem \ref{thms3nonprimitive}.
Both $\Kpm_c$ are contained in the subgroup $\gL = \U(1)\times \U(1)\cdot \gH'_c$. The construction
of cohomogeneity one diagrams is very similar to the example \textbf{R}.23 and example $N^7_H$ in
\cite{HoelscherClass}.
\end{itemize}

\begin{table}[!h]
\begin{center}
\begin{tabular}{|r|l|l|l|l|}
\hline
 & $\quad \quad \quad \quad \G$ & $\quad \quad \quad \quad \Kpm$ & $\quad \quad \quad \gH$ &  \\
\hline \hline
N.1 & $\E_6$ & $\SU(2) \U(5)$ & $\SU(2) \SU(5)$ & \\
    & & $\SU(2) \SU(6)$ & & \\
\hline
N.2 & $\Spin(6+n)$ & $\SO(n) \U(3)$ & $\SO(n) \SU(3)$ & $n\geq 1$ \\
    &              & $\SO(n) \SU(4)$ & & \\
\hline
N.3 & $\G_1 \times \Sp(n+1)$ & $\gL_1  \Sp(1) \Sp(1) \Sp(n)$ & $\gL_1 \Delta\Sp(1)\Sp(n)$ & $\G_1/(\gL_1\Sp(1))$ \\
    &  & $\gL_1 \Sp(1)\Sp(n+1)$ & & \\
\hline
N.4 & $\G_1\times \SU(n+1)$ & $\gL_1 \U(1)  S(\U(1) \U(n))$ & $\gL_1 \Delta\U(1)\SU(n)$ & $\G_1/(\gL_1\U(1))$ \\
    & & $\gL_1  \U(1)  \SU(n+1)$ & & \\
\hline
N.5 & $\G_1\times \gL_2$ & $\gL_1  \K_1$ & $\gL_1  \gH_1$ & Condition T2 \\
    & & $\gL_1 \gL_2$ & & \\
\hline \hline
N.6 & $\Spin(8) \times \G_1$ & $\Spin^-(7)\times \gL_1$ & $\Gt \times \gL_1$ & $\G_1/\gL_1$ \\
    & & $\Spin^+(7) \times \gL_1$ & & \\
\hline
N.7 & $\G_0 \times \gL_1\times \gL_2$ & $\gL_0  \gL_1 \gH_2$ & $\gL_0 \gH_1 \gH_2$ & $\G_0/\gL_0$ \\
    & & $\gL_0 \gH_1  \gL_2$ & & $\gL_i/\gH_i = \sph^{l_i}$ \\
\hline
N.8 & $\gL_1 \times \G_2$ & $\gH_1  \K_2$ & $\gH_1 \gH_2$ & $\G_2 \supset \K_2 \supset \gH_2$ in Table \ref{TableHKG2summandsKHsphere} \\
    & & $\gL_1 \gH_2$ & & $\gL_1/\gH_1 = \sph^{l}, \, \K_2/\gH_2 = \sph^{k}$ \\
\hline
N.9 & $\gL_1 \times \G_2$ & $\gH_1  \K_2$ & $\gH_1 \gH_2 \cdot \Zeit_m$ & $\G_2 \supset \K_2 \supset \gH_2$ in Table \ref{TableHKG2summandsKHsphere} \\
    & & $\gL_1 \gH_2\cdot \Zeit_m $ & & $\gL_1/\gH_1 = \sph^k, \,\K_2/\gH_2 = \sph^1$ \\
\hline
N.10 & $\U(1) \times \G'$ & $\K'$ & $\gH'$ & $\G' \supset \K'_c \supset \gH'_c$ in Table \ref{TableHKG2summandsKHsphere} \\
    & & $\U(1) \gH'$ & & $\K'_c/\gH'_c = \sph^1$ \\
\hline
\end{tabular}
\end{center}
\smallskip
\caption{Non-primitive cohomogeneity one manifolds with non-reducible actions which are not
doubles. The action has no fixed points and $s=3$.}\label{Tables3nonprimitive}
\end{table}

In Table \ref{Tablebundlenonprimitive}, we also specify the bundle structure of the examples in
Table \ref{Tables3nonprimitive}. The fiber also admits a cohomogeneity one action with diagram $\gH
\subset \set{\Kpm} \subset \gL$ where $\gL$ is a proper subgroup in $\G$. We identify the fiber
either as a known manifold or by its diagram. The first three examples \textbf{N}.5, \textbf{N}.6
and \textbf{N}.7 are products.

\begin{table}[!h]
\begin{center}
\begin{tabular}{|r|c|c|}
\hline
& Fiber & Base \\
\hline \hline
N.5 & $\gH_1 \subset \set{\K_1, \gL_2} \subset \gL_2$ & $\G_1/\gL_1$ \\
\hline
N.6 & $\sph^{15}$ & $\G_1/\gL_1$ \\
\hline
N.7 & $\sph^{l_1+l_2+1}$ & $\G_0/\gL_0$ \\
\hline \hline
N.1 & $\Cp^6$ & $\E_6/(\SU(2)\cdot \SU(6))$ \\
\hline
N.2 & $\Cp^4$ & $\SO(6+n)/(\SO(6)\times \SO(n))$ \\
\hline
N.3 & $\Hp^{n+1}$ & $\G_1/(\gL_1\cdot \Sp(1))$ \\
\hline
N.4 & $\Cp^{n+1}$ & $\G_1/(\gL_1 \cdot \U(1))$ \\
\hline
N.8 & $\sph^{l + k +1}$ & $\G_2/\K_2$ \\
\hline
N.9 & $\sph^{k + 2}$ & $\G_2/\K_2$ \\
\hline
N.10 & lens space & $\G'/\K'$ \\
\hline
\end{tabular}
\end{center}
\smallskip
\caption{The bundle structure of the manifolds in Table \ref{Tables3nonprimitive}. The first three
examples are products.}\label{Tablebundlenonprimitive}
\end{table}

\newpage

\subsection{Classifications in low dimensions}

For the convenience of the reader, we list the non-reducible cohomogeneity one manifolds with
$s\leq 3$ in low dimensions. The $4$-manifold $\Cp^2\sharp \overline{\Cp^2}$ is the connected sum
of $\Cp^2$ and another copy of $\Cp^2$ with opposite orientation.  The examples of dimensions 5, 6
and 7 are already in Table 8.2 and 8.4 in \cite{HoelscherClass}. The last column shows the type of
the corresponding diagram with connected groups.

Example $N^6_F$ and $N^7_I$ appear in Theorem \ref{thmclasss2nonprimitive}, see Remark
\ref{rems2lowdim}. Example $N^6_D$ is a special case of example $\textbf{N}.4$ in Table
\ref{Tables3nonprimitive} with $\G_1/\gL_1 = \SU(2)/\U(1)$. Example $N^7_H$ is a special case of
example \textbf{N}.10 in Table \ref{Tables3nonprimitive} with the triple $(\SU(3) \supset \U(2)
\supset \SU(2))$. The primitive one $Q^7_E$ appears in the proof of Theorem
\ref{thmrigidityconnected}. The other primitive ones $Q^6_A$, $Q^6_C$, $Q^6_D$ and $Q^7_G$ appear
in the proof of Proposition \ref{propclasscasea} and \ref{propclasscaseb}.

\begin{table}[!h]
\begin{center}
\begin{tabular}{|c|c|c|c|}
\hline
 & Diagram & $s$ & Type \\
\hline \hline
$\sph^2 \times \sph^2$ & $\Zeit_n \subset \set{\Kpm = \U(1)} \subset \SU(2)$, $n$ even & $3$ & double \\
\hline
$\Cp^2 \sharp\overline{\Cp^2}$ & $\Zeit_n \subset \set{\Kpm = \U(1)} \subset \SU(2)$, $n$ odd & $3$ & double \\
\hline
$\Cp^2$ & $\langle \qi \rangle \subset \set{e^{\qi\theta}, e^{\qj\theta}\cup \qi e^{\qj\theta}} \subset \gS^3$ & $3$ & primitive \\
\hline
$\sph^4$ & $\set{\pm 1, \pm \qi, \pm \qj, \pm \qk} \subset \set{e^{\qi\theta} \cup \qj e^{\qi\theta}, e^{\qj\theta}\cup \qi e^{\qj\theta}} \subset \gS^3$ & $3$ & primitive \\
\hline \hline
$Q^6_A$ & $\Delta \U(1) \cdot \Zeit_n \subset \set{\U(1)\times \U(1), \Delta \SU(2)\cdot \Zeit_n} \subset \SU(2)\times \SU(2)$ & $3$ & primitive\\
& where $n=1$ or $2$ & &  \\
\hline
$Q^6_C$ & $\Delta \U(1) \subset \set{\Delta \SU(2), \SU(2)\times \U(1)} \subset \SU(2)\times \SU(2)$ & $3$ & primitive \\
\hline
$Q^6_D$ & $\Delta \U(1) \subset \set{\SU(2)\times \U(1), \U(1) \times \SU(2)} \subset \SU(2)\times \SU(2)$ & $3$ & primitive \\
\hline
$N^6_D$ & $\set{(e^{\qi p\theta}, e^{\qi\theta})} \subset \set{\U(1)\times \U(1), \SU(2)\times \U(1)} \subset \SU(2) \times \SU(2)$ & $3$ & non-primitive \\
\hline
$N^6_B$ & $\set{(e^{\qi p\theta}, e^{\qi q\theta})}\cdot \Zeit_n \subset \set{\Kpm=\U(1) \times \U(1)} \subset \SU(2) \times \SU(2)$ & $3$ & double \\
\hline
$Q^6_B$ & $\Delta \U(1) \subset \set{\Kpm = \Delta \SU(2)} \subset \SU(2) \times \SU(2)$ & $3$ & double \\
\hline
$N^6_E$ & $\set{(e^{\qi p\theta}, e^{\qi\theta})} \subset \set{\Kpm = \SU(2) \times \U(1)} \subset \SU(2) \times \SU(2)$ & $3$ & double \\
\hline
$N^6_F$ & $\SU(2)\cdot \Zeit_2 \subset \set{\Kpm = \U(2)} \subset \SU(3)$ & $2$ & double \\
\hline \hline
$Q^7_E$ & $\U(1) \times \U(1) \subset \set{S(\U(1)\U(2)), \, S(\U(2)\U(1))} \subset  \SU(3)$ & $3$ & primitive \\
\hline
$Q^7_G$ & $\gH_0\cdot \Zeit_n \subset \set{(\beta(m\theta),e^{\qi\theta})}\cdot \gH_0, \, \SU(3)\times \Zeit_n \subset \SU(3)\times \U(1)$ & $3$ & primitive \\
 & $\gH_0 = \SU(2)\times 1, \, \Zeit_n\subset \set{(\beta(m\theta),e^{\qi\theta})}$, & & \\
 & $\beta(\theta)= \diag(e^{-\qi\theta},e^{\qi\theta},1), \, \gcd(m,n)=1$ & & \\
\hline
$N^7_H$ & $\gH \subset \set{(\beta(m_-\theta),e^{\qi n_-\theta})}\cdot \gH, \, \set{(\beta(m_+\theta),e^{\qi n_+\theta})}\cdot \gH  \subset \SU(3)\times \U(1)$ & $3$ & non-primitive \\
 & $\gH_0 = \SU(2)\times 1$, $\gH=\gH_-\cdot \gH_+$, $\Km\ne \Kp$, & & \\
 & $\beta(\theta)= \diag(e^{-\qi\theta},e^{\qi\theta},1)$, $\gcd(n_-,n_+,d)=1$ & & \\
 & where $d$ is the index of $\gH\cap \Km_0 \cap \Kp_0$ in $\Km_0 \cap \Kp_0$ & & \\
\hline
$Q^7_F$ & $\gH_0\cdot \Zeit_n \subset  \Kpm = \set{(\beta(m\theta),e^{\qi\theta})}\cdot \gH_0 \subset \SU(3)\times \U(1)$ & $3$ & double \\
 & $\gH_0 = \SU(2)\times 1$, $\Zeit_n\subset \set{(\beta(m\theta),e^{\qi\theta})}$, & & \\
 & $\beta(\theta)= \diag(e^{-\qi\theta},e^{\qi\theta},1)$ & & \\
\hline
$N^7_G$ & $\U(1)\times \U(1) \subset \set{\Kpm = \U(2)} \subset \SU(3)$ & $3$ & double \\
\hline
$N^7_I$ & $\Sp(1)\U(1) \subset \set{\Kpm = \Sp(1)\Sp(1)} \subset \Sp(2)$ & $2$ & double \\
\hline
\end{tabular}
\end{center}
\smallskip
\caption{Cohomogeneity one manifolds with non-reducible action and $s\leq 3$ in low dimensions. The
action is not a product or a sum action and has no fixed points.} \label{Tablemanifoldlowdim}
\end{table}


\medskip

%


\vfill
\end{document}